\documentclass[a4paper, 12pt, notitlepage]{amsart}
\usepackage[colorlinks]{hyperref} 
\usepackage{latexsym, amssymb, amscd, amsfonts, amsthm, amsmath, graphicx,
mathabx, mathrsfs, tikz}

\newcommand{\tri}{\operatorname{tri}}
\newcommand{\hex}{\operatorname{hex}}

\newcommand{\fcc}{\operatorname{fcc}}
\newcommand{\Dfour}{\operatorname{D4}}
\newcommand{\full}{\operatorname{full}}

\newcommand{\supp}{\operatorname{supp}}

\newcommand{\diam}{\operatorname{diam}}
\newcommand{\mix}{\operatorname{mix}}
\newcommand{\TV}{\operatorname{TV}}

\newcommand{\nbd}{\operatorname{nbd}}
\newcommand{\sav}{\operatorname{sav}}
\newcommand{\RE}{\operatorname{Re}}
\newcommand{\IM}{\operatorname{Im}}
\newcommand{\spn}{\operatorname{span}}

\newcommand{\bN}{\mathbb{N}}
\newcommand{\bR}{\mathbb{R}}

\newcommand{\bT}{\mathbb{T}}
\newcommand{\bU}{\mathbb{U}}
\newcommand{\zed}{\mathbb{Z}}

\newcommand{\id}{\mathrm{id}}

\newcommand{\gap}{\mathrm{gap}}

\newcommand{\one}{\mathbf{1}}
\newcommand{\E}{\mathbf{E}}
\newcommand{\Var}{\mathbf{Var}}
\newcommand{\Cov}{\mathbf{Cov}}
\newcommand{\Prob}{\mathbf{Prob}}

\newcommand{\sB}{\mathscr{B}}
\newcommand{\sF}{\mathscr{F}}
\newcommand{\sG}{\mathscr{G}}
\newcommand{\sH}{\mathscr{H}}
\newcommand{\sL}{\mathscr{L}}
\newcommand{\sN}{\mathscr{N}}
\newcommand{\sR}{\mathscr{R}}
\newcommand{\sS}{\mathscr{S}}
\newcommand{\sT}{\mathscr{T}}

\newcommand{\sI}{\mathscr{I}}
\newcommand{\sX}{\mathscr{X}}

\newcommand{\ua}{\underline{a}}

\newcommand{\sP}{{\mathscr{P}}}

\newcommand{\sC}{{\mathscr{C}}}
\newcommand{\sQ}{{\mathscr{Q}}}

\newcommand{\fS}{{\mathfrak{S}}}

\newtheorem{theorem}{Theorem}

\newtheorem{cor}[theorem]{Corollary}
\newtheorem{lemma}[theorem]{Lemma}
\newtheorem*{lemma*}{Lemma}
\newtheorem{proposition}[theorem]{Proposition}

\theoremstyle{remark}

\newtheorem*{theorem*}{Theorem}

\title[Sandpiles]{Cut-off for sandpiles on tiling graphs}
\author{Robert Hough}
\address[Robert Hough]{Department of Mathematics, Stony Brook University, Stony Brook,
NY, 11794}
\email{robert.hough@stonybrook.edu}

\author{Hyojeong Son}
\address[Hyojeong Son]{Department of Mathematics, Stony Brook University, Stony Brook,
NY, 11794}
\email{hyojeong.son@stonybrook.edu}

\subjclass[2010]{Primary 82C20, 60B15, 60J10}
\keywords{Abelian sandpile model, random walk on a group, spectral gap,  cut-off
phenomenon}

\thanks{This material is based upon work supported by the National Science
Foundation under agreements No.\ DMS-1712682 and DMS-1802336. Any opinions, findings and
conclusions or recommendations expressed in this material are those of the
authors and do not necessarily reflect the views of the National Science
Foundation.}

\thanks{Hyojeong Son was supported by a fellowship from the Stony Brook Summer Research Fund.}

\begin{document}

\begin{abstract}
Sandpile dynamics are considered on graphs constructed from periodic plane and space tilings by assigning a growing piece of the tiling either torus or open boundary conditions.  A general method of obtaining the Green's function of the tiling is given, and a total variation cut-off phenomenon  is demonstrated under general conditions.  It is shown that the boundary condition does not affect the mixing time for planar tilings, nor does it change the asymptotic mixing time for the cubic lattice $\zed^d$ for all sufficiently large $d$.  In a companion paper, computational methods are used to demonstrate that the mixing time is altered for the $\Dfour$ lattice in dimension 4.  
\end{abstract}

\maketitle

\section{Introduction}
The abelian sandpile model is an important model of self organized criticality, which has been studied extensively in the statistical physics literature since it was introduced by Bak, Tang and Wiesenfeld \cite{BTW88}, see, e.g. \cite{D89}, \cite{DM92}, \cite{BIP93}, \cite{P94}, \cite{I94}, \cite{DFF03}, \cite{LH02},  \cite{PR05}, \cite{JPR06}, \cite{SV09}, \cite{LP10}, \cite{DS10}, \cite{FLW10}, \cite{PS13}, \cite{G16}, \cite{KW16}, \cite{NOT17}.  
Sandpile dynamics on a finite connected graph $G= (V, E)$ may be described as follows. In the model, a node $s \in V$ is designated sink.  Each non-sink vertex $v$ is assigned a non-negative number $\sigma(v)$ of chips.  If at some point $\sigma(v) \geq \deg(v)$ the vertex can topple, passing one chip to each neighbor; if a chip falls on the sink it is lost from the model. A configuration $\sigma$ is called \emph{stable} if $\sigma(v) < \deg(v)$ for all $v \in V\setminus\{s\}$.  The dynamics in the model occur in discrete time steps, in which a chip is added to the model at a uniform random vertex, then all legal topplings occur until the model reaches a stable state. 
 
 In \cite{HJL17} sandpile dynamics are studied on the torus $(\zed/m\zed)^2$ and the asymptotic total variation mixing time is determined with a cut-off phenomenon as $m \to \infty$.  This article extends the techniques of \cite{HJL17} to treat sandpiles on a growing  piece of an arbitrary periodic plane or space tiling of arbitrary dimension, again determining the asymptotic total variation mixing time and proving a cut-off phenomenon. 
 A second purpose of the article is to study the effect of the boundary condition on the mixing time, and a class of tilings are considered with an open boundary in which the chips fall off the boundary and are lost from the model.  In this case, also, a cut-off phenomenon is demonstrated in the total variation mixing time and in two dimensions it is shown that the asymptotic mixing time is the same for the periodic and open boundary conditions, resolving a problem raised in \cite{HJL17}. 
 
 In a companion paper \cite{HS19}  computations are performed of the spectral gap and  `boundary spectral parameters' associated to eigenfunctions which are harmonic modulo 1 and concentrated near boundaries of a specified dimension in several specific examples including the triangular
and honeycomb 
tilings in dimension 2 and the face centered cubic sphere packing in dimension 3.  By determining these parameters for a specific set of bounding hyperplanes of the $\Dfour$ lattice in dimension 4 it is demonstrated that the total variation mixing with an open boundary is controlled by a statistic concentrated near the 3 dimensional boundary, and is thus different from the periodic boundary mixing time asymptotically. It is also  proved that for all $d$ sufficiently large, the asymptotic mixing time on the cubic lattice $\zed^d$ is the same for periodic and open boundary conditions determined by hyperplanes parallel to the coordinate axes, but that the optimization problem controlling the spectral gap does not determine the asymptotic mixing time.

\subsection{Precise statement of results}\subsubsection{Convergence of probability measures}
The results presented consider convergence of  probability measures in the \emph{total variation metric}. This is already a strong notion of convergence, and in fact all of the results hold also in $L^2$.
Recall that the total variation distance between two probability measures $\mu$ and $\nu$ on a measure space $(\sX, \sB)$ is 
\begin{equation}
 \left\|\mu - \nu\right\|_{\TV(\sX)} = \sup_{A \in \sB} |\mu(A)-\nu(A)|.
\end{equation}
Given a finite graph $G$, the set of recurrent sandpiles on the graph form an abelian group \cite{D88}. A random walk driven by a probability measure $\mu$ on a group has distribution at step $n$ given by $\mu^{*n}$ where $\mu^{*1} = \mu$ and $\mu^{*n} = \mu * \mu^{*(n-1)}$ is the group convolution.  
 Given a measure $\mu$ driving sandpile dynamics on the group of recurrent sandpile states $\sG(G)$ with uniform measure $\bU_{\sG}$, the \emph{total variation mixing time} is
\begin{equation}
 t_{\mix} = \min \left\{ k: \left\|\mu^{*k} - \bU_{\sG(G)}\right\|_{\TV(\sG(G))} < \frac{1}{e}\right\}.
\end{equation}
Given a sequence of graphs $G_n$ the sandpile dynamics is said to satisfy the \emph{cut-off phenomenon in total variation} if, for each $\epsilon > 0$,
\begin{align*}
 \left\| \mu^{*\lceil (1-\epsilon) t_{\mix} \rceil}  - \bU_{\sG(G_n)}\right\|_{\TV(\sG(G_n))} &\to 1, \\\left\| \mu^{*\lfloor (1+\epsilon) t_{\mix} \rfloor}  - \bU_{\sG(G_n)}\right\|_{\TV(\sG(G_n))} &\to 0
\end{align*}
as $n \to \infty$. 
 
 \subsubsection{Periodic tiling graphs}
To describe the tilings and graphs more precisely, let $M$ be a non-singular $d \times d$ matrix, and let $\Lambda = M \cdot \zed^d < \bR^d$ be a $d$-dimensional lattice.  A (periodic) space tiling $\sT$ is a connected graph embedded in $\bR^d$ which is connected, is $\Lambda$-periodic, has finitely many vertices in a fundamental domain for $\bR^d/\Lambda$, and has bounded degree.  Suppose without loss of generality that 0 is a vertex in $\sT$. Given an integer $m \geq 1$, two types of graphs are considered.
\begin{enumerate}
 \item[(1)] (Torus boundary condition) The graph $\bT_m = \sT/m\Lambda$ consists of $m^d$ fundamental domains with opposite faces identified.  By convention, 0 is designated sink.
\end{enumerate}
\begin{figure}
 \centering
\begin{tikzpicture}[node distance = 1.3cm, auto, place/.style = {circle,  thick, draw=blue!75, fill=blue!20}]
 \draw[blue, ultra thick] (-3,-3)--(3,-3);
 \draw[blue, ultra thick, ->] (-3,-3)--(0,-3);
 \draw (-3, -2)--(3,-2);
 \draw (-3,-1)--(3,-1);
 \draw (-3,0)--(3,0);
 \draw (-3,1)--(3,1);
 \draw (-3,2)--(3,2);
 \draw[blue, ultra thick] (-3,3)--(3,3);
 \draw[blue, ultra thick, ->] (-3,3)--(0,3);
 \draw[red, ultra thick](-3,-3)--(-3,3);
 \draw[red, ultra thick, ->](-3,-3)--(-3,0);
 \draw(-2,-3)--(-2,3);
 \draw(-1,-3)--(-1,3);
 \draw(0,-3)--(0,3);
 \draw(1,-3)--(1,3);
 \draw(2,-3)--(2,3);
 \draw[red, ultra thick](3,-3)--(3,3);
 \draw[red, ultra thick,->] (3,-3)--(3,0);
 \draw [black, ultra thick] (0,0) circle [radius = .1];
 \node [below] at (0,0) {sink};
\end{tikzpicture}
\caption{The square lattice configuration with periodic boundary condition and a single sink.}\label{fig:periodic_boundary}
\end{figure}
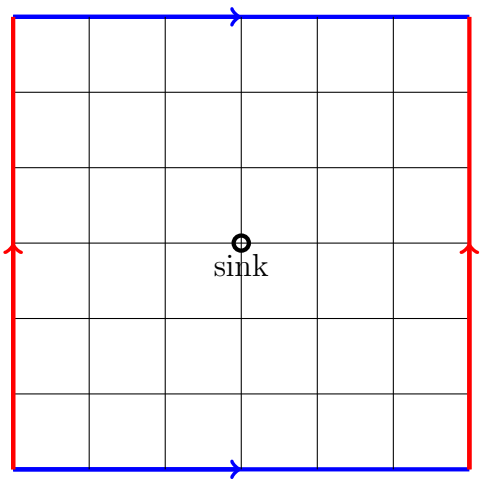

In treating graphs with open boundary condition, further symmetry on the tiling $\sT$ is assumed.  In two dimensions, assume that there are vectors $v_1, ..., v_k$ in which $\sT$ has translational symmetry, and lines $\ell_1, ..., \ell_k$, $\ell_i = \{x  \in \bR^2: \langle x, v_i\rangle = 0\}$ such that $\sT$ has reflection symmetry in the family of lines \begin{equation}\sF = \{n v_i + \ell_i: 1\leq i \leq k, n \in \zed\}.\end{equation}  In this case, let $\sR$ be an open, connected, convex region cut out by some of the lines, and assume further that $\bR^2$ is tiled by the reflections of $\sR$ in the family of lines and that any sequence of reflections which maps $\sR$ to itself is the identity map.  Examples of such families of lines are the lines in the square, triangular,  and tetrakis square tilings. 

\begin{figure}\centering
 \begin{tikzpicture}[scale = .5, node distance = 1.3cm, auto, place/.style = {circle, 
 thick, draw=blue!75, fill=blue!20}]
  \draw (0, 0)--(5,0);
  \draw (0,1)--(5,1);
  \draw (0,2)--(5,2);
  \draw (0,3)--(5,3);
  \draw (0,4)--(5,4);
  \draw (0,5)--(5,5);
  \draw (0,0)--(0,5);
  \draw (1,0)--(1,5);
  \draw (2,0)--(2,5);
  \draw (3,0)--(3,5);
  \draw (4,0)--(4,5);
  \draw (5,0)--(5,5);
\end{tikzpicture}
\begin{tikzpicture}[node distance = 1.3cm, auto, place/.style = {circle, 
thick, draw=blue!75, fill=blue!20}]
  \draw (-1, .86602540378)--(2, .86602540378);
  \draw (-1,0)--(2,0);
  \draw (-1, -.86602540378)--(2,-.86602540378);
  \draw (-1, 1.73205080757)--(2, 1.73205080757);
  \draw (-0.5, -.86602540378)--(1, 1.73205080757);
  \draw (-1, 0)--(0, 1.73205080757);
  \draw (.5, -.86602540378)--(2, 1.73205080757);
  \draw (2, 0)--(1, 1.73205080757);
  \draw (1.5, -.86602540378)--(0, 1.73205080757);
  \draw (.5, -.86602540378)--(-1, 1.73205080757);
  \draw (-.5, -.86602540378)--(-1, 0);
   \draw (1.5, -.86602540378)--(2, 0);
 \end{tikzpicture}
  \begin{tikzpicture}[scale = .5, node distance = 1.3cm, auto, place/.style = {circle, 
  thick, draw=blue!75, fill=blue!20}]
  \draw (0, 0)--(5,0);
  \draw (0,1)--(5,1);
  \draw (0,2)--(5,2);
  \draw (0,3)--(5,3);
  \draw (0,4)--(5,4);
  \draw (0,5)--(5,5);
  \draw (0,0)--(0,5);
  \draw (1,0)--(1,5);
  \draw (2,0)--(2,5);
  \draw (3,0)--(3,5);
  \draw (4,0)--(4,5);
  \draw (5,0)--(5,5);
  \draw (0,0)--(5,5);
  \draw (1,0)--(5,4);
  \draw (2,0)--(5,3);
  \draw (3,0)--(5,2);
  \draw (4,0)--(5,1);
  \draw (0,1)--(4,5);
  \draw (0,2)--(3,5);
  \draw (0,3)--(2,5);
  \draw (0,4)--(1,5);
  \draw (5,0)--(0,5);
  \draw (5,1)--(1,5);
  \draw (5,2)--(2,5);
  \draw (5,3)--(3,5);
  \draw (5,4)--(4,5);
  \draw (4,0)--(0,4);
  \draw (3,0)--(0,3);
  \draw (2,0)--(0,2);
  \draw (1,0)--(0,1);
\end{tikzpicture}
\caption{The square, triangular and tetrakis square lattices are examples of tilings with reflecting families of lines such that the quotient by the reflection group is a bounded convex region of the plane.}\label{fig:reflecting_planes}
\end{figure}
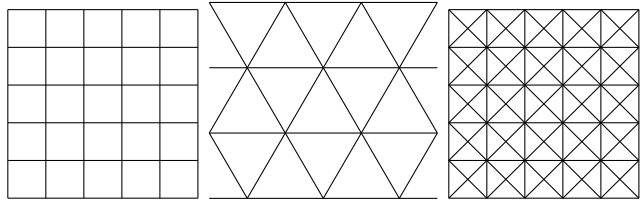

In $d \geq 3$ dimensions, impose the further constraint that, after an orthogonal transformation and dilation $\sT$ is $\zed^d$ periodic and has reflection symmetry in the family $\sF$ of coordinate hyperplanes $H_{i,j}$
\begin{equation}
 H_{i,j} = \{x \in \bR^d: x_i = j\}, \qquad 1 \leq i \leq d, j \in \zed.
\end{equation}
After the transformation, $\sR = (0,1)^d$. 

The open boundary graphs are constructed as follows.
\begin{enumerate}
 \item[(2)] (Open boundary condition) If  the following condition holds
 \begin{enumerate}
 \item[A.] No edge of $\sT$ crosses a face of $\sR$,
 \end{enumerate}
 then a graph $\sT_m$ is obtained by identifying all vertices of $\sT \cap (m \cdot \sR)^c$ and designating this `boundary' vertex the sink. 
\end{enumerate}

Note that, although many planar tilings lack lines of reflection symmetry, all of those planar tilings considered by \cite{KW16} are of the type considered, and all but the Fisher tiling satisfy A, see the examples in Figure \ref{fig:open_boundary} in which the reflecting lines are in red, and the vertices on the boundary are sinks.

The $\Dfour$ lattice in dimension four is another example  which satisfies condition A with the appropriate choice of reflecting hyperplanes.  The $\Dfour$ lattice has vertices $\zed^4 \cup \zed^4 + (\frac{1}{2}, \frac{1}{2}, \frac{1}{2}, \frac{1}{2})$ and 24 nearest neighbors of 0
\begin{equation}
U_4 = \{ \pm e_1, \pm e_2, \pm e_3 , \pm e_4\} \cup \left\{\frac{1}{2} (\epsilon_1, \epsilon_2, \epsilon_3, \epsilon_4), \epsilon_i \in \{\pm 1\}\right\},
\end{equation}
which have unit Euclidean length. The elements of the $\Dfour$ lattice are frequently identified with the `Hurwitz quaternion algebra' in which $U_4$ is the group of units. Let 
\begin{align*}v_1=(1,1,0,0),\; v_2=(1,-1,0,0),\; v_3 = (0,0,1,1),\; v_4 = (0,0,1,-1),
\end{align*}
 and define hyperplanes
\begin{equation*}
 \sP_j = \{x \in \bR^4: \langle x, v_j\rangle =0\}
\end{equation*}
and family of hyperplanes
\begin{equation}
 \sF_{\Dfour} = \{ n v_j + \sP_j: j \in \{1, 2, 3, 4\}, n \in \zed\}.
\end{equation}
\begin{lemma}
 The $\Dfour$ lattice has reflection symmetry in the family of hyperplanes $\sF_{\Dfour}$.  After a rotation and scaling, $\Dfour$ together with this family satisfy property A.
\end{lemma}
\begin{proof}
 Since $\Dfour$ is a lattice, which is invariant under permuting the coordinates, it suffices to prove the reflection symmetry property for $\sP_1$.  Given $x \in \Dfour$, its reflection in $\sP_1$ is
 $x' = x - \langle x, v_1\rangle v_1$.  Since $\langle x, v_1\rangle \in \zed$, the claim holds.
 
 Since the vectors $v_1, v_2, v_3, v_4$ are orthogonal and of equal length, after a rotation and scaling the planes in $\sF_{\Dfour}$ coincide with the coordinate hyperplanes.
 
 To prove that property A is satisfied, it suffices by symmetry to prove that there are not edges crossing $\sP_1$.  Suppose for contradiction that $x$ and $y$ are connected, so that $\|x-y\|_2 = 1$, and that the line segment connecting $x=(x_1,x_2,x_3,x_4)$ and $y=(y_1,y_2,y_3,y_4)$ crosses $\sP_1$, say at $z= (z_1,z_2,z_3,z_4)$.  It follows that $z_1 + z_2 = 0$.  Assume without loss of generality that $x_1 + x_2 >0$ and $y_1 + y_2 < 0$.  Since the sum of these coordinates is integer valued, $x_1 + x_2 \geq 1$ and $y_1 + y_2 \leq -1$.  Thus $(x_1 + x_2)-(y_1 + y_2) \geq 2$ so $\|x-y\|_2 \geq \sqrt{2}$, a contradiction.
\end{proof}

\noindent
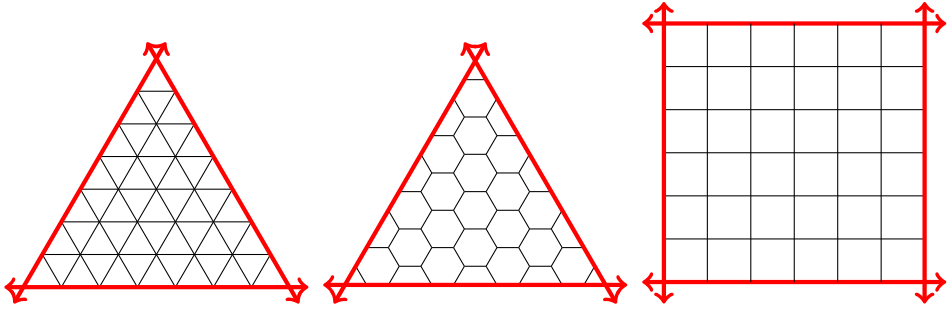
\begin{figure}
\centering
\begin{tikzpicture}[scale = .5, node distance = 1.3cm, auto, place/.style = {circle, 
thick, draw=blue!75, fill=blue!20}]
  \draw (0,0)--(7,0);
  \draw (.5, .86602540378)--(6.5, .86602540378);
  \draw (1, 2*.86602540378) --(6, 2*.86602540378);
  \draw (1.5, 3*.86602540378)--(5.5, 3*.86602540378);
  \draw (2, 4*.86602540378)--(5, 4*.86602540378);
  \draw (2.5, 5*.86602540378)--(4.5, 5*.86602540378);
  \draw (3, 6* .86602540378)--(4, 6*.86602540378);
  \draw (0,0)--(3.5, 7*.86602540378);
  \draw (1,0)--(4, 6*.86602540378);
  \draw (2,0)--(4.5, 5*.86602540378);
  \draw (3,0)--(5, 4*.86602540378);
  \draw (4,0)--(5.5, 3*.86602540378);
  \draw (5,0)--(6, 2*.86602540378);
  \draw (6,0)--(6.5, .86602540378);
  \draw (3.5, 7*.86602540378)--(7,0);
  \draw (3, 6*.86602540378)--(6,0);
  \draw (2.5, 5*.86602540378)--(5,0);
  \draw (2, 4*.86602540378)--(4,0);
  \draw (1.5, 3*.86602540378)--(3,0);
  \draw (1, 2*.86602540378)--(2,0);
  \draw (.5, .86602540378)--(1,0);
 \draw [red, ultra thick,<->] (3.25, 7.5*.86602540378)-- (7.25, -.5*.86602540378);
 \draw [red, ultra thick, <->] (3.75, 7.5*.86602540378)--(-.25, -.5*.86602540378);
 \draw [red, ultra thick, <->] (-.5, 0)--(7.5,0);
 \end{tikzpicture}
 \begin{tikzpicture}[scale =.5 *(8./14.), node distance = 1.3cm, auto, place/.style = {circle, 
 thick, draw=blue!75, fill=blue!20}]
  \draw (1,0)--(2,0);
  \draw (1,0)--(.5, .86602540378);
  \draw (2,0)--(2.5, .86602540378);
  \draw (.5, .86602540378)-- (1, 2*.86602540378);
 \draw (1, 2*.86602540378)--(2, 2*.86602540378);
  \draw (2.5, .86602540378)--(2,2*.86602540378);
  \draw(2.5, .86602540378)--(3.5, .86602540378);
  \draw (3.5,.86602540378)--(4, 2*.86602540378);
  \draw (3.5, .86602540378)--(4,0);
  \draw (4,0)--(5,0);
  \draw (5,0)--(5.5, .86602540378);
  \draw (5.5, .86602540378)-- (5, 2*.86602540378);
  \draw (4, 2*.86602540378)--(5, 2*.86602540378);
  \draw (2, 2*.86602540378)--(2.5, 3*.86602540378);
  \draw (2.5, 3*.86602540378)--(3.5, 3*.86602540378);
  \draw (3.5, 3*.86602540378)--(4, 2*.86602540378);
  \draw (5.5, .86602540378)--(6.5, .86602540378);
  \draw (6.5, .86602540378)--(7, 0);
  \draw (7,0)--(8,0);
  \draw (8,0)--(8.5, .86602540378);
  \draw (8.5, .86602540378)-- (8, 2*.86602540378);
  \draw (7, 2*.86602540378) -- (8, 2*.86602540378);
  \draw (6.5, .86602540378)--(7, 2*.86602540378);
  \draw (8.5, .86602540378)--(9.5,.86602540378);
  \draw (9.5, .86602540378)--(10,0);
  \draw (10,0)--(11,0);
  \draw (11,0)--(11.5, .86602540378);
  \draw (11.5, .86602540378)--(11, 2*.86602540378);
  \draw (10, 2*.86602540378)--(11, 2*.86602540378);
  \draw (9.5, .86602540378)--(10, 2*.86602540378);
  \draw (5, 2*.86602540378)--(5.5, 3*.86602540378);
  \draw (5.5, 3*.86602540378)--(5, 4*.86602540378);
  \draw (4, 4*.86602540378)--(5, 4*.86602540378);
  \draw (3.5, 3*.86602540378)--(4, 4*.86602540378);
  \draw (5.5, 3*.86602540378)--(6.5, 3*.86602540378);
  \draw (6.5, 3*.86602540378)--(7, 2*.86602540378);
  \draw (6.5, 3*.86602540378)--(7, 4*.86602540378);
  \draw (7, 4*.86602540378)--(8,4*.86602540378);
  \draw (8, 4*.86602540378)--(8.5, 3*.86602540378);
  \draw (8, 2*.86602540378)--(8.5, 3*.86602540378);
  \draw (8.5, 3*.86602540378)--(9.5, 3*.86602540378);
  \draw (9.5, 3*.86602540378)--(10, 2*.86602540378);
  \draw ( 2.5, 3*.86602540378)--( 2, 4*.86602540378);
  \draw (2, 4*.86602540378)--(2.5, 5*.86602540378);
  \draw (2.5, 5*.86602540378)--(3.5, 5*.86602540378);
  \draw (3.5, 5*.86602540378)--(4, 4*.86602540378);
  \draw (3.5, 5*.86602540378)--(4, 6*.86602540378);
  \draw (4, 6*.86602540378)--(5, 6*.86602540378);
  \draw (5, 6*.86602540378)--(5.5, 5*.86602540378);
  \draw (5.5, 5*.86602540378)--(5, 4*.86602540378);
  \draw (5.5, 5*.86602540378)--(6.5, 5*.86602540378);
  \draw (6.5, 5*.86602540378)--(7, 4*.86602540378);
  \draw (6.5, 5*.86602540378)--(7, 6*.86602540378);
  \draw (7, 6*.86602540378)--(8, 6*.86602540378);
  \draw (8, 6*.86602540378)--(8.5, 5*.86602540378);
  \draw (8.5, 5*.86602540378)--(8, 4*.86602540378);
  \draw (8.5, 5*.86602540378)--(9.5, 5*.86602540378);
  \draw (9.5, 5*.86602540378)--(10, 4*.86602540378);
  \draw (10, 4*.86602540378)--(9.5, 3*.86602540378);
  \draw (4, 6*.86602540378)--(3.5, 7*.86602540378);
  \draw (3.5, 7*.86602540378)--(4, 8*.86602540378);
  \draw (4, 8*.86602540378)--(5, 8*.86602540378);
  \draw (5, 8*.86602540378)--(5.5, 7*.86602540378);
  \draw (5.5, 7*.86602540378)--(5,6*.866025403780);
  \draw (5.5, 7*.86602540378)--(6.5, 7*.86602540378);
  \draw (6.5, 7*.86602540378)--(7, 6*.86602540378);
  \draw (6.5, 7*.86602540378)--(7, 8*.86602540378);
  \draw (7, 8*.86602540378)--(8, 8*.86602540378);
  \draw (8, 8*.86602540378)--(8.5, 7*.86602540378);
  \draw (8.5, 7*.86602540378)--(8, 6*.86602540378);
  \draw (5, 8*.86602540378)--(5.5, 9*.86602540378);
  \draw (5.5, 9*.86602540378)--(6.5, 9*.86602540378);
  \draw (6.5, 9*.86602540378)--(7, 8*.86602540378);
  \draw (5.5, 9*.86602540378)--(5, 10*.86602540378);
  \draw (5, 10*.86602540378)--(5.5, 11*.86602540378);
  \draw (5.5, 11*.86602540378)--(6.5, 11*.86602540378);
  \draw (6.5, 11*.86602540378)--(7, 10*.86602540378);
  \draw (7, 10*.86602540378)--(6.5, 9*.86602540378);
  \draw[red, ultra thick, <->] (-1, 0) -- (13,0);
  \draw[red, ultra thick, <->] (-.5, -.86602540378)--(6.5 , 13*.86602540378);
  \draw[red, ultra thick, <->] (12.5, -.86602540378)--(5.5, 13*.86602540378);
 \end{tikzpicture}
\begin{tikzpicture}[scale = .5 *(8./7.), node distance = 1.3cm, auto, place/.style = {circle, 
thick, draw=blue!75, fill=blue!20}]
 \draw[red, ultra thick, <->] (-3.5,-3)--(3.5,-3);
 \draw (-3, -2)--(3,-2);
 \draw (-3,-1)--(3,-1);
 \draw (-3,0)--(3,0);
 \draw (-3,1)--(3,1);
 \draw (-3,2)--(3,2);
 \draw[red, ultra thick, <->] (-3.5,3)--(3.5,3);
 \draw[red, ultra thick, <->](-3,-3.5)--(-3,3.5);
 \draw(-2,-3)--(-2,3);
 \draw(-1,-3)--(-1,3);
 \draw(0,-3)--(0,3);
 \draw(1,-3)--(1,3);
 \draw(2,-3)--(2,3);
 \draw[red, ultra thick, <->](3,-3.5)--(3,3.5);
\end{tikzpicture}
\caption{The triangular, hex and square lattice configurations with open boundary condition.}\label{fig:open_boundary}
\end{figure}

\subsubsection{Spectral factors}

The results concerning sandpile dynamics are proved by studying the spectrum of the sandpile transition kernel. 
Denote $\Delta$ the graph Laplacian $\Delta f(v) = \deg(v) f(v) - \sum_{(v,w) \in E} f(w)$. Given a function $f$ on $\sT$, say that $f$ is \emph{harmonic modulo 1} if $\Delta f \equiv 0 \bmod 1$ and denote 
\begin{equation}
\sH(\sT) = \left\{f: \sT \to \bR, \Delta f \equiv 0 \bmod 1\right\}
\end{equation}
and $\sH^2(\sT) = \sH(\sT) \cap \ell^2(\sT)$.
Define, also, the function classes
\begin{align*}
 C^0(\sT) &= \{f: \sT \to \zed: f \in \ell^1(\sT)\},\\
 C^1(\sT) &= \left\{f \in C^0(\sT): \sum_{t \in \sT} f(t) = 0\right\}.
\end{align*}
In the case of a torus boundary condition, define the \emph{spectral parameter}
\begin{equation}
 \gamma = \inf \left\{ \sum_{x \in \sT} 1 - \cos(2\pi \xi_x): \xi \in \sH^2(\sT), \Delta \xi \in C^1(\sT), \xi \not \equiv 0 \bmod 1\right\}.
\end{equation}

In two dimensions, let $\sL$ denote the set of lines which make up a segment of the boundary of $\sR$ and let $\sC$ be the set of pairs of lines from $\sL$ which intersect at a corner of the boundary of $\sR$.  Write an affine line $a \in \sL$ as $a = nv + \ell$ where $v \in \bR^2$ and $\ell$ is the perpendicular line.  A pair of affine lines $(a_1, a_2) \in \sC$ have $\ell_1$ and $\ell_2$ that split $\sT$ into four quadrants.   Let $Q_{(a_1, a_2)}$ be the quadrant whose translate contains $\sR$.   Given $a \in \sL$, let $\sH_a^2(\sT)$ be those functions $\xi \in \sH^2(\sT)$ which are anti-symmetric in $\ell$, similarly given $(a_1, a_2) \in \sC$, let $\sH_{(a_1, a_2)}^2(\sT)$ be those functions in $\sH(\sT)$ which are anti-symmetric in $\ell_1$ and $\ell_2$.  Define spectral parameters
\begin{align*}
 \gamma_0 &= \inf_{\substack{\xi \in \sH^2(\sT)\\ \xi \not \equiv 0 \bmod 1}} \sum_{x \in \sT} 1 - \cos(2\pi \xi_x),\\
 \gamma_1 &= \frac{1}{2} \inf_{a \in \sL} \inf_{\substack{\xi \in \sH_a^2(\sT)\\ \xi \not \equiv 0 \bmod 1}} \sum_{x \in \sT} 1 - \cos(2\pi \xi_x),\\
 \gamma_2 &= \inf_{(a_1, a_2) \in \sC} \inf_{\substack{\xi \in \sH_{(a_1, a_2)}^2(\sT)\\ \xi \not \equiv 0 \bmod 1}} \sum_{x \in Q_{(a_1,a_2)}} 1 - \cos(2\pi \xi_x).
\end{align*}

In the case of $d \geq 3$, assume that a rotation and dilation have been performed so that reflecting hyperplanes are given by $H_{i,j}$ as above.  
Given a set $S \subset \{1, 2, ..., d\}$  let $\fS_S$ be the group generated by reflections in $\{H_{i,0}: i \in S\}$, and let 
$\sH_S^2(\sT)$ denote those $\sH^2(\sT)$ functions which are anti-symmetric in $H_{i,0}$ for all $i \in S$, identified with functions on $\sT/\fS_S$.  Again, for $0 \leq i \leq d-1$ define the \emph{spectral parameters}
\begin{equation}
 \gamma_i =  \inf_{\substack{S \subset \{1, 2, ..., d\}\\ |S| = i}} \inf_{\substack{\xi \in \sH_S^2(\sT)\\ \xi \not \equiv 0 \bmod 1}} \sum_{x \in \sT/\fS_S} 1- \cos(2\pi \xi_x).
\end{equation}
Note that the definition of $\gamma_0$ differs from that of $\gamma$ in that the inf requires only that $\Delta \xi \in C^0(\sT)$, not $C^1(\sT)$.
In dimension $d \geq 2$ define the $j$th \emph{spectral factor} 
\begin{equation}
 \Gamma_j = \frac{d-j}{\gamma_j}
\end{equation}
and $\Gamma = \max_j \Gamma_j$.

\subsubsection{Statement of results}

The following theorem determines the spectral gap of sandpile dynamics for plane and space tiling graphs asymptotically.
\begin{theorem}\label{spectral_gap_theorem}
 Given a tiling $\sT$, as $m \to \infty$, the spectral gap of the transition kernel of sandpile dynamics on $\bT_m$ satisfies
 \begin{equation}
  \gap_{\bT_m} = (1 + o(1)) \frac{\gamma}{|\bT_m|}.
 \end{equation}
If $\sT$ has a family of reflection symmetries $\sF$ and satisfies condition A, then the spectral gap of the transition kernel of sandpile dynamics on $\sT_m$ satisfies
 \begin{equation}
  \gap_{\sT_m} = (1 + o(1)) \frac{\min(\gamma,  \gamma_j: j \geq 1)}{|\sT_m|}.
 \end{equation}

\end{theorem}


The following theorem demonstrates a cut-off phenomenon in sandpile dynamics on general tiling graphs with either a torus or open boundary condition. Whereas the mixing of sandpiles with torus boundary condition is controlled by the spectral gap, when there is an open boundary condition, the mixing time is controlled by the spectral factor.
\begin{theorem}\label{mixing_theorem}
For a fixed tiling $\sT$ in $\bR^d$, sandpiles started from a recurrent state on $\bT_m$ have asymptotic total variation mixing time
\begin{equation}
 t_{\mix}(\bT_m) \sim \frac{\Gamma_0}{2} |\bT_m| \log m
\end{equation}
with a cut-off phenomenon as $m \to \infty$.

If the tiling $\sT$ satisfies the reflection condition and condition A  then sandpile dynamics started from a recurrent configuration on $\sT_m$ have total variation mixing time
\begin{equation}
 t_{\mix}(\sT_m) \sim \frac{\Gamma}{2} |\sT_m| \log m
\end{equation}
with a cut-off phenomenon as $m \to \infty$. 
\end{theorem}

Motivated by Theorem \ref{mixing_theorem}, if $\Gamma = \Gamma_0$  say that the \emph{bulk} or \emph{top dimensional behavior}  controls the total variation mixing time, and otherwise that the \emph{boundary behavior} controls the total variation mixing time.  The proof of Theorem \ref{mixing_theorem} will in fact generate a statistic which randomizes at the mixing time, and this statistic is either distributed throughout the graph, or concentrated near the boundary of the dimension controlling the spectral factor.

\begin{cor}\label{open_boundary_corollary}
 All plane tilings satisfying the reflection condition and condition A  have total variation mixing time controlled by the bulk behavior.
\end{cor}
\begin{proof}
 It suffices that $\Gamma_1 \leq \Gamma_0$.  Indeed, the factor of $2^{-1}$ in $\gamma_1$ is canceled by the ratio $\frac{2}{2-1}$ of dimensions, and the anti-symmetry condition imposes an extra constraint on the harmonic modulo 1 function in the inf, so that $\frac{1}{2}\gamma_0 \leq\gamma_1$.
\end{proof}
In particular, Corollary \ref{open_boundary_corollary} implies that asymptotic mixing time of sandpile dynamics on the square grid with open and periodic boundary condition are the same to top order, answering a question raised in \cite{HJL17}.

In \cite{HS19} the above theorems are supplemented  by explicit verifications of spectral gaps for several tilings.  
The plane tilings considered  are the triangular ($\tri$) and  honeycomb ($\hex$) tilings the triangular face centered cubic tiling ($\fcc$) in 3 dimensions, which is the lattice tiling generated by vectors
\begin{equation}
 v_1 = (1,0,0), \qquad v_2 = \left(\frac{1}{2}, \frac{\sqrt{3}}{2},0\right), \qquad v_3 = \left(\frac{1}{2}, \frac{1}{2\sqrt{3}}, \sqrt{\frac{2}{3}}\right)
\end{equation}
with nearest neighbor edges.   All of the spectral parameters are obtained for the $\Dfour$ lattice in dimension 4 for a specific set of bounding hyperplanes. The results of \cite{HS19} are summarized in the following theorem.

\begin{theorem}\label{spectral_gap_calc_theorem}
 The triangular, honeycomb, 
 and face centered cubic tilings have periodic boundary spectral parameters\footnote{The digit in parenthesis indicates the last significant digit.}
 \begin{align*}
 \gamma_{\tri} &= 1.69416(6)\\
 \notag \gamma_{\hex} &= 5.977657(7)\\
 \notag \gamma_{\fcc} &= 0.3623(9).
 \end{align*}
The spectral parameters of the $\Dfour$ lattice with reflection planes $\sF_{\Dfour}$ and open boundary condition are ($\vartheta$ denotes a parameter bounded by 1 in size)
\begin{align*}
\gamma_{\Dfour, 0} &= 0.075554+ \vartheta 0.00024, \\
\gamma_{\Dfour, 1} &= 0.0440957 +\vartheta 0.00017, \\
\gamma_{\Dfour, 2} &= 0.0389569 +\vartheta 0.00013, \\
\gamma_{\Dfour, 3} &= 0.036873324 +\vartheta 0.00012, \\
\gamma_{\Dfour, 4} &= 0.0357604+ \vartheta 0.00011.
\end{align*}
The spectral factors are given by
\begin{align*}
 \Gamma_{\Dfour,0}  &= 52.9428 + \vartheta 0.17\\
 \Gamma_{\Dfour,1}  &= 68.03486+ \vartheta 0.27\\
 \Gamma_{\Dfour,2}  &=  51.3393 + \vartheta 0.17\\
 \Gamma_{\Dfour,3}  &= 27.1201 + \vartheta 0.084.
\end{align*}

\end{theorem}
Since $\Gamma_{\Dfour, 1}> \Gamma_{\Dfour, 0}$, a particular consequence of Theorem \ref{spectral_gap_calc_theorem} is that the total variation mixing time of the dynamics on the $\Dfour$ lattice is dominated by the three dimensional boundary behavior.

In \cite{HS19} the cubic lattices $\zed^d$ with coordinate hyperplanes are also treated asymptotically. 

\begin{theorem}\label{asymptotic_theorem}
 As $d \to \infty$, the spectral parameter of the $\zed^d$ lattice with periodic boundary condition is
  \begin{equation}
  \gamma_{\zed^d} = \frac{\pi^2}{d^2}\left(1 + \frac{1}{2d} + O\left(d^{-2}\right)\right)
 \end{equation}
 and the parameters with open boundary condition are
\begin{equation}
 \gamma_{\zed^d,j} = \frac{\pi^2}{2d^2} \left(1 + \frac{3}{2d} + O_j\left(d^{-2}\right)\right)
\end{equation}
and, uniformly in $j$,
\begin{equation}
 \gamma_{\zed^d, j} \geq \frac{\pi^2}{2d^2 + d}.
\end{equation}
For each fixed $j$,
\begin{equation}
 \Gamma_j = \frac{2d^3 -(2j+3)d^2 + O_j(d)}{\pi^2}.
\end{equation}
In particular, for all $d$ sufficiently large, the total variation mixing time on $\zed^d$ is dominated by the bulk behavior and $\Gamma = \frac{2d^3}{\pi^2}\left(1 - \frac{3}{2d} + O\left(d^{-2}\right)\right)$.
\end{theorem}
For all sufficiently large $d$, $\gamma_{\zed^d} > \gamma_{\zed^d, 0}$, with $\gamma_{\zed^d, 0}$ achieved by a configuration $\xi$ with $\Delta \xi \in C^0(\sT) \setminus C^1(\sT)$.  Hence the lattice $\zed^d$ with periodic boundary condition gives an example of a tiling for which $\gamma \neq \gamma_0$, that is, the spectral gap and mixing times are controlled by different limiting eigenfunctions.

An important object in this work is the Green's function of a tiling $\sT$ started from a node $v \in \sT$, indicated $g_v(x)$, which satisfies $\Delta g_v(x) = \delta_v(x)$.  Given a function $\eta$ on $\sT$ of bounded support, define the convolution $g*\eta = g_\eta = \sum_{v \in \sT} \eta(v)g_v$.  Theorem \ref{Greens_fn_theorem} of Section \ref{Greens_function_section} gives a general explicit method for obtaining a frequency space representation of the Green's function, which is useful in applications, see \cite{HS19}.  The following theorem is  proved in Section \ref{Greens_function_section}.
\begin{theorem}\label{ell_2_theorem}
 Let $\sT$ be a periodic plane or space tiling in $\bR^d$, $d \geq 2$, and let $\eta$ be a function on $\sT$ of bounded support.  Then $g_\eta \in \ell^2(\sT)$ if and only if $\eta \in C^j(\sT)$ where $j = 2$ if $d = 2$, $j=1$ if $d = 3, 4$ and $j =0$ if $d \geq 5$.
 In particular, \begin{equation}\sH^2(\sT)  = \left\{g_\eta: \eta \in C^j(\sT)\right\}\end{equation} and if $\xi \in \sH^2(\sT)$ then $\xi = g*(\Delta \xi)$.
\end{theorem}
The functions $g_\eta$ are extremal functions for the spectral parameter optimization problems.
The fact that the Green's function itself  just fails to be in $\ell^2(\sT)$ in dimension 4 motivated the calculation of the $\Dfour$ example in which the 3 dimensional boundary dominates the mixing time. 

\subsection{Discussion of method}
  The results  build on the recent work of the the first author, Jerison and Levine \cite{HJL17}, which determined the asymptotic mixing time and obtained a cut-off phenomenon for sandpile dynamics on the torus $(\zed/m\zed)^2$ as $m \to \infty$. Since the sandpile group of a graph with sink $s$ is isomorphic to $\sG = \zed^{V \setminus \{s\}}/\Delta' \zed^{V \setminus \{s\}}$ where $\Delta'$ is the reduced graph Laplacian obtained by omitting the row and column corresponding to the sink, the dual group is isomorphic to $\hat{\sG} = (\Delta')^{-1} \zed^{V \setminus \{s\}}/ \zed^{V \setminus \{s\}}$.  Thus $\Delta'$ provides a natural mapping from $\hat{\sG} \to \sG$.  A map in the reverse direction may be constructed via convolution with the graph Green's function. The necessary theory and analytic properties needed to study the Green's function on a periodic or open piece of a plane or space tiling is developed here, using a stopped random walk on the graph and is obtained by combining a local limit theorem for the random walk in time domain with a frequency domain representation.  Since a tiling lacks the abelian group structure of a lattice, compared to the previous work, the determination of the Green's function in the tiling as opposed to lattice case is more involved.  It is reduced to the lattice case by stopping a random walk on the tiling when it hits the period lattice, and using the resulting stopped measure to determine the Green's function restricted to the lattice.  An explicit formula for the Fourier transform of the Green's function restricted to the lattice is given in Theorem \ref{Greens_fn_theorem}.   As in \cite{HJL17}, van der Corput's method from the theory of exponential sums is used to reduce the determination of the maximum spectral factor to a finite check, and to prove an approximate spectral disjointness for frequencies $\xi \in \hat{\sG}$ for which $\nu = \Delta' \xi$ is separated into a small number of separated clusters.

\subsection{Historical review}
 
Sandpile dynamics on a finite piece of the square lattice were first considered by Bak, Tang and Wiesenfeld \cite{BTW88} in a study of self-organized criticality, see also Dhar \cite{D89}, where an arbitrary graph is considered.  In \cite{LP10} driven dynamics on the square grid with open boundary are considered and a picture is given of the identity element in the sandpile group.  In \cite{S15} numerical studies are made of sandpile statistics on a square grid with open boundary, but the statistics are measured at a point prior to the mixing time in Theorem \ref{mixing_theorem}.  

Sandpiles have been studied on a large number of different graph geometries.  The hex tiling is considered in \cite{ADMR10}, the graph of the dihedral group $D_n$ is considered in \cite{DFF03}, the Husimi lattice is studied in \cite{PS95}. Mixing time estimates for sandpiles on many graphs, some of which are asymptotics with cut-off, are given in \cite{JLP15}.  All of the cases obtained here are new, however, except for the square tiling with periodic boundary treated in \cite{HJL17}.  Several sandpile statistics are calculated for two dimensional tilings in \cite{KW16}, which was the original motivation for this project.

The effect of the boundary condition on sandpile behavior has been studied extensively, although this is the first treatment of the spectral gap and mixing time.  See \cite{BIP93}, \cite{I94}, \cite{JPR06}, \cite{PR05} and \cite{ADMR10} for height probabilities and correlation functions. In \cite{HJL17} the asymptotic mixing time and a cut-off phenomenon were proved for sandpile dynamics on the rectangular grid with periodic boundary condition.  Theorem \ref{mixing_theorem} generalizes this result to sandpile dynamics on an arbitrary plane or space tiling.  In \cite{HJL17} it was conjectured that a cut-off phenomenon also exists on the square grid with open boundary condition, which is proved here, and in Corollary \ref{open_boundary_corollary} it is demonstrated that the asymptotic mixing time is the same as for the periodic boundary case.  Theorem \ref{spectral_gap_calc_theorem} gives an example in 4 dimensions in which the two mixing times are asymptotically unequal.

In \cite{H17} a random walk is studied with generators given by the powers of 2 in $\zed/p\zed$ and a similar treatment of the boundary occurs where multiples of the largest power of 2 wrap around $p$.  In that case, the boundary does not influence the leading order asymptotic mixing time.
\section{Notation and conventions}
The additive character on $\bR/\zed$ is written $e(x) = e^{2\pi i x}$.  Write, also, $c(x) = \cos 2\pi x$ and $s(x) = \sin 2\pi x$.  For real $x$, $\|x\|_{\bR/\zed}$ denotes the distance to the nearest integer, while for $x \in \bR^d$, $\|x\|_{\bR^d/\zed^d}$ denotes the distance to the nearest lattice point in $\zed^d$.

We use the notations $A \ll B$ and $A = O(B)$ to mean that there is a constant 
$0<C<\infty$ such that $|A| < CB$, and $A \asymp B$ to mean $A \ll B \ll A$. A subscript such as $A \ll_R B$, $A = O_R(B)$ means that the constant $C$ depends on $R$. The notation $A = o(B)$ means that $A/B$ tends to zero as the relevant parameter tends to infinity.

Given a graph $G = (V,E)$ and vertices $v, w \in V$, the degree of $v$ is $\deg(v)$ and the number of edges from $v$ to $w$ is $\deg(v,w)$. The notation $d(v,w)$ indicates the graph distance from $v$ to $w$, which is the length of the shortest path from $v$ to $w$. The graph Laplacian $\Delta$ operates on functions on $G$ by
\begin{equation}
 \Delta f(v) = \deg(v) f(v) - \sum_{(v,w) \in E} f(w).
\end{equation}
The notation $\delta_v(w)$ indicates a point mass at $v$, which takes value 1 if $v = w$ and 0 otherwise.  The Green's function started at $v$ on an infinite  graph $G$ is a function $g_v(w)$ such that $\Delta g_v(w) = \delta_v(w)$. If the graph is finite, $\Delta (g_{v_1}-g_{v_2})(w) = \delta_{v_1}(w) - \delta_{v_2}(w)$. Convolution of the Green's function $g$ with a function $\eta$ of finite support in $V$ with sum of values 0 is defined by
\begin{equation}
 g*\eta = \sum_{v \in V} \eta(v) g_v.
\end{equation}
Thus, on a finite graph, $\Delta g*\eta = \eta$ if the sum of the values of $\eta$ is 0. The notation $g_\eta$ for $g*\eta$ is also used.  When $G$ is a finite graph and a node $s$ has been designated sink, the reduced Laplacian $\Delta'$ is obtained from $\Delta$ by removing the row and column corresponding to the sink. 

A random walk on $G$ proceeds in discrete time steps.  At a given time step, each edge from a given node $v$ is chosen with equal probability as a transition. The transition kernel of this random walk is
\begin{equation}
 P(v,w) = \frac{\deg(v,w)}{\deg (v)}.
\end{equation}
Let $Y_{v, n}$ indicate the walk started at random or deterministic node $v$ and at random or deterministic step $n$.  A \emph{stopping time} adapted to the random walk $Y_{v, n}$ is a random variable $N$ taking values in $\{1, 2, ...\} \cup \{\infty\}$, such that for each deterministic $n$, the event $\{N=n\}$ is measurable in the $\sigma$-algebra generated by the first $n$ steps of the walk.  An important stopping time in this work is the time $T_v$ which is the first positive step when random walk started from $v \in \sT$ reaches the lattice $\Lambda$.

A \emph{sandpile} on a graph $G$ is a map $\sigma: G \to \zed_{\geq 0}$.  The map $\sigma_{\full} = \deg - 1$ is the full sandpile.  The set of stable sandpiles is indicated 
\begin{equation}
\sS(G) = \{ \sigma: G \to \zed_{\geq 0}: \sigma \leq \sigma_{\full}\}.
\end{equation}
The set of recurrent states form the \emph{sandpile group} and are indicated $\sG(G)$.  Its dual group is $\hat{\sG}$.

A function $f$ on $G$ is harmonic if $\Delta f =0$, and harmonic modulo 1 if $\Delta f \equiv 0 \bmod 1$. Let 
\begin{equation}
 \sH(G) = \{ f: G \to \bR, \Delta f \equiv 0 \bmod 1\}.
\end{equation}

Throughout, $\sT$ denotes a plane or space tiling, which is periodic in a lattice $\Lambda$.  The periodic graph $\sT/m\Lambda$ is indicated $\bT_m$, while $\sT_m$ indicates the open boundary graph obtained from a family of reflecting hyperplanes $m\sF$. The notation $g_{\bT_m}$ and $g_{\sT_m}$ indicate the Green's functions on $\bT_m$ or $\sT_m$.  $\sR$ indicates an open convex region (fundamental domain) cut out by the family $\sF$, and whose reflections in $\sF$ tile the plane or space.  $\sT_m$ may be identified with the intersection of $m \sR$ with $\sT$, together with an added point identified with the boundary. Functions on $\sT_m$ are identified with functions on $\sT$, which are reflection anti-symmetric in each hyperplane of $m \sF$. 

The ball $B_R(x) \subset \sT$ is defined to be
\begin{equation}
 B_R(x) = \{y\in \sT: d(x,y) \leq R\}
\end{equation}
where $d(x,y)$ is the graph distance. Since $|\sT/\Lambda| < \infty$, for $x \in \Lambda$, $d(0, x) \asymp \|x\|$, and $\#\{B_R(0)\} \asymp R^d$ as $R \to \infty$.  In $\bT_m$, $B_{R, \bT_m}(x)$ is defined via the quotient distance,  treating points which are equivalent modulo $m\Lambda$ as identified.  On $\sT_m$, $B_{R, \sT_m}(x)$ is defined via the quotient distance in which points which are equivalent under $m \sF$ reflections are identified.

\subsection{Function spaces}
In handling the analysis on a periodic tiling, a key tool is the `harmonic measure' on the period lattice $\Lambda$ obtained by stopping simple random walk started on the tiling when it reaches the lattice.  Random walk started from $v$ in $\sT$ is defined by
\begin{equation}Y_{v,0} = v, \qquad \forall n \geq 0,\; Y_{v,n+1} = P \cdot Y_{v,n},\end{equation} where $P$ is the transition kernel of simple random walk. Let 
\begin{equation}
T_v = \{\min n \geq 1: Y_{v, n} \in \Lambda\} 
\end{equation}
 be the stopping time for simple random walk started at $v$ in $\sT$ and stopped at the first positive time that it returns to $\Lambda$.  For $v \not \in \Lambda$, let 
 \begin{equation}
  \varrho_v \sim Y_{v, T_v}
 \end{equation}
 be the probability distribution of $Y_{v,T_v}$ on $\Lambda$, while for $v \in \Lambda$, let $\rho_v = \delta_v$ be the distribution of a point mass at $v$. Let $\varrho$ have the distribution of $Y_{0, T_0}$ which is the distribution of the first return to $\Lambda$ started at 0.
 
 The following lemma is used to justify convergence when working with the corresponding stopping times and harmonic measures.
\begin{lemma}
 There is a constant $c > 0$ such that, as $n \to \infty$, for all $v \in \sT$, $\Prob(T_v > n) \ll e^{-cn}$.  The measure $\varrho_v$ satisfies $\varrho_v(\{x: d(x,v) > N\}) \ll e^{-cN}$ as $N \to \infty$. Similarly, $\varrho(\{x: d(x,0)>N\}) \ll e^{-cN}$ as $N \to \infty$.
\end{lemma}
\begin{proof}
 The second statement follows from the first, since $T_v$ bounds $d(Y_{v,T_v}, v)$. To prove the first, note that $T_v$ is the same stopping time as the first positive time reaching  0 on the finite state Markov chain given by random walk on $\sT/\Lambda$.  The conclusion follows, since, given any state on $\sT/\Lambda$, there is a bounded number $k$ such that the walk has a positive probability of returning to 0 from the state after $k$ steps.  
\end{proof}

Given a finite, possibly signed measure $\eta$ on $\sT$,  define
\begin{equation}
 \varrho_\eta = \sum_{v \in \sT} \eta(v) \varrho_v.
\end{equation}

Define function classes on $\sT$ by
\begin{align*}
 C^0(\sT) &= \left\{f: \sT \to \zed, \|f\|_1 < \infty\right\}, \\
 C^1(\sT) &= \left\{f \in C^0(\sT), \sum_{x \in \sT} f(x) = 0\right\},\\
 C^2(\sT) &= \left\{f \in C^1(\sT), \sum_{x \in \sT}f(x) \E[Y_{x, T_x}]=0\right\}.
\end{align*}
Hence $C^0(\sT)$ is the set of integer functions of finite support, $C^1(\sT)$ is those functions of sum 0, and $C^2(\sT)$ are those $C^1(\sT)$ functions with zero moment.  

Given a set $S \subset \sT$, say that $f \in C^j(S)$ if, viewed as a function on $\sT$ with support in $S$, $f \in C^j(\sT)$.

Although the definition of $C^2(\sT)$ depends on the lattice $\Lambda$, it is invariant under translating $\sT$  as the following lemma shows.
\begin{lemma}\label{C_2_translation_invariance_lemma}
 Suppose $f \in C^2(\sT)$.  For any $t \in \sT \setminus \Lambda$ let $T_v^t$ denote the stopping time of random walk started at $v$ and stopped the first time that it reaches $t + \Lambda$.  Then
 \begin{equation}
  \sum_{x \in \sT}f(x) \E[Y_{x, T_x^t}] = 0.
 \end{equation}

\end{lemma}
\begin{proof}
Let $\tilde{T}_v^{t}$ be the stopping time of random walk started from $v$ and stopped at the first time greater than $T_v$ at which the walk reaches $t + \Lambda$. Note that, by conditioning on the first visit to $\Lambda$,
\begin{align*}
 &\sum_{x \in \sT} f(x) \E[Y_{x, \tilde{T}^t_x}] \\
 &=\sum_{w \in \Lambda} \left(\sum_{x \in \sT} f(x) \Prob(Y_{x, T_x} = w) \right) \sum_{v \in t+\Lambda} v \cdot \Prob(Y_{w, T^t_w} = v)\\
 &= \left(\sum_{w \in \Lambda} \sum_{x \in \sT} f(x) \Prob(Y_{x, T_x} = w) \right) \\ & \times \left( w + \sum_{v \in t + \Lambda} (v-w) \Prob(Y_{0, T^t_0} = (v-w))\right) = 0.
\end{align*}
The last equality holds, since
\begin{equation}
 \sum_{v \in t + \Lambda} (v-w) \Prob(Y_{0, T_0^t} = (v-w)) 
\end{equation}
is a constant independent of $w$ and 

\begin{align*}
\sum_{w \in \Lambda} \sum_{x \in \sT} f(x)\Prob(Y_{x, T_x} = w) = \sum_{x \in \sT} f(x) &= 0\\
 \sum_{w \in \Lambda} \sum_{x \in \sT} f(x) \Prob(Y_{x, T_x} = w) w = \sum_{x \in \sT} f(x)\E\left[Y_{x, T_x}\right] &= 0.
\end{align*}
The equality  $\E\left[Y_{x, \tilde{T}^t_x}\right] = \E\left[Y_{x, T^t_x}\right]$ holds since random walk started from a node $t$ and stopped at the first time it reaches a node in $\Lambda + t$ has mean $t$.  To check this,  let $t = v_0, v_1, v_2, ..., v_n = x+t$ be a path from $t$ to $x+t \in \Lambda + t$ such that $v_i \not \in \Lambda+t$ for $1 \leq i \leq n-1$.  Let $e_1, e_2, ..., e_n$ be edges with $e_i$ connecting $v_{i-1}$ and $v_i$.  The probability of following $e_1, e_2, ..., e_n$ in succession is $\prod_{j=0}^{n-1} \frac{1}{\deg v_j}$. The probability of following that path in reverse is $\prod_{j=1}^n \frac{1}{\deg v_j}$.  Since $\deg v_0 = \deg v_n$ by $\Lambda$-periodicity, running the path in reverse has the same probability.  This is also true of the path translated by $-x$, which proves the claim regarding expectation.
\end{proof}

Given $\lambda$ in the lattice $\Lambda$, the translation operator $\tau_\lambda$ acts on functions $f$ on $\sT$ or on $\Lambda$ by
\begin{equation}
 \tau_\lambda f(x) = f(x-\lambda).
\end{equation}
Given $f \in C^0(\sT)$, the function
\begin{equation}
f_{\bT_m} = \sum_{\lambda \in \Lambda} \tau_{m\lambda} f
\end{equation}
is $m\Lambda$ periodic.
The classes $C^j$ are extended to $\bT_m$ and $\sT_m$ as follows.
Say $f \in C^j(\bT_m)$ if there is a function $f_0 \in C^j(\sT)$ such that $f = f_{0, \bT_m}$.
Given a family of hyperplanes $\sF$ there is a lattice $\Lambda$ such that any function $f$ having reflection anti-symmetry in $\sF$ is $\Lambda$ periodic.  Say that $f \in C^j(\sT_m)$ if $f$ has reflection anti-symmetry in $m\cdot \sF$ and if there is a function $f_0 \in C^j(\sT)$ such that $f = f_{0, \bT_m}$.

Given $f \in \ell^1(\Lambda)$ and $h \in \ell^\infty(\sT)$,
\begin{equation}
 f*h(x) = \sum_{y \in \Lambda} f(y) h(x-y).
\end{equation}
Similarly, given $f \in \ell^1(\Lambda/m\Lambda)$ and $h \in \ell^\infty(\sT/m\Lambda)$,
\begin{equation}
 f*h(x) = \sum_{y \in \Lambda/m\Lambda} f(y)h(x-y).
\end{equation}

In dimension $d$, identify $\Lambda$ with $\zed^d$ by choice of basis, which is fixed throughout the argument, and let $e_i$ be the $i$th standard basis vector.  Discrete differentiation in the $e_i$ direction is defined by
\begin{equation}
 D_{e_i} f(x) = D_i f(x) = f(x + e_i) - f(x).
\end{equation}
Given a vector $\ua \in \bN^d$, define the differential operator
\begin{equation}
 D^{\ua} f(x) = D_1^{a_1} \cdots D_d^{a_d} f(x).
\end{equation}
The discrete derivatives can be expressed as convolution operators.  Let 
\begin{equation}
 \delta_i(x) = \left\{\begin{array}{ccl} -1 && x = 0\\ 1 && x = -e_i\\ 0 && \text{otherwise}\end{array}\right..
\end{equation}
Thus $D^{\ua}f = \delta_1^{*a_1}* \cdots * \delta_d^{*a_d}*f$.  

Given functions $f_1, ..., f_n$ on $\Lambda$, define
\begin{equation}
 \langle f_1, ..., f_n \rangle = \spn_\zed\{\tau_x f_1, ..., \tau_x f_n: x \in \Lambda\}.
\end{equation}
On the lattice $\Lambda$, 
\begin{align*}
 C^0(\Lambda) &= \langle \one(x = 0) \rangle = \{f: \Lambda \to \zed, \|f\|_1 < \infty\},\\
 C^1(\Lambda) &= \langle \delta_i: 1 \leq i \leq d \rangle,\\
 C^2(\Lambda) &= \langle \delta_i * \delta_j: 1 \leq i\leq j \leq d\rangle.
\end{align*}

Given $f \in \ell^1(\Lambda)$, its Fourier transform is
\begin{equation}
 \hat{f}(x) = \sum_{n \in \Lambda} f(n) e(-n \cdot x).
\end{equation}
On $\Lambda/m\Lambda$, the discrete Fourier transform is
\begin{equation}
 \hat{f}(x) = \sum_{n \in \Lambda/m\Lambda} f(n) e\left(-\frac{n \cdot x}{m} \right).
\end{equation}

\subsection{Results from classical analysis}
The sandpile chain is studied in frequency space, and the techniques combine methods which are probabilistic and from the theory of distribution modulo 1.  
Several techniques from the classical theory of exponential sums are used, including van der Corput's inequality \cite{S04}.
\begin{theorem}[van der Corput's Lemma] Let $H$ be a positive integer.  Then for any complex numbers $y_1, y_2, ..., y_N$,
\begin{equation}
 \left| \sum_{n=1}^N y_n \right|^2 \leq \frac{N+H}{H+1} \sum_{n=1}^N |y_n|^2 + \frac{2(N+H)}{H+1} \sum_{h=1}^{H} \left(1 - \frac{h}{H+1}\right)\left| \sum_{n=1}^{N-h} y_{n+h}\overline{y_n}\right|.
\end{equation}
\end{theorem}
The following basic estimate for the sum of a linear phase is also used.
\begin{lemma}
 Let $\alpha \in \bR/\zed$ and let $N \geq 1$. Then
 \begin{equation}
  \left|\sum_{j=1}^N e(\alpha j)\right| \ll \min\left(N, \left\|\alpha\right\|_{\bR/\zed}^{-1} \right).
 \end{equation}

\end{lemma}
\begin{proof}
 This follows on summing the geometric series.
\end{proof}

Chernoff's inequality is used to control the tail of sums of independent variables, see \cite{TV06}.

\begin{lemma}[Chernoff's inequality]
 Let $X_1, X_2, ..., X_n$ be i.i.d. random variables satisfying $|X_i - \E[X_i]| \leq 1$ for all $i$.  Set $X:= X_1 + \cdots + X_n$ and let $\sigma := \sqrt{\Var(X)}$. For any $\lambda >0$,
 \begin{equation}
  \Prob\left(X-\E[X] \geq \lambda \sigma \right) \leq \max\left(e^{-\frac{\lambda^2}{4}}, e^{\frac{-\lambda \sigma}{2}} \right).
 \end{equation}

\end{lemma}

The following variant of Chernoff's inequality applies to unbounded random variables with exponentially decaying tails.
\begin{lemma}\label{Chernoff_variant_lemma}
 Let $X_1, X_2, ..., X_n$ be i.i.d. non-negative random variables of variance $\sigma^2$, $\sigma >0$, satisfying the tail bound, for some $c>0$ and for all $Z>0$, $\Prob(X_1 > Z) \ll e^{-cZ}$. Let $X = X_1 + X_2 + \cdots + X_n$.  Then for any $\lambda > 1$, for $c_1 = \frac{\sqrt{c\sigma}}{2}$,
 \begin{equation}
  \Prob\left(|X - \E[X]| \geq \lambda \sigma \sqrt{n}\right) \ll e^{-\frac{\lambda^2}{16}} + n e^{-c_1 \lambda^{\frac{1}{2}} n^{\frac{1}{4}}}.
 \end{equation}

\end{lemma}
\begin{proof}
 Let $Z$ be a parameter, $Z \gg n^{\frac{1}{4}}$.  Let $X_i'$ be $X_i$ conditioned on $X_i \leq Z$.  Let $\mu' = \E[X_i']$.  Let $X_i'' = X_i \cdot \one(X_i \leq Z) + \mu' \cdot \one(X_i >Z)$ and $X'' = X_1'' + X_2'' + \cdots + X_n''$. We have
 \begin{align*}
  \E[X_i \cdot \one(X_i \geq Z)] &= - \int_Z^\infty x d\Prob(X_i \geq x)\\
  &=Z \Prob(X_i \geq Z) + \int_Z^\infty \Prob(X_i \geq x) dx\\
  &\ll Z e^{-cZ} + \int_Z^\infty e^{-cx}dx \leq \left(Z + \frac{1}{c}\right) e^{-cZ}.
 \end{align*}
Thus, for some $c'>0$, $\E[X''] = \E[X] + O(ne^{-c'Z})$.  Also,
\begin{align*}
 \Var(X_i) &= \E[(X_i-\E[X_i])^2] \\&\geq \E[(X_i-\E[X_i])^2 \one(X_i \leq Z)] \\&\geq \E[(X_i-\mu')^2 \one(X_i \leq Z)]\\& = \Var(X_i'').
\end{align*}
Since $|X_i''| \leq Z$, for all $n$ sufficiently large, applying Chernoff's inequality,
\begin{align*}
 \Prob(|X-\E[X]| > \lambda \sigma \sqrt{n})&\leq \sum_{i=1}^n \Prob(X_i'' \neq X_i) \\&+ \Prob\left(|X'' -\E[X'']| > \frac{\lambda}{2} \sigma \sqrt{n}\right)\\
 &\ll n e^{-cZ} + 2\max\left(e^{-\frac{\lambda^2}{16}}, e^{-\frac{\lambda \sigma \sqrt{n}}{4 Z}}\right).
\end{align*}
To optimize the exponents, choose $Z^2 = \frac{\lambda \sigma \sqrt{n}}{4c}$ to obtain the claim.
\end{proof}

The local limit theorem for sums of lattice random variables is used in the argument.  As discrete derivatives are needed, a self-contained proof  is given. This is similar to the treatment in \cite{LL10}, but the claim here extends further into the tail of the distribution. The proof is given in Appendix \ref{Green_fn_appendix}.

\begin{theorem}[Local limit theorem]\label{local_limit_theorem}
 Let $\mu$ be a probability measure on $\zed^d$, satisfying the following conditions
 \begin{enumerate}
  \item (Lazy) $\mu(0)>0$
  \item (Symmetric) $\mu(x) = \mu(-x)$
  \item (Generic) $\supp(\mu)$ generates $\bR^d$.  There is a constant $k > 0$ such that $\mu^{*k}$ assigns positive measure to each standard basis vector.
  \item (Exponential tails) There is a constant $c > 0$ such that, for all $r \geq 1$,
  \begin{equation}
   \mu(|x| > r) \ll e^{-c r}.
  \end{equation}
 \end{enumerate}
 Let $\Cov(\mu) = \sigma^2$ where $\sigma$ is a positive definite symmetric matrix. For all $\ua \in \bN^d$ there is a polynomial  $Q_{\ua}(x_1, ..., x_d)$, depending on $\mu$, of degree at most $a_i$ in $x_i$ such that,   for all $N \geq 1$, and all $n \in \zed^d$,
 \begin{align*}
  \delta_1^{*a_1} * &\delta_2^{*a_2} * \cdots * \delta_d^{*a_d}* \mu^{*N}(n) = \frac{\exp\left(-\frac{\left|\sigma^{-1} \left(n + \frac{\ua}{2}\right) \right|^2}{2N} \right)}{N^{\frac{d + |\ua|}{2}}}\\
  &\times \Biggl(Q_{\ua}\left(\frac{n + \frac{\ua}{2}}{\sqrt{N}} \right)+ O\left(\frac{1}{N}\left(1 +\frac{\|n\|}{\sqrt{N}}\right)^{|\ua|+4} \right)  \Biggr)\\
  &+ O_\epsilon\left(\exp\left( -N^{\frac{3}{8} -\epsilon}\right) \right).
 \end{align*}
In the case of the gradient convolution operator $\nabla = \begin{pmatrix} \delta_1 \\ \vdots \\ \delta_d\end{pmatrix}$,
\begin{align*}
 \nabla \mu^{*N}(n) &= -\frac{\sigma^{-2} n}{N} \frac{\exp\left(-\frac{\|\sigma^{-1}n\|^2}{2N} \right)}{(2\pi)^{\frac{d}{2}} N^{\frac{d}{2}} \det \sigma}\\&+ O\left(\frac{\exp\left(-\frac{\|\sigma^{-1}n\|^2}{2N} \right)}{N^{\frac{d+2}{2}}}\left(1+\frac{\|n\|}{\sqrt{N}} \right)^5\right) + O_\epsilon\left(\exp\left( -N^{\frac{3}{8} -\epsilon}\right) \right).
\end{align*}

\end{theorem}

\section{The Green's function of a tiling}\label{Greens_function_section}
This section constructs the Green's function of a periodic tiling and records some of its analytic properties, which are proved in Appendix \ref{Green_fn_appendix}.
For the potential theory of random walks, see \cite{S13}. Special cases are worked out in \cite{KW16}.  

Let $\sT \subset \bR^d$ be a tiling which is $\Lambda$-periodic for a lattice $\Lambda$, $|\sT/\Lambda| < \infty$.  Assume $0 \in \sT$.  Given $v, x \in \sT$, a Green's function $g_v(x)$,  which satisfies
\begin{equation}
 \Delta g_v(x) = \delta_v(x),
\end{equation}
may be obtained iteratively by imposing  the \emph{mean value property}
\begin{equation}\label{iterative_estimate}
 g_v(x) = C + \frac{1}{\deg v}\left(\delta_v (x) + \sum_{(v,w) \in E} g_w(x)\right),
\end{equation}
which is formally justified, since
\begin{align}\label{laplace_mean_value}
 \Delta g_v(x) &= \Delta C + \delta_v(x) - \frac{1}{\deg v}\sum_{(x,y) \in E} \delta_v(y)  + \frac{1}{\deg v} \sum_{(v,w) \in E} \Delta g_w(x)\\
 \notag &= 0 + \delta_v(x) - \frac{1}{\deg v} \sum_{(x,y) \in E} \delta_v(y) + \frac{1}{\deg v} \sum_{(v,w) \in E} \delta_w(x)\\
\notag &= \delta_v(x) - \frac{\deg(v,x)}{\deg v} + \frac{\deg (v,x)}{\deg v} \\\notag &= \delta_v(x).
\end{align}

Let $P$ be the transition kernel of random walk on $\sT$, and $P^n$ the transition kernel of $n$ steps of the random walk, thus $P^n(v,w)$ is the probability of transitioning from $v$ to $w$ in $n$ steps. Equation (\ref{iterative_estimate}) may be written
\begin{equation}
 g_v(x) = C + \frac{\delta_v(x)}{\deg v} + \sum_{w \in V} P^1(v,w) g_w(x).
\end{equation}
Iterating, for any $n \geq 1$, 
\begin{equation}
g_v(x) = C + \sum_{j=0}^n \frac{P^j(v,x)}{\deg x} + \sum_{w \in V} P^{n+1}(v,w) g_w(x).
\end{equation}
In dimension 2 it is common to regularize this by setting
\begin{equation}\label{2d_greens_function}
 g_v(x) = \sum_{n=0}^\infty \left(\frac{P^n(v,x)}{\deg x}  - \frac{P^n(v,v)}{\deg v}\right).
\end{equation}
In dimensions $d \geq 3$ it is customary to set $C = 0$ above, and 
\begin{equation}\label{gtr_2d_greens_function}
 g_v(x) = \sum_{n=0}^\infty \frac{P^n(v,x)}{\deg x}.
\end{equation}
Assuming the sums converge, which is justified shortly,
\begin{equation}\label{green_fn_P}
 \Delta g_v(x) = P^0(v,x) + \sum_{n=0}^\infty \left(P^{n+1}(v,x) - \sum_{(w,x) \in E} \frac{P^{n}(v,w)}{\deg w}\right)
\end{equation}
and each summand vanishes, while $P^0(v,x) = \delta_v(x)$.

For computations, an alternative description of the Green's function is more useful.  Recall that $\varrho$ is the measure on $\Lambda$ of random walk started from 0 and stopped at the first positive time $T_0$ at which it reaches $\Lambda$.  Let $\E[T_0] = \alpha > 0$.   

\begin{lemma}
 The measure $\varrho$ is symmetric, that is $\varrho(x) = \varrho(-x)$.
\end{lemma}
\begin{proof}
 Let $0 = v_0, v_1, v_2, ..., v_n = x$ be a path from 0 to $x$ such that $v_i \not \in \Lambda$ for $1 \leq i \leq n-1$.  Let $e_1, e_2, ..., e_n$ be edges with $e_i$ connecting $v_{i-1}$ and $v_i$.  The probability of following $e_1, e_2, ..., e_n$ in succession is $\prod_{j=0}^{n-1} \frac{1}{\deg v_j}$. The probability of following that path in reverse is $\prod_{j=1}^n \frac{1}{\deg v_j}$.  Since $\deg v_0 = \deg v_n$ by $\Lambda$-periodicity, running the path in reverse has the same probability.  This is also true of the path translated by $-x$.  Summing over all paths that lead to $x$ proves that $\varrho(x) \leq \varrho(-x)$. By symmetry, $\varrho(x) = \varrho(-x)$.
\end{proof}

\begin{lemma}
 In dimension 2, for $x \in \Lambda$,
 \begin{equation}
  g_0(x) = \sum_{n=0}^\infty \frac{P^n(0,x)}{\deg x} - \frac{P^n(0,0)}{\deg 0} = \sum_{n=0}^\infty \frac{\varrho^{*n}(x)}{\deg x} - \frac{\varrho^{*n}(0)}{\deg 0}
 \end{equation}
and both sums converge.  If the dimension is $\geq 3$, then
\begin{equation}
g_0(x) = \sum_{n=0}^\infty \frac{P^n(0,x)}{\deg x} = \sum_{n=0}^\infty \frac{\varrho^{*n}(x)}{\deg x}
\end{equation}
and both sums converge. Restricted to $\Lambda$, in dimension 2, $g_0(x) \ll 1 + \log (2 + \|x\|)$ and in dimension $n > 2$, $g_0(x) \ll \frac{1}{(1 +\|x\|)^{n-2}}$.
\end{lemma}
\begin{proof}
There is a positive probability that $\varrho^{*2}(0) >0$, since the measure $\varrho$ is symmetric. Let $\sigma^2$ be the covariance matrix. It follows that the local limit theorem, Theorem \ref{local_limit_theorem}  applies to $\varrho^{*2}$, see also \cite{LL10}.  This implies the following bounds on the density of $\varrho^{*n}(x)$, for any $A>0$,
\begin{align*}
 \varrho^{*n}(x) \ll \left\{\begin{array}{lll}\frac{e^{- \frac{\|\sigma^{-1} x\|^2}{n}}}{n^{\frac{d}{2}}} && n \geq \frac{\|x\|^2}{(\log (2+\|x\|))^2}\\ O_A\left((1+\|x\|)^{-A}\right) && n < \frac{\|x\|^2}{(\log (2+\|x\|))^2} \end{array} \right..
\end{align*}

When $d \geq 3$ this justifies the convergence of the $\varrho$ sums for $d \geq 3$ and also the bound on $g_0(x)$ as $x \to \infty$, since the sum is concentrated around $n$ of order $\|x\|^2$.  

To treat the case $d = 2$, notice that in defining the stopping time related to the measure $\varrho$, there is a positive probability that $Y_{0,2} = 0$, so that if $\varrho$ has periodicity, the only possible periodicity is 2.  Again by the local limit theorem on $\bR^2$, either $\varrho^{*n}(x) - \varrho^{*n}(0) \ll n^{-\frac{3}{2}}$ or $\varrho^{*n}(x) - \varrho^{*(n+1)}(0) \ll n^{-\frac{3}{2}}$ as $n \to \infty$, which again justifies the convergence.  The bound on $g_0$ can be proved by noting that
 $P^n(0,0), P^n(0,x) \ll \frac{1}{n}$ so that 
 \begin{equation}
\sum_{n \ll \|x\|^2} \frac{P^n(0,x)}{\deg x} - \frac{P^n(0,0)}{\deg 0} \ll \log(2 + \|x\|^2)
 \end{equation}
 while, since $\deg 0 = \deg x$, for some $c > 0$,
 \begin{align*}
  &\left|\sum_{n \gg \|x\|^2} \frac{P^n(0,x)}{\deg x} - \frac{P^n(0,0)}{\deg 0}\right| \\&\ll \frac{1}{\deg 0}\sum_{n \gg \|x\|^2} \frac{\left|e^{-\frac{\|\sigma^{-1}x\|^2}{n}}-1\right|}{n} + O\left(n^{-\frac{3}{2}} \right)= O(1).
 \end{align*}
 The last line uses the leading order term of the local limit theorem, which is proportional to the Gaussian density at the point in this range.

To show the equality of the $P$ and $\varrho$ sums, given the random walk $Y_{0,n}$ let $S_0 = 0< S_1<S_2< ...$ be the return times to $\Lambda$.  The distribution of $Y_{0,S_n}$ is the same as that of $\varrho^{*n}$.  Let $T(n)$ be the least $j$ such that $S_j > n$.  Since for $x \in \Lambda$,
 \begin{align*}
  \sum_{j=0}^n P^j(0, x) &= \E\left[\sum_{j = 0}^n \one(Y_{0,S_j} = x \wedge S_j \leq n) \right]\\
  &= \sum_{j=0}^{\frac{n}{\alpha}} \varrho^{*j}(x) + O\left(\sum_{\left| j - \frac{n}{\alpha}\right| \leq n^{\frac{3}{4}}} \varrho^{*j}(x) \right) \\ &+ O\left(n \Prob\left(\left|T(n) - \frac{n}{\alpha}\right|> n^{\frac{3}{4}}\right)\right).
 \end{align*}
The first error term tends to 0 as $n \to \infty$ by the local limit theorem for $\varrho$, since $\varrho^{*j}(x) \ll j^{-\frac{d}{2}}$.  To bound the second error term, write $S_j = T_1 +  T_2 + \cdots + T_j$, where $T_1, ..., T_j$ are independent copies of the random variable $T_0$ which is the first return time to the lattice $\Lambda$.  These variables have exponentially decaying tails, and hence the variant of Chernoff's inequality in Lemma \ref{Chernoff_variant_lemma}  with $\lambda$ of order $n^{\frac{1}{4}}$ implies that, for some $c > 0$,  
\begin{equation}
 \Prob\left(\left|T(n) - \frac{n}{\alpha}\right|> n^{\frac{3}{4}}\right) \ll n e^{-c n^{\frac{3}{8}}}.
\end{equation}
This shows that the second error term tends to 0 as $n \to \infty$.  Since both error terms tend to 0 as $n \to \infty$, it is possible to replace the $P$ sums with the $\varrho$ sums.
\end{proof}

It is now possible to show that equations (\ref{2d_greens_function}) and (\ref{gtr_2d_greens_function}) converge and define Green's functions.
\begin{lemma}\label{green_fn_growth_lemma}
 In dimension 2,
 \begin{equation}
  g_0(x) = \sum_{n=0}^\infty \frac{P^n(0,x)}{\deg x} - \frac{P^n(0,0)}{\deg 0}
 \end{equation}
and, in dimension at least 3,
\begin{equation}
 g_0(x) = \sum_{n=0}^\infty \frac{P^n(0,x)}{\deg x}
\end{equation}
converge for all $x \in \sT$ and are Green's functions.  The functions satisfy the bounds, in dimension 2,
\begin{equation}
 g_0(x) \ll 1 + \log(2 + d(0,x))
\end{equation}
and in dimensions $d \geq 3$, 
\begin{equation}
 g_0(x) \ll  \frac{1}{(1 + d(0,x))^{d-2}}.
\end{equation}

\end{lemma}
\begin{proof}
Assume that $x \not \in \Lambda$. Let $Y_{x,0}=x,$ $Y_{x,i+1} = P \cdot Y_{x,i}$ be random walk on $\sT$ started from $x$. Since, for $n \geq 1$,
 \begin{equation}
  \frac{P^n(0,x)}{\deg x} = \frac{1}{\deg x} \sum_{(w,x) \in E} \frac{P^{n-1}(0,w)}{\deg w} = \E\left[\frac{P^{n-1}(0,Y_{x,1})}{\deg Y_{x,1}} \right]
 \end{equation}
it follows that for the finite stopping time $T_n$, which is the minimum of $n$ and the first time that $Y$ reaches $\Lambda$,
\begin{equation}
 \frac{P^n(0,x)}{\deg x} = \E\left[\frac{P^{n-T_n}(0, Y_{x,T_n})}{\deg Y_{x,T_n}} \right].
\end{equation}

Since $T_n$ has exponentially decaying tail, two exceptional cases may be excluded.  
\begin{itemize}
 \item (Case 1) If $n \leq 2+ d(0,x)^2$, by choosing the constants appropriately, the event $T_n \gg \log (2+d(0,x))$ has probability, for any $A>0$,  $\ll_A \frac{1}{ d(0,x)^A}$. 
 \item (Case 2) If $n >  2+d(0,x)^2$ the event $T_n \gg \log n$ has probability, for any $A>0$,  $\ll_A n^{-A}$. 
\end{itemize}
\noindent
Choosing an $A$ sufficiently large, the sum in $n$ of the probabilities of Case 1 or 2 is $O(1)$.

Let $S_n$ be the event that neither Case 1 nor 2 holds.  First consider the case that the dimension $d$ is at least 3.  Conditioned on $S_n$, $Y_{x,n} \in \Lambda$, and $P^{n-T_n}(0,Y_{x,n})$ satisfies,  uniformly in $Y_{x,n}$, for some $c>0$, and all $A>0$,
\begin{equation}
 P^{n-T_n}(0,Y_{x,n}) \ll \left\{\begin{array}{lll} O_A \left(\frac{1}{d(0,x)^A} \right) && n \leq \frac{d(0,x)^2}{(\log (2 + d(0,x)))^2}\\
\frac{e^{-\frac{cd(0,x)^2}{n}}}{n^{\frac{d}{2}}} && n>\frac{d(0,x)^2}{(\log (2 + d(0,x)))^2}  
\end{array}\right..
\end{equation}
The first bound holds, since conditioned on $S_n$, $d(0, Y_{x,n}) = d(0,x) + O(\log(2 + d(0,x)))$ and from the local limit theorem for $\varrho^{*n}$ on the lattice $\Lambda$.  The second bound holds, since conditioned on $S_n$, $d(0, Y_{x,n}) = d(0,x) + O(\log n)$, so that the claim follows again from the local limit theorem for $\varrho^{*n}$.
Since 
\begin{equation}
\sum_{n> \frac{d(0,x)^2}{(\log (2 + d(0,x)))^2}} \frac{e^{-\frac{cd(0,x)^2}{n}}}{n^{\frac{d}{2}}} \ll \frac{1}{(1 + d(0,x))^{d-2}} 
\end{equation}
and the contribution of smaller $n$ is negligible by taking $A$ sufficiently large, the claimed bound holds.

In the case of dimension 2, conditioned on $S_n$, either $|P^{n-T_n}(0, Y_{x,T_n}) - P^n(0,0)|$ or $|P^{n-T_n}(0, Y_{x,T_n}) - P^{n+1}(0,0)|$ satisfies for some $c>0$,
\begin{equation}
 \ll \left\{\begin{array}{lll} \frac{1}{n}  && n \leq \frac{d(0,x)^2}{(\log (2 + d(0,x)))^2}\\
\frac{1-e^{-\frac{cd(0,x)^2}{n}}}{n}+ O\left(n^{-\frac{3}{2}+\epsilon} \right) && n > \frac{d(0,x)^2}{(\log (2 + d(0,x)))^2}  
\end{array}\right..
\end{equation}
The proof of these bounds is again by the local limit theorem of $\varrho^{*n}$.  The reason for considering either the difference of steps $n$ or $n+1$ is to handle the case of 2-periodicity.

Since restricted to $S_n$,  $\deg 0 = \deg Y_{x,T_n}$,
\begin{align*}
 \sum_{n=0}^\infty \frac{P^n(0,x)}{\deg x} - \frac{P^n(0,0)}{\deg 0} &= \sum_{n=0}^\infty  \E \left[\frac{P^{n-T_n}(0, Y_{x,T_n})}{\deg Y_{x,T_n}} - \frac{P^n(0,0)}{\deg 0} \right]\\
 &= \sum_{n=0}^\infty \frac{1}{\deg 0} \E\left[(P^{n-T_n}(0, Y_{x,T_n}) - P^n(0,0))\one_{S_n} \right] \\&+ O\left(\sum_n  \Prob(S_n^c) \right)\\
 &\ll \log (2 + d(0,x)).
\end{align*}
This proves the convergence and obtains the claimed bound.  

Since the sums converge, $g_0$ defines a Green's function by (\ref{green_fn_P}).
\end{proof}

When $v \not \in \Lambda$, it follows from the Laplace equation that
\begin{equation}\label{mean_value_property}
 g_0(v) = \frac{1}{\deg v} \sum_{(v,w) \in E} g_0(w).
\end{equation}
\begin{lemma}\label{off_lattice_green_lemma}
Given a Green's function $g_0$ started from zero  on $\sT$, satisfying,  for $x \in \sT$, $g_0(x) \ll \log (2 + d(0,x))$, the Green's function can be recovered from its values on $\Lambda$ by, for $v \in \sT$,  $g_0(v) = \E[g_0(Y_{v, T_v})]$.
\end{lemma}
\begin{proof}
 By iterating the mean value property (\ref{mean_value_property}), for the stopping time $T_n =  \min (T_v, n)$,
 \begin{equation}
  g_0(v) = \E[g_0(Y_{v, T_n})].
 \end{equation}
 Meanwhile, $\E[g_0(Y_{v, T_v}) \one(T_v \leq n)]$ converges as $n \to\infty$, since $g_0$ grows at most logarithmically on the lattice $\Lambda$ and $T_v$ has exponentially decaying tails.  Both limits are equal to $\E[g_0(Y_{v, T_v})]$ by the growth assumption on $g_0$.
\end{proof}

Finally, to obtain the Green's function in general, for $v \not \in \Lambda$  iterate the identity
\begin{equation}\label{laplace_again}
 g_v(x) = \frac{\delta_v(x)}{\deg v} + \frac{1}{\deg v} \sum_{(v,w) \in E} g_w(x).
\end{equation}
\begin{lemma}
 For $v \not \in \Lambda$, a Green's function $g_v(x)$ is given by
 \begin{equation}
  g_v(x) = \frac{1}{\deg x} \E\left[\sum_{j=0}^{T_v-1} \one(Y_{v,j}=x) \right] + \E\left[g_{Y_{v, T_v}}(x)\right].
 \end{equation}
In particular, for $x \in \Lambda$, $g_v(x) = g_0*\varrho_v (x)$.
\end{lemma}
\begin{proof}Convergence of the two expectations is guaranteed by the exponential decay of the tail of $T_v$ and by the growth bound of the Green's function.     From the definition of the Green's function on the lattice $\Lambda$, $\Delta \E\left[g_{Y_{v, T_v}}(x) \right] = \Prob(Y_{v, T_v} = x)$.  Meanwhile,
\begin{align}
& \Delta\left(\frac{1}{\deg x} \E\left[\sum_{j=0}^{T_v-1}\one(Y_{v,j}=x) \right] \right) \\&\notag= \E\left[ \sum_{j=0}^{T_v-1} \one(Y_{v,j}=x) \right] -  \sum_{(x,y) \in E}\frac{1}{\deg y} \E\left[ \sum_{j=0}^{T_v-1} \one(Y_{v,j}=y)\right]\\
&=\notag \E\left[\sum_{j=0}^{T_v-1} \one(Y_{v,j}=x) \right]- \E\left[\sum_{j=1}^{T_v} \one(Y_{v,j}=x) \right]\\
&= \notag \delta_{v}(x) - \Prob(Y_{v, T_v}=x).
\end{align}
Adding these two contributions completes the proof of the first claim.

To prove the second, note that for $x \in \Lambda$, $ \E\left[\sum_{j=0}^{T_v-1} \one(Y_{v,j}=x) \right] = 0$, and thus,  the claim follows since $Y_{v, T_v}$ has the distribution of $\varrho_v$. 
\end{proof}

\begin{lemma}
 For any $\eta \in C^0(\sT)$, for all $x \in \Lambda$,
 \begin{equation}
  g_\eta(x) = g_0 * \varrho_\eta(x).
 \end{equation}

\end{lemma}
\begin{proof}
Since $\varrho_v$ has been defined to be a point mass at $v$ when $v \in \Lambda$, the previous Lemma demonstrates that for all $v \in \sT$ and all $x \in \Lambda$, $g_v(x) = g_0 * \varrho_v(x)$.  It thus follows that if $\eta$ is a function of bounded support on $\sT$, then, for $x \in \sT$,
\begin{equation}
 g_\eta(x) = g_0 * \varrho_\eta(x).
\end{equation}
\end{proof}

The following theorem demonstrates that the above methods may be used to obtain an explicit formula for the Green's function of a periodic tiling,  which is useful in practical calculations.
Let $\sT \subset \bR^d$ be a tiling with period lattice $\Lambda$ identified with  $\zed^d$ after a linear map and suppose $0 \in \sT$.  Split $\bR^d$ into unit cubes by identifying $(y_1, ..., y_d)$ with $(\lfloor y_1 \rfloor, ..., \lfloor y_d \rfloor)$. Let $z_i = e(-x_i)$, $i = 1, 2, ..., d$, and assign each directed edge $e = (u,v)$ of $\sT$ a weight $w_e$ which is the product of all $z_i$ such that the floor of the $i$th coordinate of $v$ is greater than the floor of the $i$th coordinate of $u$, divided by the product of all $z_j$ such that the opposite is true.  Choose a system of representatives $0 = v_0, v_1, ..., v_m$ for $\sT/\Lambda$, and let $Q$ be the $m \times m$ matrix with 
 \begin{equation}
  Q(i,j) = \sum_{\substack{v \equiv v_j \bmod \Lambda\\ e = (v_i,v) \in E}}\frac{w_e}{\deg v_i} .
 \end{equation}
Thus, when $z \equiv 1$, $Q$ is the transition matrix of simple random walk on $\sT/\Lambda$. Let $c_0$ be the column of $Q$ corresponding to $v_0$, and $r_0$ the row corresponding to 0, and let $Q'$ be the $(m-1)\times (m-1)$ minor obtained by deleting $c_0$ and $r_0$, and similarly let $c_0', r_0'$ be obtained by deleting the $(0,0)$ entry.
\begin{theorem}\label{Greens_fn_theorem}

The characteristic function of $\varrho$ is

\begin{align*}
 \hat{\varrho}(x) &= \sum_{\lambda \in \zed^d} \varrho(\lambda) e(-x \cdot \lambda)\\ &=Q_{0,0}(z) + r_0(z) (I-Q'(z))^{-1}c_0(z)
\end{align*}
and the Fourier transform of $g_0$ restricted to $\Lambda$ is given by 
\begin{equation}
(\deg 0) \hat{g}_0(x) = \frac{1}{1 - (Q_{0,0}(z) + r_0'(z) (I-Q'(z))^{-1}c'_0(z))}.
\end{equation}

\end{theorem}
\begin{proof}
 The stopped random walk either transitions directly from 0 to another point in $\Lambda$ with partial characteristic function given by $Q_{0,0}(z)$, or transitions from 0 to another state, makes $n \geq 0$ moves between states not in $\Lambda$ and then returns to $\Lambda$. 
 Given a probability measure $\nu$ on $\sT$, define
 \begin{align*}
   &\hat{\nu}(x) =  [\hat{\nu}_0, ..., \hat{\nu}_m]\\
   &\hat{\nu}_j = \sum_{y \equiv v_j \bmod \Lambda} \nu(y) z_1^{\lfloor y_1\rfloor} \cdots z_d^{\lfloor y_d \rfloor}.
 \end{align*}
 By the periodicity, the change in $(\lfloor y_1\rfloor, ..., \lfloor y_d\rfloor)$ in each transition, and the corresponding chances of a transition depend only on the current state $v \bmod \Lambda$, and the changes are additive, hence, given a probability $\nu$ on $\sT$,  with transition in $\sT$ given by $P\cdot \nu$, the mixture after one transition  satisfies $\widehat{P \cdot \nu} = \hat{\nu} Q$. Conditioning on $n$, the number of steps before a transition back into $\Lambda$, 
 \begin{equation}\hat{\varrho}(x) = Q_{0,0}(z) + r_0'(z) (I + Q'(z) + Q'(z)^2 + \cdots ) c_0'(z).
  \end{equation}
 The justification of the geometric series formula $\sum_{n=0}^\infty (Q'(z))^n = (I-Q'(z))^{-1}$ is that pointwise, $Q'(z)^n$ is bounded by $Q'(1)^n$, which tends to 0 with $n$, since the random walk has a positive probability of returning to $\Lambda$ in boundedly many steps from any state. 
  
Since, restricted to $\Lambda$, $g_0(x) = \frac{1}{\deg 0} \left(\sum_{n=0}^\infty \varrho^{*n}(x) -\varrho^{*n}(0) \right)$ in dimension 2, or in dimension at least 3, $g_0(x) = \frac{1}{\deg 0} \sum_{n=0}^\infty \varrho^{*n}(x)$, the Fourier transform of $g_0$ is given by, for $x \neq 0$,
\begin{equation}
 (\deg 0)\hat{g}_0(x) = \sum_{n=0}^\infty \hat{\varrho}(x)^n, 
\end{equation}
with the caveat that in dimension 2, the Green's function can be considered as dual to functions of bounded support and sum 0. The formula for the Green's function's characteristic function follows from applying the geometric series formula to the characteristic function of $\varrho$.
\end{proof}

\begin{proof}[Proof of Theorem \ref{ell_2_theorem}]
Identify $\Lambda$ with $\zed^d$.  We show that the conditions are necessary and sufficient for $g_\eta$ to be in $\ell^2(\Lambda)$.  This suffices for the theorem, since the condition of being in $C^j(\sT)$ is invariant under translation so that the same conditions are necessary and sufficient for $g_\eta$ to be in $\ell^2(t + \Lambda)$ for any $t \in \sT/\Lambda$.

Since on $\Lambda$, $g_\eta =  g_0 * \varrho_\eta$, $g_\eta$ has Fourier transform
\begin{equation}
 \hat{g}_\eta(\xi) = \frac{1}{\deg 0} \frac{\hat{\varrho}_\eta(\xi)}{1-\hat{\varrho}(\xi)}. 
\end{equation}
Since $\supp \varrho$ generates $\Lambda$, $|\hat{\varrho}(\xi)| \neq 1$ if $\xi \neq 0$, and hence $\frac{1}{1-\hat{\varrho}(\xi)}$ is bounded outside neighborhoods of 0.  Thus, by Parseval, it suffices to consider the behavior on a neighborhood of 0.  By Taylor expansion, using that $\varrho$ has exponentially decaying tails,
\[
 1-\hat{\varrho}(\xi) = \sum_n \varrho(n) (1-e^{-2\pi i n \cdot \xi}) = 2\pi^2 \xi^t \sigma^2 \xi + O(\|\xi\|^3)
\]
where we have used that the first moment of $\varrho$ vanishes, since $\varrho$ is symmetric.

By Parseval, for $\delta > 0$, 
\begin{align*}
 \|g_\eta\|_2^2& = \frac{1}{(\deg 0)^2} \int_{\bR^d/\zed^d} \frac{\left|\hat{\varrho}_\eta(\xi)\right|^2}{|1 - \hat{\varrho}(\xi)|^2} d\xi\\ & = O(1) + \frac{1}{(\deg 0)^2}\int_{\|\xi\|< \delta} \frac{\left|\hat{\varrho}_\eta(\xi)\right|^2}{4 \pi^4 (\xi^t \sigma^2 \xi)^2 + O(\|\xi\|^5)}d\xi.
\end{align*}
Since $\varrho_\eta$ has exponentially decaying tails, it follows that $\hat{\varrho}_\eta(\xi)$ is equal to its Taylor expansion at 0, which is necessarily bounded.  The constant term is the total mass, the linear term is given by the first moment.  Switching to polar coordinates gains a factor of $r^{d-1}$ against the factor of $\asymp r^{-4}$ from the definite quadratic form in the denominator.  Thus, in dimension 2, it is necessary and sufficient for $g_\eta$ to be in $\ell^2$ that $\hat{\varrho}_\eta$ vanish to degree 2, in dimension 3,4 that it vanish to degree 1, and in higher dimensions, that it is bounded.  This gives the condition claimed.

To prove the characterisation of $\sH^2(\sT)$, let $\xi \in \sH^2(\sT)$ and let $\nu = \Delta \xi$.  Since $\Delta: \ell^2(\sT) \to \ell^2(\sT)$ is bounded, $\|\nu\|_2 < \infty$ and hence $\nu$ has finite support.  It follows that $g_\nu$ is well-defined as a function on $\sT$, and $\Delta(\xi - g_\nu) = 0$.  If $(\xi - g_\nu)(x) \to 0$ as $d(0,x) \to \infty$, then by the maximum modulus principle $\xi - g_\nu = 0$. This applies unless $d = 2$ and $\nu \not \in C^1(\sT)$.  To rule out the remaining case, let $y \in \Lambda$ and let $\tau_y$ denote translation by $y$.  Since $\nu - \tau_y \nu$ is at least $C^1(\sT)$, $g*(\nu - \tau_y\nu)$ tends to 0 at infinity, and hence, for any $y$, $\xi - \tau_y \xi = g*(\nu - \tau_y \nu)$.  Since $\nu \not \in C^1(\sT)$, $g*\nu$ is unbounded, and hence $g*(\nu - \tau_y \nu)$ can take arbitrarily large values.  But $\xi - \tau_y \xi$ is bounded, a contradiction.  Hence, $\xi \in \sH^2(\sT)$ implies $\xi = g* (\Delta \xi)$ and $\Delta \xi \in C^j(\sT)$.
\end{proof}
\subsection{The Green's function of periodic and reflected tilings}
Let $\sT \subset \bR^d$ be a tiling, which is periodic with period $\Lambda$. A mean zero Green's function started from 0 may be defined on $\sT/m\Lambda$ as follows.  On $\Lambda/m\Lambda$ define $\varrho_{\bT_m}(x) = \varrho(x + m\Lambda)$ and 
\begin{equation}
 g_{0, \bT_m}(x) = \frac{1}{\deg(0)} \sum_{n=0}^\infty \left(\varrho_{\bT_m}^{*n}(x) - \frac{1}{m^d}\right).
\end{equation}
This may be extended to all of $\sT/m\Lambda$ by the formula, 
\begin{equation}
 g_{0, \bT_m}(v) = \E[g_{0, \bT_m}(Y_{v,T_v})].
\end{equation}
This is still mean 0, since for any $v$,
\begin{equation}
  \sum_{\lambda \in \Lambda/m\Lambda} g_{0, \bT_m}(v + \lambda)  = \sum_{\lambda \in \Lambda/m\Lambda} \E[g_{0, \bT_m}(Y_{v, T_v} + \lambda)] = 0
\end{equation}
by the $\Lambda$ translation invariance.

As above, the Green's function started from an arbitrary point $v$ is obtained by
\begin{equation}
 g_{v, \bT_m}(x) = -c_v + \frac{1}{\deg x} \E\left[ \sum_{j=0}^{T_v-1} \one\left(Y_{v,j} = x\right)\right] + \E\left[g_{Y_{v, T_v}, \bT_m}(x)\right],
\end{equation}
where the constant $c_v$ is chosen to make the Green's function mean 0. Since $T_v$ has exponentially decaying tail and $g_{Y_{v, T_v}}$ is mean 0, $c_v = O\left(\frac{1}{m^d} \right).$

\begin{lemma}
 The Green's function satisfies \begin{equation}\Delta g_v(x) = \delta_v(x) - \frac{1}{m^d} \delta(x \in \Lambda).\end{equation}
\end{lemma}
\begin{proof}
If $v \not \in \Lambda$,
\begin{equation}
 \Delta g_{0, \bT_m}(v) = (\deg v) g_{0, \bT_m}(v) - \sum_{(v,w) \in E} g_{0, \bT_m}(w) = 0
\end{equation}
since the sum over $w$ corresponds to taking one step in the random walk $Y_v$.

When $x \in \Lambda$, by splitting off the $n = 0$ term in the sum defining $g_{0, \bT_m}$,
\begin{align*}
 &\Delta g_{0, \bT_m}(x) = \delta_0(x) - \frac{1}{m^d} + \sum_{n=1}^\infty \left(\rho_{\bT_m}^{*n}(x) - \frac{1}{m^d} \right) - \sum_{(x,y) \in E} g_{0, \bT_m}(y)\\
 &=\delta_0(x) - \frac{1}{m^d} + \sum_{n=1}^\infty \left(\rho_{\bT_m}^{*n}(x) - \frac{1}{m^d} \right) \\&  - \deg(0) \E\left[ \frac{1}{\deg x} \left(\sum_{(x,y) \in E, y \not \in \Lambda} g_{0, \bT_m}(Y_{y, T_y})+ \sum_{(x,y) \in E, y \not \in \Lambda} g_{0, \bT_m}(y) \right)\right]\\
 &=\delta_0(x) - \frac{1}{m^d} + \sum_{n=1}^\infty \left(\rho_{\bT_m}^{*n}(x) - \frac{1}{m^d} \right)\\ &  - \deg(0) \E\left[g_{0, \bT_m}(Y_{x, T_x})\right].
\end{align*}
Since $Y_{x, T_x}$ has the distribution of $\delta_x * \varrho_{\bT_m}$, and since $\varrho_{\bT_m}$ is symmetric the sum and the expectation cancel, leaving $\delta_0(x) - \frac{1}{m^d}$.

The values of $g_{x, \bT_m}$ for $x \in \Lambda$ are obtained by translation invariance.  

To check the property at $v \not \in \Lambda$,
\begin{align*}
 \Delta g_{v, \bT_m}(x) &= -\Delta c_v + \E\left[ \sum_{j=0}^{T_v-1} \one(Y_{v,j} = x)\right] \\&-  \sum_{(x,y) \in E} \frac{1}{\deg y}\E\left[ \sum_{j=0}^{T_v-1} \one(Y_{v,j} = y)\right] + \E[\Delta g_{Y_{v, T_v}, \bT_m}(x)]\\
 &= \E\left[ \sum_{j=0}^{T_v-1} \one(Y_{v,j} = x)\right] - \E\left[ \sum_{j=1}^{T_v} \one(Y_{v,j} = x)\right]\\& + \Prob(Y_{v, T_v} = x) - \frac{1}{m^d} \delta(x \in \Lambda)\\
 &= \delta_v(x) - \frac{1}{m^d} \delta(x\in \Lambda).
\end{align*}

\end{proof}

It follows that $g_{\bT_m}$ has the property that \begin{equation}\Delta(g_{v_1, \bT_m}(x) - g_{v_2, \bT_m}(x)) = \delta_{v_1}(x) - \delta_{v_2}(x).\end{equation}
Given an integer valued function $\eta$ on $\sT/m\Lambda$, define
\begin{equation}
 g_{\eta, \bT_m} = g_{\bT_m}* \eta(x) = \sum_{v \in \sT/m\Lambda} \eta(v) g_{v, \bT_m}(x).
\end{equation}
Abusing notation, given $\eta \in C^0(\sT)$, define $g_{\eta, \bT_m} = g_{\bT_m} * \eta_{\bT_m}$.

Given a tiling $\sT$ with reflection symmetry in family of hyperplanes
\begin{equation}
 \sF = \{ n v_i + H_i: n \in \zed, H_i = \{x: \langle x, v_i\rangle = 0\}\}
\end{equation}
which its edges do not cross,
the tiling is periodic with period lattice $\Lambda$ generated by $\{2v_i: i = 1, 2, ..., d\}$.  A Green's function for $\sT$ with reflection symmetry in $m\sF$ is obtained by letting $\tilde{g}$ be a Green's function for $\sT/m\Lambda$, and then imposing reflection anti-symmetry by forming an alternating sum over reflections in a bounded number of hyperplanes.

\subsection{Derivative estimates}
The following results are needed regarding discrete derivatives of the Green's function on $\sT/m\Lambda$, and are proved in Appendix \ref{Green_fn_appendix}.

\begin{lemma}\label{deriv_bound_lemma}
Let $\sT$ be a tiling of $\bR^d$ which is $\Lambda \cong \zed^d$ periodic.
 Let $\eta$ be of class $C^j(\sT)$ for some $0 \leq j \leq 2$.  Let $D^{\ua}$ be a discrete differential operator on the lattice $\Lambda$ and assume that $|\ua| + j + d - 2 > 0$.  For $x \in \Lambda$, for $m \geq 1$,
 \begin{equation}
  D^{\ua} g_{\eta, \bT_m} (x) \ll \frac{1}{1 + \|x\|_{(\zed/m\zed)^d}^{|\ua| + j + d -2}}.
 \end{equation}

\end{lemma}

Note that, although Lemma \ref{deriv_bound_lemma} applies to $x \in \Lambda$, by Lemma \ref{C_2_translation_invariance_lemma} the property of being $C^j$ is invariant under translating $\sT$, and hence the same estimate holds for arbitrary $x \in \sT$ up to changing the norm by $O(1)$.

\begin{lemma}\label{G_eta_eval}
 Let $\sT$ be a tiling of $\bR^d$ with period lattice $\Lambda$ identified with $\zed^d$ via a choice of basis. Set $\sigma^2 = \Cov(\varrho_{\frac{1}{2}})$.  Let $\eta$ be of class $C^1(\sT)$, and let $\varrho_\eta$ be the signed measure on $\Lambda$ obtained by stopping simple random walk on $\sT$ started from $\eta$ when it reaches $\Lambda.$ Let $\varrho_\eta$ have  mean $v$. For $m \geq 1$, for $n \in \Lambda$, $1 \leq \|n\| \ll \left(\frac{m^2}{\log m}\right)^{\frac{d-1}{2d}}$,
 \begin{equation}
  g_{\eta, \bT_m}(n) = \frac{\Gamma\left(\frac{d}{2} \right) v^t \sigma^{-2}n}{ \deg(0)\pi^{\frac{d}{2}}\|\sigma^{-1}n\|^d \det \sigma} + O\left(\frac{1}{\|\sigma^{-1}n\|^d} \right).
 \end{equation}
If $d \geq 3$ and $\eta \not \in C^1(\sT)$ has total mass $C$,
\begin{equation}
 g_{\eta, \bT_m}(n) = \frac{C \Gamma\left(\frac{d}{2} -1\right) }{2 \deg(0)\pi^{\frac{d}{2}}\|\sigma^{-1}n\|^{d-2} \det \sigma} + O\left(\frac{1}{\|\sigma^{-1}n\|^{d-1}} \right).
\end{equation}

\end{lemma}

As in the previous lemma, $g_\eta$ may be recovered on all of $\sT/m\Lambda$ by translating $\sT$ to translate the period lattice.

\begin{lemma}\label{Green_fn_asymptotic}
 Let $d \geq 2$ and let $\ua \in \bN^d$.  If $|\ua| + \frac{d}{2} > 2$ then  for each fixed $n,v \in \sT$, 
 \begin{equation}
  D^{\ua} g_{v, \bT_m}(n) \to D^{\ua} g_{v}(n)
 \end{equation}
as $m \to \infty$.
\end{lemma}

\section{The sandpile group and dual group}
The reader is referred to Section 2 of \cite{JLP15}, which gives a clear discussion of the sandpile group of a simple connected finite graph.  The arguments given there go through with only slight  changes to handle  graphs with multiple edges, which are used to handle the case of a sink at the boundary. 

Let $G = (V,E)$ be a graph which is connected, with possibly multiple edges but no loops.  Let $s \in V$ be the sink.  A \emph{sandpile} on $G$ is a map $\sigma: V \setminus \{s\} \to \zed_{\geq 0}$. The sandpile is \emph{stable} if $\sigma(v) < \deg(v)$ for all $v \in V \setminus\{s\}$.  If $\sigma$ is unstable, so that for some $v \in V \setminus \{s\}$, $\sigma(v) \geq \deg(v)$, the sandpile at  $v$ can topple to $\sigma'$ which has
\begin{equation*}
 \sigma'(v) = \sigma(v) - \deg(v),
\end{equation*}
for $w \in V\setminus\{s\}$ such that $(v,w) \in E$, 
\begin{equation*}
 \sigma'(w) = \sigma(w) + \deg(v,w)
\end{equation*}
where $\deg(v,w)$ is the number of edges between $v$ and $w$ in $E$,
and
\begin{equation*}
 \sigma'(w) = \sigma(w)
\end{equation*}
otherwise.

Topplings commute, and a vertex's height does not decrease unless it topples, hence given a sandpile $\sigma$ there is a unique stable sandpile $\sigma^o$ which can be obtained from $\sigma$ by repeated toppling. Let
\begin{equation}
 \sS(G) = \{\sigma: V \setminus\{s\} \to \zed_{\geq 0}, \sigma \leq \deg-1\}.
\end{equation}
The set $\sS(G)$ becomes an additive monoid under the law $\sigma \oplus \eta(v) = (\sigma + \eta)^o(v),$ in which the heights are added and then the sandpile is stabilized.

Sandpile dynamics on $\sS(G)$ are given by letting $\mu$ be the probability measure
\begin{equation}
 \mu = \frac{1}{|V|} \left(\delta_{\id} + \sum_{v \in V \setminus\{s\}} \delta_{v} \right)
\end{equation}
in which $\id$ is the identity element of the sandpile group, and $\delta_v$ is the Kronecker delta function at $v$.
Given an initial probability distribution $\nu$ on $\sS(G)$, the distribution at step $n$ of the dynamics is $\mu^{*n}*\nu$ where $\mu^{*n}$ is the $n$-fold repeated convolution.

Since the full state $\sigma_{\full}(v) = \deg(v)-1$ has a positive probability of being reached from any given state in a bounded number of steps, $\sigma_{\full}$ is recurrent for the dynamics, and hence the recurrent states are those reachable from $\sigma_{\full}$.  Let $\Delta'$ denote the reduced graph Laplacian, which is obtained from the graph Laplacian 
\begin{equation}
 \Delta f(v) = \sum_{(v,w) \in E} f(v)-f(w)
\end{equation}
by omitting the row and column corresponding to the sink.
The recurrent states form an abelian group $\sG(G) \cong \zed^{V\setminus \{s\}}/\Delta' \zed^{V\setminus\{s\}}$, see \cite{JLP15} for a proof. Since $\Delta'$ is a symmetric matrix, the dual lattice to $\Delta' \zed^{V \setminus \{s\}}$ is $(\Delta')^{-1} \zed^{V \setminus \{s\}}$ and hence the dual group is isomorphic to
\begin{equation}
 \hat{\sG}(G) \cong (\Delta')^{-1} \zed^{V \setminus \{s\}}/\zed^{V \setminus \{s\}}.
\end{equation}
Given $\xi \in \hat{\sG}$ and $g \in \sG$, viewed as functions on $V \setminus \{s\}$, the pairing is $\xi(g) = \xi \cdot g \in \bR/\zed$.  

In this article, attention is limited to the random walk $\mu^{*n}$ restricted to the group $\sG$ of recurrent states.  This is the long term behavior, and in any case, in \cite{HJL17} it is shown that on the torus $(\zed/m\zed)^2$, the random walk started from any stable state is absorbed into $\sG(G)$ with probability $1-o(1)$ in a lower order number of steps than the mixing time; the proof given there could be adapted to this situation as well.

Since the random walk considered is a random walk on an abelian group, in terms of the mixing behavior there is no loss in assuming that the walk is started at the identity.  Also, the transition kernel is diagonalized by the Fourier transform, that is, the characters, for $\xi \in \hat{\sG}$, $\chi_\xi(g) = e^{2\pi i \xi(g)}$ are eigenfunctions for the transition kernel, and the eigenvalues are the Fourier coefficients
\begin{equation}
 \hat{\mu}(\xi) = \frac{1}{|V|} \left(1 + \sum_{v \in V \setminus \{s\}} e(\xi_v) \right).
\end{equation}
Since the Fourier transform has the usual property of carrying convolution to pointwise multiplication, Cauchy-Schwarz and Plancherel give the following lemma, see \cite{D88}.
\begin{lemma}[Upper bound lemma]
 Let $\bU_{\sG}$ denote the uniform measure on the sandpile group $\sG(G)$.  For $n \geq 1$,
 \begin{equation}\left\|\mu^{*n} - \bU_{\sG}\right\|_{\TV(\sG)} \leq \frac{1}{2} \left\|\mu^{*n} - \bU_{\sG}\right\|_2 = \frac{1}{2} \left(\sum_{\xi \in \hat{\sG}\setminus \{0\}} \left|\hat{\mu}(\xi)\right|^{2n}\right)^{\frac{1}{2}}.
 \end{equation}

\end{lemma}

Several further representations of the dual group $\hat{\sG}(G)$ are useful.
\begin{lemma}
 The group $\hat{\sG}$ may be identified with the restriction to $V \setminus\{s\}$ of functions $\xi: V \to \bR/\zed$ such that $\xi(s) = 0$ and $\Delta\xi \equiv 0 \bmod 1$.
\end{lemma}
\begin{proof}
Given $\xi \in \hat{\sG}$, extend $\xi$ to a function $\xi_0$ on $V$ by defining $\xi_0(s) = 0$.  For $v \neq s$, $\Delta \xi_0(v) = \Delta' \xi(v)$.  Since $\Delta \xi_0$ is mean 0 on $V$, it follows that $\Delta \xi_0(s) \equiv 0 \bmod 1$, also, so each element of $\hat{\sG}$ can be recovered this way.

Conversely, given such a $\xi$, for $v \in V\setminus \{s\}$, $\Delta \xi(v) = \Delta' \xi|_{V \setminus \{s\}}$ so the claim follows from the structure of the dual group.
\end{proof}
Abusing notation, given any function $\xi: V \to \bR$, define $\hat{\mu}(\xi) = \frac{1}{|V|} \sum_{v \in V} e(\xi_v)$.
\begin{lemma}
 Let $\xi : V \to \bR/\zed$ be such that $\xi(s) = 0$ and $\Delta \xi \equiv 0 \bmod 1$.  Let $\nu = \Delta \xi$ and $\overline{\xi} = g* \nu$ where $g$ is a Green's function of the graph.  Then $\xi - \xi^*$ is a constant, and in particular,
 \begin{equation}
  \left|\hat{\mu}(\xi)\right|= \left|\hat{\mu}(\overline{\xi})\right|.
 \end{equation}

\end{lemma}
\begin{proof}
 Note that the image of $\Delta$ has sum 0 on $V$, so $\nu = \Delta \xi$ has mean 0.  Hence $\Delta(\overline{\xi}) = \nu$ and $\Delta(\xi - \overline{\xi}) = 0$.  The conclusion thus holds, since the kernel of $\Delta$ is the space of constant functions.
\end{proof}
Note that, since the image of $\Delta$ are functions of mean 0, treated as a function on $V$, $\nu = \Delta \xi$ has mean 0, and hence for $\xi \neq 0$,  $\|\nu\|_1 \geq 2$.

The above representation is useful in considering sandpiles on periodic tilings, where $\xi$ may be understood to be a harmonic modulo 1 function on $\sT$ which is $m\Lambda$ periodic and vanishes at the periodic images of the sink.

In the case of an open boundary, another representation is more useful.
\begin{lemma}
Let $m \geq 1$.  Let $G = \sT_m$ be the graph associated to a tiling $\sT$ with reflection symmetry in a family of hyperplanes $\sF$ and fundamental region $\sR$. Identify $\sT_m \setminus \{s\}$ with $\sT \cap m \cdot \sR$.  Given $\xi \in \hat{\sG}(\sT_m)$, there is a function $\xi_0 : \sT \to \bR$ which is harmonic modulo 1, has reflection anti-symmetry in each hyperplane in $m \cdot \sF$, and such that $\xi_0|_{\sT \cap m\cdot \sR} = \xi$.
\end{lemma}
\begin{proof}
 Since any sequence of reflections in $m \cdot \sF$ which maps $m \cdot \sR$ onto itself is the identity, it follows that there is a unique extension $\xi_0$ of $\xi$, thought of as a function on $\sT \cap m \cdot \sR$ to a function which is reflection anti-symmetric in $m\cdot \sF$.  Such a function necessarily vanishes on the vertices of $\sT$ which lie on a hyperplane from $m\cdot \sF$.  Since $\xi_0$ vanishes on the boundary of $m\cdot \sR$, $\Delta \xi_0$ and $\Delta' \xi$ agree on the interior $m \cdot \sR$.  By reflection anti-symmetry, $\Delta \xi_0$ vanishes on $m \cdot \sF \cap \sT$.  Thus $\xi_0$ is harmonic modulo 1.
\end{proof}

Given $\xi \in \hat{\sG}$, the choice of $\xi$ is only determined modulo 1.  As in \cite{HJL17} it is useful for ordering purposes to make a preferred choice of the representation.  Let 
\begin{equation}
 C(\xi) = \frac{1}{2\pi} \arg(\hat{\mu}(\xi)) \in \left[-\frac{1}{2}, \frac{1}{2}\right).
\end{equation}
Let $\xi'$ be defined by choosing, for $x \in V \setminus \{s\}$,
\begin{equation}
 \xi'_{x} \equiv \xi_{x} \bmod 1, \qquad \xi'_{x} \in \left(C(\xi)-\frac{1}{2}, C(\xi) + \frac{1}{2}\right].
\end{equation} Define the \emph{distinguished prevector} of $\xi$,  $\nu(\xi) = \Delta' \xi'$.  
\begin{lemma}\label{size_bound_lemma}
 Let $\xi \in \hat{\sG}$ with distinguished prevector $\nu$.  The Fourier coefficient $\hat{\mu}(\xi)$ satisfies
 \begin{equation}
  1 - \left|\hat{\mu}(\xi)\right| \gg \frac{\|\nu\|_2^2}{|V|} \geq \frac{\|\nu\|_1}{|V|}.
 \end{equation}

\end{lemma}
\begin{proof}The last inequality is true, since $\nu$ is integer valued.

 Treat $\xi$ as defined on $V$ by setting $\xi(s) = 0$ and define $\xi^* = \xi - C(\xi)$.  Since 
 \begin{equation}
  \left|\hat{\mu}(\xi)\right| = \frac{1}{|V|}\sum_{v \in V} e(\xi_v^*) =\frac{1}{|V|} \sum_{v \in V} c(\xi_v^*)
 \end{equation}
is real, and since $\|\xi^*\|_\infty \leq \frac{1}{2}$, it follows from $1 - c(x) \geq 8 x^2$ for $|x|\leq \frac{1}{2}$ that
\begin{equation}
 1 - \left|\hat{\mu}(\xi)\right| \geq \frac{8\|\xi^*\|_2^2}{|V|}.
\end{equation}
Since $\|\Delta\|_{2\to 2}$ is bounded, 
\begin{equation}
 \frac{\|\nu\|_2^2}{|V|} = \frac{\|\Delta \xi^*|_{V \setminus \{s\}}\|_2^2}{|V|} \ll \frac{\|\xi^*\|_2^2}{|V|} \ll 1- \left|\hat{\mu}(\xi)\right|.
\end{equation}

\end{proof}

\section{Spectral estimates}
This section collects together the spectral gap and spectral disjointness estimates needed to prove Theorem \ref{mixing_theorem}.  In the periodic case, the estimates are closely related to the spectral estimates in Section 6 of \cite{HJL17}, although some adjustments are needed to generalize the estimates to higher dimensions, and to treat the Green's function of a tiling.

\subsection{Periodic case}  Given $\xi \in \hat{\sG}_m$, recall that $\hat{\mu}(\xi) = \E_{x \in \bT_m}[e(\xi_x)]$. Given $S \subset \bT_m$, define the `savings from $S$' for the frequency $\xi$, denoted by $\sav(\xi;S)$,
\begin{equation}
 \sav(\xi;S) := |S| - \left|\sum_{x \in S} e(\xi_x)\right|.
\end{equation}
If $S_1, S_2 \subset \bT_m$ are disjoint, then by the triangle inequality,
\begin{equation}
 \sav(\xi; S_1 \cup S_2) \geq \sav(\xi; S_1) + \sav(\xi; S_2).
\end{equation}
The `total savings' for $\xi$ is defined by
\begin{equation}
 \sav(\xi):= \sav(\xi; \bT_m) = |\bT_m| - \left|\sum_{x \in \bT_m} e(\xi_x)\right|
\end{equation}
and satisfies
\begin{equation}
 1 - \left|\hat{\mu}(\xi)\right| = \frac{\sav(\xi)}{|\bT_m|}.
\end{equation}
The spectral gap is given by
\begin{equation}
 \gap_m = \min_{0 \neq \xi \in \hat{\sG}_m} \frac{\sav(\xi)}{|\bT_m|}.
\end{equation}

Set 
\begin{equation}\rho = \left\{\begin{array}{lll} 2 && d = 2\\ 1 && d = 3,4\\ 0 && d \geq 5  \end{array}\right.,
\end{equation}
and $\beta = d-2 + \rho$,
 \begin{equation}
  \beta = \left\{\begin{array}{lll}2 && d = 2,3\\ 3&& d = 4\\ d-2 && d \geq 5  \end{array} \right..
 \end{equation}
By Lemma \ref{greens_fn_decay_lemma_T}, if $\nu \in C^\rho(\sT)$  and $\xi = g*\nu$ then  for $y \neq 0$,
 \begin{equation}
  |\xi_y| \ll \frac{1}{d(0, y)^\beta},
 \end{equation}
and hence $\xi \in \ell^2(\sT)$.
 
Let $B_R(0) = \{x \in \sT: d(x,0) \leq R\}$ and define 
\begin{equation}
 \sC(B,R) := \{\nu \in C^\rho(B_R(0)): \|\nu\|_1 \leq B\}.
\end{equation}
Recall that for $\xi : \sT \to \bR$,
\begin{equation}
 f(\xi) := \sum_{x \in \sT} (1 - c(\xi_x))
\end{equation}
and
\begin{equation}
 \sI := \{\Delta w: w \in C^0(\sT)\} \subset C^2(\sT).
\end{equation}

\begin{lemma}
 The spectral parameter $\gamma$ has characterization, in dimension 2,
 \begin{equation*}
  \gamma = \inf\{f(g*\nu): \nu \in C^2(\sT) \setminus \sI\},
 \end{equation*}
and in dimension at least 3,
\begin{equation*}
 \gamma = \inf\{f(g*\nu): \nu \in C^1(\sT) \setminus \sI\}.
\end{equation*}
The parameter $\gamma_0$ has characterization,
\begin{equation*}
 \gamma_0 = \inf\{f(g*\nu): \nu \in C^{\rho}(\sT) \setminus \sI\}.
\end{equation*}

\end{lemma}
\begin{proof}
 Recall the characterizations,
 \begin{align*}
  \gamma &= \inf\left\{\sum_{x \in \sT} 1 - \cos(2\pi \xi_x): \xi \in \sH^2(\sT),  \Delta \xi \in C^1(\sT), \xi \not \equiv 0 \bmod 1\right\},\\
  \gamma_0 &= \inf\left\{\sum_{x \in \sT} 1- \cos(2\pi \xi_x): \xi \in \sH^2(\sT), \xi \not \equiv 0 \bmod 1\right\}.
 \end{align*}
 First consider the case of $\gamma_0$.  If $\nu \in C^\rho(\sT)\setminus \sI$ then $\Delta (g*\nu) = \nu \in C^\rho(\sT)$ so $\xi = g*\nu$ is harmonic modulo 1 and in $\ell^2(\sT)$.  This demonstrates $\xi \not \equiv 0 \bmod 1$, since otherwise $\xi$ has finite support so that $\nu = \Delta \xi \in \sI$.  Thus 
 \[
  \gamma_0 \leq \inf\{f(g*\nu): \nu \in C^\rho(\sT) \setminus \sI\}.
 \]
To prove the reverse inequality, suppose $\xi \in \sH^2(\sT)$, $\xi \not \equiv 0 \bmod 1$.  Let $\nu = \Delta \xi$.  By Theorem \ref{ell_2_theorem}, $\xi = g*\nu$, $\nu \in C^\rho(\sT)$ and since $\xi$ is not integer valued, $\nu \not \in \sI$. This proves the reverse equality.

The case of $\gamma$ is essentially the same, except that in dimensions at least 5, there is the further restriction that $\nu = \Delta \xi \in C^1(\sT)$.  
\end{proof}

\begin{proposition}\label{savings_convergence_prop}
 Fix $B, R_1 > 0$.  For any $\nu \in \sC(B, R_1)$ and $m > 2R_1$, let $\xi^{(m)} = \xi^{(m)}(\nu)$ be the frequency in $\hat{\sG}_m$ corresponding to $\nu$, namely
 \begin{equation}
  \xi_x^{(m)} = (g_{\bT_m}*\nu)(x) - (g_{\bT_m}*\nu)(0)
 \end{equation}
and let $\xi = \xi(\nu) = g*\nu$.  Then
\begin{equation}
 \sav(\xi^{(m)}) \to f(\xi)\qquad \text{as } m \to \infty.
\end{equation}

\end{proposition}

\begin{proof}
 Let $\xi^* = g_{\bT_m} * \nu$.  Since $\sav(\xi^*) = \sav(\xi^{(m)})$ it suffices to show that  $\sav(\xi^*) \to f(\xi)$ as $m \to \infty$.  By Lemma \ref{deriv_bound_lemma}, $\xi^*(y) \ll \frac{1}{d(0,y)^{\beta}}.$ Observe, using that $\#\{x: d(0,x) \leq R\} \ll R^d$ and partial summation, that
 \begin{align*}
  \sum_{d(0,y) > R} (1 - c(\xi_y^*)) &= O_{B, R_1}(R^{-2\beta + d}),\\
  \sum_{y \in \bT_m} (1-c(\xi_y^*)) &= O_{B, R_1}(1),\\
  \left|\sum_{y \in \bT_m} s(\xi_y^*)\right|& = O_{B, R_1}(1)
 \end{align*}
where in the last line  $s(\xi_y^*)$ is expanded in Taylor series to degree 3 and the linear term sums to zero, since $g_{\bT_m}$ is mean 0.  Note that  the linear term sums to 0.  Use, for $a>0$, $\sqrt{a^2 + b^2} - a \leq \frac{b^2}{2a}$, and hence
\begin{equation}
 \left| \sum_{y \in \bT_m} c(\xi_y^*) + i s(\xi_y^*)\right| = \left|\sum_{y \in \bT_m} c(\xi_y^*)\right| + O\left(\frac{1}{m^d} \right),
\end{equation}
it follows that
\begin{align*}
 \sav(\xi^*) &= O_{B, R_1}(m^{-d}) + \sum_{y \in \bT_m}(1-c(\xi_y^*))\\
 &= O_{B, R_1}(R^{-2\beta + d}) + \sum_{d(0,y) \leq R} (1-c(\xi_y^*)).
\end{align*}
Letting $m \to \infty$ for fixed $R$ obtains $\xi_y^* \to \xi_y$.  Then letting $R \to \infty$ obtains $\lim_{m \to \infty} \sav(\xi^*) = f(\xi)$.
\end{proof}

On $\sT$, $\xi = g* \nu \in \ell^2(\sT)$ if and only if $\nu \in C^\rho(\sT)$.  The following lemma gives a local version of this statement by showing that if a local part of $\nu$ is not in $C^\rho$, subject to some technical conditions, there is arbitrarily large savings near the local piece.

\begin{lemma}\label{irregular_bound_lemma}
 For all $A, B, R_1 > 0$ there exists an $R_2(A, B, R_1) > 2R_1$ such that if $m$ is sufficiently large, then for any $x \in \bT_m$ and any $\nu \in \zed^{\bT_m}$ satisfying the following conditions:
 \begin{enumerate}
  \item $\|\nu\|_1 \leq B$
  \item $\nu|_{B_{R_1}(x)} \not \in C^\rho(\bT_m)$
  \item $d(x, \supp \nu|_{B_{R_1}(x)^c}) > 2R_2$
 \end{enumerate}
it holds
\begin{equation}
 \sav(g_{\bT_m}*\nu; B_{R_2}(x)) \geq A.
\end{equation}
Thus, if $\nu$ has mean zero, then the corresponding frequency $\xi \in \hat{\sG}_m$ satisfies $\sav(\xi; B_{R_2}(x)) \geq A$.
\end{lemma}
\begin{proof}
It suffices to show that $\sav(g_{\bT_m}* \nu; B_{R_2}(x) \cap \Lambda) \geq A$, which simplifies the estimates.

 Assume that the dimension is at most 4, since otherwise $\nu \in C^\rho(\bT_m)$.  The proof is similar to the proof of Lemma 22 from \cite{HJL17}, so only the necessary modifications are indicated.
 
 As there, let $\overline{\xi} = \xi^i + \xi^e$ with
\begin{equation}
 \nu^i := \nu|_{B_{R_1}(x)}, \qquad \nu^e :=  \nu|_{B_{R_1}(x)^c}
\end{equation}
and
 \begin{align}
  \xi^i := g_{\bT_m} * \nu^i, \qquad \xi^e := g_{\bT_m} * \nu^e
 \end{align}
and treat $R_2$ as a parameter which  can be taken arbitrarily large, but fixed.  Let $R$ be a second parameter depending on $R_2$ such that $\frac{R^{d+1}}{R_2^{d-1}} \to 0$ as $R_2 \to \infty$.  By the estimate  $|\nabla g_\eta(y)| \ll \frac{1}{\|y\|^{d-1}}$ from Lemma \ref{deriv_bound_lemma}, it follows that for $\|y\| \leq R$, $\xi_{x+y}^e = \xi_x^e + O\left(\frac{BR}{R_2^{d-1}} \right).$ Thus,
\begin{equation}
 \left|\sum_{\|y\| \leq R} e(\overline{\xi}_{x+y}) \right| = O\left(\frac{BR^{d+1}}{R_2^{d-1}} \right) + \left|\sum_{\|y\| \leq R} e(\xi_{x+y}^i)\right|.
\end{equation}
Thus it suffices to prove that as $R \to \infty$,
\begin{equation}
 \# \{y : \|y\| \leq R\} - \left|\sum_{\|y\| \leq R} e(\xi_{x+y}^i)\right| \to \infty.
\end{equation}

First consider the case $d = 2$.  
If $\nu^i   \not \in C^1(\bT_m)$, the asymptotic for $\nabla g_\eta$ in Lemma \ref{G_eta_eval} implies $|\xi^i_{x + j e_1} - \xi^i_{x}| \to \infty$ while $|\xi^i_{x + (j+1)e_1} - \xi^i_{x +j e_1}| \to 0$ as $j \to \infty$, and hence
\begin{equation}
 R - \left| \sum_{j=1}^R e(\xi^i_{x + j e_1} - \xi^i_x) \right| \to \infty
\end{equation}
as $R \to \infty$, so that the claim holds by choosing $R$ sufficiently large. Suppose instead that $\nu^i \in C^1(\bT_m) \setminus C^2(\bT_m)$. By Lemma \ref{G_eta_eval}, if $\varrho_\eta$ has mean $v_0 \neq 0$,   
\begin{equation}
 g_\eta(n) = \frac{v_0^t \sigma^{-2}n}{ \deg(0)\pi \|\sigma^{-1}n\|^2 \det \sigma} + O\left(\frac{1}{\|\sigma^{-1}n\|^2} \right).
\end{equation}
It follows that there are $0\leq \theta_1 < \theta_2\leq 2\pi$ such that if $\theta_1 \leq \arg(y) \leq \theta_2$, then $|\xi_{x+y}^i| \asymp \frac{1}{\|y\|}$.  It follows that
\begin{align*}
 \sum_{\|y\| \leq R} (1-c(\xi_{x+y}^i)) &\asymp \log R\\
 \sum_{\|y\| \leq R} |s(\xi_{x+y}^i)| & \ll R.
\end{align*}
Using, for $a, b \in \bR$ and $a>0$, $\sqrt{a^2 + b^2} - \sqrt{a^2} \leq \frac{b^2}{2a}$, it follows that
\begin{equation}
 \#\{y: \|y\| \leq R\}- \left| \sum_{\|y\| \leq R} e(\xi_{x+y}^i)\right| \asymp \log R.
\end{equation}

In the case that $d \geq 3$, assume that $\nu^i \in C^0(\bT_m)\setminus C^1(\bT_m)$. Apply Lemma \ref{G_eta_eval} to find that for $\eta$ of mass $C$ with support in a bounded neighborhood of 0,
\begin{align}
 g_\eta(n) &= \frac{C\Gamma\left(\frac{d}{2}-1 \right)}{2 \deg(0) \pi^{\frac{d}{2}}\|\sigma^{-1}n\|^{d-2}\det \sigma} + O\left(\frac{1}{\|\sigma^{-1}n\|^{d-1}} \right).
\end{align}
It follows that for $\|n\| \gg 1$,
\begin{equation}
 |g_\eta(n)| \asymp \frac{1}{\|n\|^{d-2}}.
\end{equation}
In the case $d = 3$, sum in a dimension 2 plane to find
\begin{align*}
 \sum_{\|y\|\leq R, y_3 = 0} 1- c(\xi_{x + y}^i)  &\asymp \log R\\
 \sum_{\|y\| \leq R, y_3 = 0} |s(\xi_{x + y}^i)| & \ll R
\end{align*}
so that \begin{equation}\#\{y: \|y\| \leq R\} - \left|\sum_{\|y\| \leq R} e(\xi_{x + y}^i)\right| \gg \log R.\end{equation}

In the case $d = 4$,
\begin{align*}
 \sum_{\|y\| \leq R} 1 - c(\xi^i_{x+y}) &\asymp \log R\\
 \sum_{\|y\| \leq R} |s(\xi_{x+y}^i)| & \ll R^2.
\end{align*}
Applying the inequality for $\sqrt{a^2 + b^2} - \sqrt{a^2}$ to the real and imaginary parts of $\sum_{\|y\| \leq R} e(\xi_{x+y}^i)$ obtains
\begin{equation}
 \#\{y: \|y\| \leq R\} - \left|\sum_{\|y\| \leq R} e(\xi_{x+y}^i)\right| \asymp \log R.
\end{equation}

\end{proof}
\begin{lemma}\label{regular_bound_lemma}
 For all $B, R_1 > 0$ and $\alpha < 1$, there exists $R_2(\alpha, B, R_1) > 2 R_1$ such that if $m$ is sufficiently large, then for any $x \in \bT_m$ and any $\nu \in \zed^{\bT_m}$ satisfying the following conditions:
 \begin{enumerate}
  \item $\|\nu\|_1 \leq B$
  \item $\nu|_{B_{R_1}(x)} \in C^\rho(\bT_m)$
  \item $d\left(x, \supp \nu|_{B_{R_1}(x)^c} \right) > 2 R_2$
 \end{enumerate}
the bound holds
\begin{equation}
 \sav(g_{\bT_m} * \nu; B_{R_2}(x)) \geq \alpha \sav(\xi^*); \qquad \xi^* = g_{\bT_m} * \nu|_{B_{R_1}(x)}.
\end{equation}

\end{lemma}

\begin{proof}
 The proof is essentially the same as for Lemma 23 of \cite{HJL17}, but is included here for completeness. First it is shown that there is $\delta = \delta(B,R_1)>0$ such that for sufficiently large $m$, if $\sav(\xi^*) < \delta$ then $\sav(\xi^*) = 0$. After making a translation in $\Lambda$, $\sav(\xi^*) = \sav(g_{\bT_m}*\nu')$ for some $\nu' \in \sC(B, R_1)$. Let 
 \begin{equation}
  \gamma' = \min\{f(g*\nu'): \nu' \in \sC(B, R_1) \setminus \sI\}>0.
 \end{equation}
 Since, by Proposition \ref{savings_convergence_prop}, for all sufficiently large $m$, 
 \begin{equation}
  \left|\sav(g_{\bT_m}*\nu') - f(g*\nu')\right| < \frac{\gamma'}{2}
 \end{equation}
for all $\nu' \in \sC(B, R_1)$. Thus, if $\nu' \in\sC(B, R_1) \setminus \sI$, then $\sav(\xi^*) > \frac{\gamma'}{2}$.
Since $\sav(g_{\bT_m} *\nu') = 0$ if $\nu' \in \sI$, it follows that the claim holds with $\delta = \frac{\gamma'}{2}$.

Now set $\epsilon = \epsilon(\alpha, B, R_1) = (1-\alpha)\delta > 0$.  It suffices to show that
\begin{equation}
 \sav(g_{\bT_m}*\nu; B_{R_2}(x)) > \sav(\xi^*) -\epsilon
\end{equation}
which implies the lemma, since the claim is trivial if $\sav(\xi^*) = 0$, while otherwise $\sav(\xi^*) \geq \delta$ so that $\sav(\xi^*)-\epsilon \geq \alpha \sav(\xi^*).$  It suffices to show that if $R$ is fixed, but sufficiently large, that
\begin{equation}
 \sav(\xi^*; B_R(x)) > \sav(\xi^*) - \frac{\epsilon}{2},
\end{equation}
since the difference between $\sav(g_{\bT_m}*\nu; B_R(x))$ and $\sav(\xi^*; B_R(x))$ may be made arbitrarily small by taking $R_2$ sufficiently large. 

By the decay estimates for the Green's function in Lemma \ref{deriv_bound_lemma}, for $y \neq x$,
\begin{equation}
\left|\xi^*_{y}\right| \ll  \frac{1}{d(x,y)^\beta} .
\end{equation}
Thus, for all $d \geq 2$,
\begin{equation}
 \sum_{d(x,y) \leq R} 1-c(\xi_y^*) = O(1).
\end{equation}
Meanwhile,
\begin{equation}\label{large_distance_bound}
 \sum_{d(x,y) > R} 1- c(\xi_y^*) \ll \left\{\begin{array}{lll} \frac{1}{R^2} && d= 2\\ \frac{1}{R} && d=3\\ \frac{1}{R^2} && d = 4\\ \frac{1}{R^{d-4}} && d \geq 5
\end{array}\right..
\end{equation}
Also,
\begin{equation}
 \left|\sum_{d(x,y) \leq R} s(\xi_y^*)\right| \leq \sum_{d(x,y) \leq R} |s(\xi_y^*)| \ll \left\{ \begin{array}{lll} \log R && d=2\\ R && d = 3,4\\ R^2 && d \geq 5 \end{array}\right..
\end{equation}
Using the formula, for $a>0$, $\sqrt{a^2 + b^2} - a \leq \frac{b^2}{2a}$ for the real and imaginary parts obtains
\begin{equation}
 \left|\sum_{d(x,y) \leq R} (1-c(\xi_y^*)) - \sav(\xi^*; B_R(x))\right| \ll \left\{\begin{array}{lll} \frac{(\log R)^2}{R^2} && d = 2\\ \frac{1}{R} && d=3 \\ \frac{1}{R^2} && d= 4\\ \frac{1}{R^{d-4}} && d \geq 5 \end{array}\right..
\end{equation}
Since
\begin{equation}
 \sav(\xi^*) = |\bT_m| - \left| \sum_{z \in \bT_m} e(\xi_z^*)\right| \leq \sum_{z \in \bT_m} (1 - c(\xi_z^*))
\end{equation}
and
\begin{equation}
  \sum_{z \in \bT_m} (1 - c(\xi_z^*)) - \sum_{d(x,y) \leq R} (1 - c(\xi_y^*)) 
\end{equation}
has the bound given in (\ref{large_distance_bound}), the claim follows by letting $R \to \infty$.
\end{proof}

\begin{proposition}\label{spectral_constant_prop}
The spectral constant $\gamma$ is positive, and there exist constants $B_0, R_0 > 0$ such that
 \begin{enumerate}
  \item For sufficiently large $m$, if $\gamma = \gamma_0$ any $\xi \in \hat{\sG}_m$ that achieves the spectral gap, $\sav(\xi) = |\bT_m| \gap_m$, has a prevector $\nu$ which is a translate of some $\nu' \in \sC(B_0, R_0) \subset C^\rho(\bT_m)$. If $\gamma_0 < \gamma$ then the support of $\nu$ is contained in at most two such neighborhoods.
  \item For any $\nu \in C^\rho(\sT)$ satisfying $f(g*\nu) < \frac{3}{2}\gamma_0$, there exists $\nu' \in \sC(B_0, R_0) \subset C^\rho(\sT)$ such that a translate of $\nu'$ differs from $\nu$ by an element of $\sI$.  In particular, $f(g*\nu) = f(g*\nu')$.
 \end{enumerate}

\end{proposition}
\begin{proof}
 This closely follows the proof of Proposition 20 from \cite{HJL17}. The first step in this proof finds a constant $B_0$ such that:
 \begin{enumerate}
  \item [(I)] For sufficiently large $m$, if $\xi^{(m)} \in \hat{\sG}_{m}$ achieves $\sav(\xi^{(m)}) = |\bT_m|\gap_{m}$ then its distinguished prevector $\nu^{(m)}$ must satisfy $\left\|\nu^{(m)}\right\|_1 \leq B_0$.
  \item [(II)] If $\nu \in C^{\rho}(\sT)$  satisfies $f(g *\nu) \leq \frac{3}{2}\gamma_0 + 1$, then $\nu$ differs by an element of $\sI$ from some $\tilde{\nu} \in C^\rho(\sT)$ with $\|\tilde{\nu}\|_1 \leq B_0$.
 \end{enumerate}
 To prove (I), fix $\nu' \in C^\rho(\sT)\setminus \sI$.  Choose $B', R'$ large enough so that $\nu' \in \sC(B', R')$.  For each $m$ sufficiently large let $\nu_m$ be a translation of $\nu'$ and $\xi^{(m)}$ the corresponding element of $\hat{\sG}_m$.  By Proposition \ref{convergence_prop},
 \begin{equation}
  \sav\left(\xi^{(m)} \right) \to f(g*\nu')=\gamma'\; \text{as } m \to \infty,
 \end{equation}
and therefore $\sav(\xi^{(m)}) < \gamma' + 1$ for sufficiently large $m$.

Let $\xi^{(m)} \in \hat{\sG}_{m}$ achieve the spectral gap, and let $\nu^{(m)}$ be the distinguished prevector of $\xi^{(m)}$. By Lemma \ref{size_bound_lemma},
\begin{equation}
 \left\|\nu^{(m)}\right\|_1 \ll \sav\left(\xi^{(m)}\right) < \gamma' + 1.
\end{equation}
This proves (I).

To prove (II), let $\nu \in C^\rho(\sT)$, let $\xi = g*\nu$. Since $\nu \in C^\rho(\sT)$,  $\xi \in \ell^2(\sT)$.  There is a version $\tilde{\xi}: \sT \to \left[-\frac{1}{2}, \frac{1}{2}\right)$, $\tilde{\xi} \equiv \xi \bmod 1$ such that $\Delta \tilde{\xi} = \nu-\Delta w$ which differs from $\nu$ by $\Delta w \in \sI$. Because $\tilde{\nu}$ is integer valued and $\Delta$ is bounded $\ell^2 \to \ell^2$,
\begin{equation}
 \|\tilde{\nu}\|_1 \leq \|\tilde{\nu}\|_2^2 = \|\Delta \tilde{\xi}\|_2^2 \ll \|\tilde{\xi}\|_2^2 = \sum_{x \in \sT} |\tilde{\xi}_x|^2 \ll \sum_{x \in \sT}(1-c(\tilde{\xi}_x)).
\end{equation}
An upper bound on $f(g*\nu)$ thus implies and upper bound on $\|\tilde{\nu}\|_1$.

The covering process described in Proposition 20 of \cite{HJL17} takes as input a vector $\nu \in \zed^\sT$ or $\nu \in \zed^{\bT_m}$ with $\|\nu\|_1 \leq B_0$ and returns a set $\sX'$ and radia $R_1(x), R_2(x)$ satisfying the conditions of Lemma \ref{irregular_bound_boundary_lemma} or Lemma \ref{regular_bound_boundary_lemma}, and such that
\begin{equation}
 \supp \nu \subset \bigcup_{x \in \sX'} B_{R_1(x)}(x), \qquad d\left(x, \supp \nu|_{B_{R_1(x)}(x)^c}\right) > 2 R_2(x)
\end{equation}
for each $x \in \sX'$ and the balls $\{B_{R_2(x)}(x)\}_{x \in \sX'}$ are pairwise disjoint. Let $\sX'' = \{x \in \sX': \nu|_{B_{R_1(x)}(x)} \not \in \sI\}$.   Given $x \in \sX''$, if $u^{(x)} \not \in C^\rho$ then $\sav(\xi)\geq \frac{3}{2}\gamma_0 + 1$.  Also, if $|\sX''| \geq 2$ the savings from $g* u^{(x_1)}$ together with $g* u^{(x_2)}$ is approximately twice the savings from an individual component. Thus the savings is minimized by a $\nu$ with $|\sX''| = 1$.  This reduces the search for the minimizing prevector in the inf defining $\gamma_0$ to a finite check, which proves that $\gamma_0 > 0$.  Since the inf defining $\gamma$ is further restricted by $\Delta \xi \in C^1(\sT)$, $\gamma>0$.  Note that $\gamma \leq 2 \gamma_0$ since if $\nu$ achieves $\gamma_0$ and $\nu^y = \nu - \tau_y \nu \in C^1(\sT)$ and $\xi^y = g*(\nu^y)$ satisfies $\lim_{d(0,y)\to \infty} \sav(\xi^y) = 2 \gamma_0$ as $y \to \infty$. 

To prove item (2) of the Proposition, let $\nu$ be a function on $\sT$.  
Let $u^{(x)} = \nu|_{B_{R_1(x)}(x)}$.  Given $x \in \sX''$, if $u^{(x)} \not \in C^\rho$ then $\sav(\xi)\geq \frac{3}{2}\gamma_0 + 1$.  Also, if $|\sX''| \geq 2$ the savings from $g* u^{(x_1)}$ together with $g* u^{(x_2)}$ is approximately twice the savings from an individual component.  Thus $|\sX''| = 1$ and $u^{(x)} \in C^\rho(\sT)$, so that the difference between $u^{(x)}$ and $\nu$ is in $\sI$.  
Since the savings is translation invariant, the claim  holds.  This suffices for (2).

To prove (1), let $\nu \in C^\rho(\sT)$ satisfy $\xi = g*\nu$ achieves $\gamma_0$.  If $\nu \in C^1(\sT)$, so that $\gamma = \gamma_0$, let $\xi_m$ be the corresponding element of $\hat{\sG}_m$.  As $m \to \infty$, $\sav(\xi_m) \to f(\xi) = \gamma_0$.  If $\xi_0 \in \sG_m$ achieves the spectral gap, perform the clustering algorithm on $\nu_0 = \Delta \xi_0$.  If $|\sX''| > 1$ then the total savings is at least roughly double the savings of the cluster with the least savings.  Since this is asymptotically as large as $\gamma_0$, we obtain a contradiction.  If $\nu \not \in C^1(\sT)$, for any fixed $y$, $\nu^y = \nu - \tau^y \nu \in C^1(\sT)$ and $\xi^y = g_{\bT_m} * \nu^y$ has $\sav(\xi^y) \to f(\xi^y)$ as $m \to \infty$.  Letting $y \to \infty$ obtains a minimal savings which is asymptotically at most twice $\gamma_0$ as $\xi$ ranges in $\hat{\sG}_m \setminus \{0\}$. Arguing as before, we conclude that the optimal prevector has at most two clusters.
\end{proof}

The following lemma is the analogue of Lemma 24 of \cite{HJL17}.

\begin{lemma}\label{frequency_sum_lemma}
 Let $k \geq 1$ be fixed, and let $\nu_1, ..., \nu_k \in C^\rho(\bT_m)$ be bounded functions of bounded support which are $R$-separated, in the sense that their supports have pairwise  distance at least $R$.  Set $\nu = \sum_{i=1}^k \nu_i$.
 As $R \to \infty$,
 \begin{equation}
  1 - \left|\hat{\mu}(\xi(\nu))\right| = O\left( \frac{\log R}{R^{2\beta - d}m^d} \right) + \sum_{i=1}^k \left(1- \left|\hat{\mu}(\xi(\nu_i))\right| \right).
 \end{equation}
The implicit constant depends upon $k$ and the bounds for the functions and their supports.
\end{lemma}

\begin{proof}
 Set $\overline{\xi} = g_{\bT_m} *\nu$ and $\overline{\xi}_i = g_{\bT_m} * \nu_i$, so that $|\hat{\mu}(\xi(\nu))| = |\hat{\mu}(\overline{\xi})|$ and $|\hat{\mu}(\xi(\nu_i))| = |\hat{\mu}(\overline{\xi}_i)|$.  Choose $x_i \in \supp \nu_i$ for each $i$ and let $R' = \lfloor (R-1)/2\rfloor$, so that the balls $B_{R'}(x_i)$ are disjoint. By Lemma \ref{deriv_bound_lemma},
 \begin{equation}
  |\overline{\xi}_i(y)| \ll \frac{1}{d(x_i,y)^\beta} .
 \end{equation}
Also,
\begin{equation}
 \sum_{d(x_i,y) > R'} 1- c\left(\xi_i(y) \right) \ll \frac{1}{R^{2\beta -d}}.
\end{equation}
Taylor expand $s(\xi_i(y))$ to degree 3 to find
\begin{equation}
 \sum_{y \in \bT_m} s(\xi_i(y)) = \sum_{y \in \bT_m} 2\pi \xi_i(y) + O(|\xi_i(y)|^3) \ll \sum_{y \in \bT_m} O(|\xi_i(y)|^3) = O(1),
\end{equation}
by using the fact that $\xi$ is mean 0 on $\bT_m$.
It follows that 
\begin{equation}1 - \left|\hat{\mu}\left(\overline{\xi}_i\right)\right| = \frac{1}{|\bT_m|} \sum_{d(x_i,y) \leq R'} \left(1 - c\left(\overline{\xi}_i(y) \right) \right) + O\left(\frac{1}{R^{2\beta-d}m^d} \right).
\end{equation}
If $d(x_i, y) \leq R'$, then $\overline{\xi}(y) = \overline{\xi}_i(y) + O(R^{-\beta})$, so that
\begin{equation}
 c\left(\overline{\xi}(y) \right) = c\left(\overline{\xi}_i( y) \right) + O\left(\frac{\left|s\left(\overline{\xi}_i(y) \right) \right|}{R^\beta} \right) + O(R^{-2\beta}).
\end{equation}
The estimates hold
\begin{equation}
 \sum_{d(x_i, y) \leq R'} \left|s\left(\overline{\xi}_i(y) \right)\right| \ll \left\{\begin{array}{lll}\log R && d = 2\\
 R && d=3,4\\ R^2 && d \geq 5\end{array} \right..
\end{equation}
Thus
\begin{equation}
 1 - \left|\hat{\mu}\left(\overline{\xi}_i \right) \right| =\frac{1}{|\bT_m|}\sum_{d(x_i, y) \leq R'} \left(1-c\left(\overline{\xi}_y\right)\right) + O\left(\frac{\log R}{R^{2\beta-d} m^d} \right) .
\end{equation}

For $z \not \in \bigcup_{i=1}^k B_{R'}(x_i)$, let $r_i = d(x_i, z)$, so that \begin{equation}\left|\overline{\xi}(z)\right| = O\left(\frac{1}{r_1^{\beta}}+ \cdots + \frac{1}{r_k^{\beta}} \right).\end{equation}
It follows that 
\begin{equation}
 \sum_{z \not \in \bigcup_{i=1}^k B_{R'}(x_i)} \left(1-c\left(\overline{\xi}(z) \right) \right) = O\left(\frac{1}{R^{2\beta-d}} \right),
\end{equation}
and thus
\begin{equation}
 \sum_{i=1}^k \left(1 - \left|\hat{\mu}\left(\overline{\xi}_i \right) \right| \right) = O \left(\frac{\log R}{R^{2\beta - d}m^d} \right) + \frac{1}{|\bT_m|}\sum_{z \in \bT_m} \left(1-c\left(\overline{\xi}(z) \right) \right).
\end{equation}
We have $\RE(\hat{\mu}(\overline{\xi})) \gg 1$.  Meanwhile, by Taylor expanding $\sin$ to degree 3,
\begin{equation}
 \IM\left(\hat{\mu}\left(\overline{\xi}\right) \right) = \frac{1}{|\bT_m|} \sum_{z \in \bT_m} s\left(\overline{\xi}(z) \right) = O\left(\frac{1}{m^d} \right).
\end{equation}
It follows that
\begin{equation}
  \sum_{i=1}^k \left(1 - \left|\hat{\mu}\left(\overline{\xi}_i \right) \right| \right) = O \left(\frac{\log R}{R^{2\beta - d}m^d} \right) + 1-\left|\hat{\mu}(\xi)\right|.
\end{equation}

\end{proof}

For larger frequencies $\xi$, a clustering is used on the prevector $\nu = \Delta \xi$.  Given a radius $R$, say two points $x_o, x_t \in \supp \nu$ are $R$-path connected if there exist points $x_o = x_0, x_1, ..., x_n = x_t$ in $\supp \nu$ such that for all $0 \leq i < n$, $d(x_i, x_{i+1}) < R$. Given $\nu$, let $\sC = \sC(\nu)$ be the $R$-path connected components in $\supp \nu$.  Say that $\nu$ is $R$-\emph{reduced} if for all $C \in \sC$, $\nu|_{C} \not \in \sI$.  The $R$-reduction of $\nu$ is the prevector $\nu'$ which is equivalent to $\nu$ and omits any clusters $C$ such that $\nu|_{C} \in \sI$.  Evidently, $\nu$ and $\nu'$ generate the same frequency $\xi \in \hat{\sG}_m$ and each norm of $\nu'$ is no larger than the norm of $\nu$.

\begin{lemma}\label{savings_lemma}
 Let $B \geq 1$ be a fixed parameter.  There is a function $\eta(B, R)$ tending to 0 as $R \to \infty$ such that for all $m$ sufficiently large, if $\nu \in \zed^{\bT_m}$ satisfies the following conditions:
 \begin{enumerate}
  \item $\nu$ is $R$-reduced
  \item $\|\nu\|_{L^\infty} = O(1)$
  \item $\nu$ has an $R$-cluster $C$ for which  $\left\| \nu|_C \right\|_1 \leq B$
 \end{enumerate}
then
\begin{equation}
 \sav(g_{\bT_m}*\nu; \nbd(C)) \geq \gamma_0-o(1) - \eta(B,R),
\end{equation}
with $o(1)$ tending to 0 as $m \to \infty$.
\end{lemma}

\begin{proof}
 Decompose the phase function $\overline{\xi} = g_{\bT_m} * \nu$ into an internal and external component, $\overline{\xi} = \xi^i + \xi^e$,  where
 \begin{equation}
  \xi^i := g_{\bT_m} * \nu|_{C}, \qquad \xi^e := g_{\bT_m} * \nu|_{C^c}.
 \end{equation}
 
Let $\sQ = \sT/\Lambda$.  The argument sums over each of the individual classes in $\sQ$ so that discrete derivatives may be applied in $\Lambda$.  The first observation is that, in a fixed class $q \in \sQ$, the external phase $\xi^e$ may be well approximated in the cluster $\nbd(C)$ by a polynomial of degree at most 2. Note that for $x \in \nbd(C)$, any $y \in \supp(\nu) \setminus C$ satisfies $d(x,y) > R$. By the bound in Lemma \ref{deriv_bound_lemma}, for $|\ua| = 3$,  
\begin{align*}
 |D^{\ua} \xi_x^e| &= \left|\sum_{y \in C^c} \nu(y) D^{\ua} g_{y,\bT_m}(x)\right|\\
 &\leq \|\nu\|_{\infty} \sum_{y \in B_R(x)^c} \left| D^{\ua} g_{y, \bT_m}(x)\right|\\
 &\ll \sum_{d(x,y) \geq R} \frac{1}{d(x,y)^{d+1}}\\
 & \ll \int_R^\infty \frac{dr}{r^2} \ll \frac{1}{R}.
\end{align*}

The proof now proceeds essentially as in Lemma 25 of \cite{HJL17}.  Let $R_1, R_2, R_3$ which tend to $\infty$ with $R$ and satisfy $R_1 < R_2 < R_3 < R$, and
\begin{equation}
 R_1 \to \infty, \; \frac{R_2}{R_1^4} \to \infty, \; \frac{R_3}{R_1 R_2^2}\to \infty,\; \frac{R}{R_1^2 R_3^2} \gg 1, \; \text{as } R \to \infty.
\end{equation}
First, for each $x \in C$,  choose a representative $q$ for each class in $\sQ$ with $d(q,x) = O(1)$ and assume that for all $\lambda \in \Lambda$ such that $\|\lambda\|_1 \leq R_1$, 
\begin{equation}
\left\|\xi_{q}^e - \xi_{q+\lambda}^e\right\|_{\bR/\zed} < \frac{1}{R_1^{d+1}}.
\end{equation}
The clustering process of Proposition 20 of \cite{HJL17} obtains a cover 
\begin{equation}
 \supp \nu \subset \bigsqcup_{x \in \sX'} B_{\tilde{R}_1(x)}(x)
\end{equation}
and such that, for each $x \in \sX'$ there are radii $2\tilde{R}_1(x) < \tilde{R}_2(x)$ such that
\begin{equation}
 d\left(x, \supp \nu|_{B_{\tilde{R}_1(x)}(x)^c}\right) > 2 \tilde{R}_2(x), \qquad x \in \sX',
\end{equation}
and the balls $\{B_{\tilde{R}_2(x)}(x)\}_{x \in \sX'}$ are disjoint, and meet the conditions of either Lemma \ref{irregular_bound_lemma} or \ref{regular_bound_lemma}.
The radii $\tilde{R}_2$ are uniformly bounded by some $R_0$, with a bound depending only on $B$.  By taking $R$ sufficiently large assume that $R_0$ is arbitrarily small compared to $R_1$.

Let $x \in \sX'$ and $q \in \sQ$ with $d(x,q) = O(1)$ to find
\begin{equation}
 \left|\sum_{d(y,x) \leq R', y \equiv q \bmod \Lambda} e(\xi_{y}^i + \xi_{y}^e) \right| = \left|\sum_{d(y,x) \leq R', y \equiv q \bmod \Lambda} e(\xi_{y}^i ) \right| + O\left(\frac{1}{R_1}\right).
\end{equation}
Thus,
\begin{align*}
\sum_{d(y,x) \leq R'} e(\xi_y^i + \xi_y^e) = \sum_{q \in \sQ} e(\xi_q^e) \sum_{d(y,x) \leq R', y \equiv q \bmod \Lambda} e(\xi_y^i) + O\left(\frac{1}{R_1}\right).
\end{align*}
Let $\sX'' = \{x \in \sX': \nu_{B_{\tilde{R}_1(x)}(x)} \not \in \sI\}$.  Let $u^x = \nu|_{B_{\tilde{R}_1(x)}(x)}$.  If $u^x \not \in C^\rho(\bT_m)$ then each sum 
\begin{equation}
 \sum_{d(y,x)\leq R', y \equiv q \bmod \Lambda} e(\xi_y^i)
\end{equation}
can be made to save an arbitrary constant by choosing $R'$ sufficiently large, which suffices to complete the proof of the lemma, so  assume $u^x \in C^\rho(\bT_m)$.  Under this condition, 
\begin{align*}
\sum_{d(y,x) \leq R', y \equiv q \bmod \Lambda} |s(\xi_y^i)| \ll \frac{1}{R^{\beta - d}}
\end{align*}
so that
\begin{align*}
 \sav(\xi^i; B_{R'}(x) \cap q \bmod \Lambda) = \sum_{d(y,x) \leq R', y \equiv q \bmod \Lambda} 1- c(\xi_y^i) + O\left(\frac{1}{{R'}^{2\beta - d}} \right).
\end{align*}
In particular, 
\begin{align*}
 &\left|\sum_{q \in \sQ} e(\xi_q^e) \sum_{d(y,x) \leq R', y \equiv q \bmod \Lambda} e(\xi_y^i)\right|\leq \sum_{q \in \sQ} \left|\sum_{d(y,x) \leq R', y \equiv q \mod \Lambda} e(\xi_y^i) \right|\\& \leq \left|\sum_{d(y,x) \leq R'} e(\xi_y^i)\right| + O\left(\frac{1}{{R'}^{2\beta-d}} \right).
\end{align*}
If there are two or more elements of $\sX''$, appealing to Lemma \ref{regular_bound_lemma} saves more than $\gamma_0$ if $m$ is sufficiently large.  If $|\sX'|=1$ let $u^x \in C^\rho(\bT_m) \setminus \sI$. This obtains
\begin{align*}
 \sav(\xi^i; B_{R_1}(x)) &= \sav(g_{\bT_m}*u^x; B_{R_1}(x)) \\&= \sav(g_{\bT_m}*u^x) + O\left( \frac{(\log R_1)^2}{R_1^{2\beta-d}}\right).
\end{align*}
Since the support of $u^x$ is treated as bounded and for fixed $\nu \in C^\rho(\sT)$, $\sav(g_{\bT_m}*\nu) \to f(g*\nu)$ as $m \to \infty$, $\sav(g_{\bT_m}*u^x) \geq \gamma_0 - o(1)$ as $m \to \infty$. This again suffices for the lemma.

The remainder of the proof is the same as the proof of Lemma 25 of \cite{HJL17}, which handles the case of a linear or quadratic external phase.

\end{proof}

\subsection{Open boundary case}
In the case of a reflected boundary let $\sF$ be the family of reflecting hyperplanes, and $\sR$ the fundamental open region.  The number of vertices is $|\sT_m| = 1+ |m \sR \cap \sT|$. Consider $\xi \in \hat{\sG}_m$ to be an $m\sF$-anti-symmetric function on $\sT$.

Recall that in two dimensions, the spectral parameters are defined by
\begin{align*}
 \gamma_0 &= \inf_{\substack{\xi \in \sH^2(\sT)\\ \xi \not \equiv 0 \bmod 1}} \sum_{x \in \sT}1- c( \xi_x)\\
 \gamma_1 &= \frac{1}{2} \inf_{a \in \sL} \inf_{\substack{\xi \in \sH_a^2(\sT)\\ \xi \not \equiv 0 \bmod 1}} \sum_{x \in \sT} 1 - c( \xi_x)\\
 \gamma_2 &= \inf_{(a_1, a_2) \in \sC} \inf_{\substack{\xi \in \sH_{(a_1, a_2)}^2(\sT)\\ \xi \not \equiv 0 \bmod 1}} \sum_{x \in Q_{(a_1,a_2)}} 1- c( \xi_x).
\end{align*}
Define corresponding functionals
\begin{align*}
 f(\xi) &= \sum_{x \in \sT} 1-c( \xi_x)\\
 f_a(\xi) &= \frac{1}{2} \sum_{x \in \sT} 1-c( \xi_x)\\
 f_{(a_1, a_2)}(\xi) &= \sum_{x \in Q_{(a_1,a_2)}} 1 - c( \xi_x).
\end{align*}

In dimension $d \geq 3$, for $0 \leq i < d$,
\begin{equation}
 \gamma_i =   \inf_{\substack{S \subset \{1, 2, ..., d\}\\ |S| = i}}\inf_{\substack{ \xi \in \sH^2_S(\sT)\\ \xi \not \equiv 0 \bmod 1}} \sum_{x \in \sT/\fS_S} 1- c( \xi_x).
\end{equation}
Define the corresponding functional 
\begin{equation}
 f_S(\xi) = \sum_{x \in \sT/\fS_S} 1 - c(\xi_x).
\end{equation}

Recall
\begin{equation}
 \rho = \left\{\begin{array}{lll} 2 && d=2\\ 1 && d=3,4\\ 0 && d\geq 5 \end{array}\right..
\end{equation}
Also, let $B_R(0) = \{x \in \sT: d(x,0) \leq R\}$ and 
\begin{equation}
 \sC(B,R) := \{\nu \in C^\rho(B_R(0)): \|\nu\|_1 \leq B\}.
\end{equation}
The graph $\sT/\fS_S$ is given the quotient distance.

Let $r_m \to \infty$ with $m$ be a parameter, say $r_m = \log m$. Let $\nu : \sT_m \to \zed$.  Say $\nu$ is a \emph{co-dimension $j$ cluster} if its support has distance at most $r_m$ to a boundary of co-dimension $j$, but not any boundary of codimension $i > j$.  Let 
\begin{equation}
\hat{\sG}_{m,j} =\{\xi \in \hat{\sG}_m: \Delta \xi \text{ co-dim } j \text{ cluster}\}.
\end{equation}
Define the $j$th spectral gap to be
\begin{equation}
 \gap_{m,j} = \inf_{\substack{\xi \in \hat{\sG}_{m,j}\\ \xi \not \equiv 0 \bmod 1}} 1- \left|\hat{\mu}(\xi)\right|.
\end{equation}

Versions of the claims in the periodic case adapted to the reflected boundary case are as follows.
\begin{proposition}\label{convergence_prop}
 Fix $B, R_1 > 0$.  Let $\nu \in \sC(B, R_1)$ have reflection anti-symmetry in a family $S$ of $j$ hyperplanes.  For any $m > 2R_1$, let $\nu_m$ be any $j$-cluster in $\sT_m$ obtained by translating $\nu$ parallel to the reflecting hyperplanes, then by imposing reflection anti-symmetry in $m\cdot \sF$.  Let $\xi^{(m)} = \xi^{(m)}(\nu)$ be the frequency in $\hat{\sG}_m$ corresponding to $\nu_m$, 
and let $\xi = \xi(\nu) = g*\nu$.  Then
\begin{equation}
 \sav(\xi^{(m)}) \to f_S(\xi)\qquad \text{as } m \to \infty.
\end{equation}

\end{proposition}

\begin{proof}
In this proof, identify $\sT_m$ with $m\cdot \sR \cap \sT$.  Recall that functions which are reflection anti-symmetric in $\sF$ are periodic in a lattice $\Lambda$, and that $\sR$ has finite index in $\sT/\Lambda$.    Treated as a function on $\bT_m = \sT/m\Lambda$, $\nu_m$ may be considered as the sum of $I$ functions of bounded support 
\begin{equation}
\nu_m=\sum_{i=1}^I \nu_{m,i}
\end{equation}
one of whose support, say $\nu_{m,1}$ intersects $\sT_m$ and is a translate of $\nu$.  By the condition of being a $j$ cluster, the distance from the support of the next nearest component to $\sT_m$ is at least $r_m$, since $\nu_m$ has distance at least $r_m$ from the corresponding reflecting boundary.  Thus, $\xi^{(m)} = \sum_{i=1}^I g_{\bT_m} * \nu_{m,i} = \sum_{i=1}^I \xi^{(m)}_i.$ Let $x_i \in \supp \nu_{m,i}$.  By the decay estimate in Lemma \ref{deriv_bound_lemma},
\begin{equation}
 \xi_i^{(m)}(y) \ll \frac{1}{d(x_i,y)^\beta}.
\end{equation}
By Taylor expansion to degree 1,
\begin{equation}
 \sum_{y \in \sT_m} |s(\xi^{(m)}(y))|  \ll \left\{\begin{array}{lll} \log m && d=2\\ m && d=3,4\\ m^2 && d \geq 5\end{array}\right..
\end{equation}
Thus,
\begin{equation}
 \left|\sav(\xi^{(m)}) - \sum_{y \in \sT_m} (1 - c(\xi^{(m)}(y)))\right| \ll \left\{ \begin{array}{lll} \frac{(\log m)^2}{m^2} && d = 2\\ \frac{1}{m} && d=3\\ \frac{1}{m^2} && d= 4\\ \frac{1}{m^{d-4}} && d \geq 5 \end{array}\right..
\end{equation}
For $R$ a fixed parameter, which may be taken arbitrarily large, using $1 - c(x) \ll x^2$ 
\begin{align*}
 \sum_{y \in \sT_m, d(y, x_1)> R} 1 - c(\xi^{(m)}(y))&\ll \sum_{y \in \sT_m, d(y, x_1)> R} \left(\sum_{i=1}^I \frac{1}{d(y, x_i)^{2\beta}} \right)\ll \frac{1}{R^{2 \beta -d}}.
\end{align*}
Meanwhile, for $d(y, x_1) \leq R$,
\begin{equation}
1-c(\xi^{(m)}(y)) = 1 - c(\xi^{(m)}_1(y)) + o(1)
\end{equation}
with the error holding as $m \to \infty$. Thus, 
\begin{align*}
 \sav(\xi^{(m)}) = \sum_{y \in \sT_m, d(y, x_1)\leq R} 1 - c(\xi^{(m)}(y)) + o(1) + O\left( \frac{1}{R^{2 \beta -d}}\right).
\end{align*}
Letting $m \to \infty$, $\xi^{(m)}$ converges pointwise to a translated version of $\xi$, then letting $R \to \infty$ obtains the claim.
\end{proof}

\begin{lemma}\label{irregular_bound_boundary_lemma}
 For all $A, B, R_1 > 0$ there exists an $R_2(A, B, R_1) > 2R_1$ such that if $m$ is sufficiently large, then for any $x \in \sT_m$ and any $\nu \in \zed^{\sT_m}$ satisfying the following conditions:
 \begin{enumerate}
  \item $\|\nu\|_1 \leq B$
  \item $\nu|_{B_{R_1}(x)} \not \in C^{\rho}(\sT_m)$ and $\nu$ is a $j$-cluster
  \item $d(x, \supp \nu|_{B_{R_1}(x)^c}) > 2R_2$
 \end{enumerate}
the bound holds
\begin{equation}
 \sav(g*\nu; B_{R_2}(x)) \geq A.
\end{equation}
Thus, if $\nu$ has mean zero, then the corresponding frequency $\xi \in \hat{\sG}_m$ satisfies $\sav(\xi; B_{R_2}(x)) \geq A$.
\end{lemma}

\begin{proof}
The proof is the same as of Lemma \ref{irregular_bound_lemma}.
\end{proof}

\begin{lemma}\label{regular_bound_boundary_lemma}
 For all $B, R_1 > 0$ and $\alpha < 1$, there exists $R_2(\alpha, B, R_1) > 2 R_1$ such that if $m$ is sufficiently large, then for any $x \in \sT_m$ and any $\nu \in \zed^{\sT_m}$ satisfying the following conditions:
 \begin{enumerate}
  \item $\|\nu\|_1 \leq B$
  \item $\nu|_{B_{R_1}(x)} \in C^{\rho}(\sT_m)$ and $\nu|_{B_{R_1}(x)}$ is a $j$-cluster
  \item $d\left(x, \supp \nu|_{B_{R_1}(x)^c} \right) > 2 R_2$
 \end{enumerate}
the bound holds
\begin{equation}
 \sav(g * \nu; B_{R_2}(x)) \geq \alpha \sav(\xi^*); \qquad \xi^* = g * \nu|_{B_{R_1}(x)}.
\end{equation}

\end{lemma}

\begin{proof}
 The proof is the same as of Lemma \ref{regular_bound_lemma}.
\end{proof}

\begin{proposition}\label{spectral_gap_prop}
The spectral parameters $\gamma_j $ are positive.  If $j = 0$ or $\gamma_j < \gamma_{j-1}$ then there exist constants $B_0, R_0 > 0$ such that
 \begin{enumerate}
  \item  For sufficiently large $m$, any $\xi \in \hat{\sG}_{m,j}$ that achieves the $j$th spectral gap, $\sav(\xi) = (1 + |\sT_m|) \gap_{m,j}$, has a prevector $\nu\in C^\rho(\sT_m)$ which is a translate of some $\nu' \in \sC(B_0, R_0) \subset C^{\rho}(\sT)$ with reflection anti-symmetry in a family $S$ of $j$ hyperplanes.
  \item For any $\nu \in C^\rho(\sT)$ which has reflection anti-symmetry in a family $S$ of hyperplanes, $|S| = j$, and satisfying $f_S(g*\nu) < \gamma_{j-1}$ or $j = 0$, there exists $\nu' \in \sC(B_0, R_0) \subset C^\rho(\sT)$ with reflection anti-symmetry in $S$ such that a translate of $\nu'$ differs from $\nu$ by an element of $\sI$.  In particular, $f_S(g*\nu) = f_S(g*\nu') $.
 \end{enumerate}
If $j > 0$ and $\gamma_j = \gamma_{j-1}$ the above statements hold with the caveat that the prevector has bounded $\ell^1$ norm and bounded support, but that the support of the prevector may be arbitrarily far from 0.
\end{proposition}

\begin{proof}
 This is similar to the proof of Proposition \ref{spectral_constant_prop}. The first step in this proof finds a constant $B_0$ such that:
 \begin{enumerate}
  \item [(I)] For sufficiently large $m$, if $\xi^{(m)} \in \hat{\sG}_{m,j}$ achieves $\sav(\xi^{(m)}) = (1 + |\sT_m|)\gap_{m,j}$ then its distinguished prevector $\nu^{(m)}$ must satisfy $\left\|\nu^{(m)}\right\|_1 \leq B_0$.
  \item [(II)] If $\nu \in C^{\rho}(\sT)$ has reflection symmetry in a set $S$ of hyperplanes, $|S|=j$ and satisfies $f_S(g *\nu) \leq \frac{3}{2}\gamma_j + 1$, then $\nu$ differs by an element of $\sI$ from some $\tilde{\nu} \in C^\rho(\sT)$ with $\|\tilde{\nu}\|_1 \leq B_0$.
 \end{enumerate}
 To prove (I), fix $\nu' \in C^\rho(\sT)\setminus \sI$ with reflection anti-symmetry in a family $S$ of hyperplanes.  Choose $B', R'$ large enough so that $\nu' \in \sC(B', R')$.  For each $m$ sufficiently large let $\nu_m$ be a translation of $\nu'$ along hyperplanes on the boundary of $m \cdot \sR$ and $\xi^{(m)}$ the corresponding element of $\hat{\sG}_m$.  By Proposition \ref{convergence_prop},
 \begin{equation}
  \sav\left(\xi^{(m)} \right) \to f_S(g*\nu')=\gamma'\; \text{as } m \to \infty,
 \end{equation}
and therefore $\sav(\xi^{(m)}) < \gamma' + 1$ for sufficiently large $m$.

Let $\xi^{(m)} \in \hat{\sG}_{m,j}$ achieve the $j$th spectral gap, and let $\nu^{(m)}$ be the distinguished prevector of $\xi^{(m)}$. By Lemma \ref{size_bound_lemma},
\begin{equation}
 \left\|\nu^{(m)}\right\|_1 \ll \sav\left(\xi^{(m)}\right) < \gamma' + 1.
\end{equation}
This proves (I).

To prove (II), let $\nu \in C^\rho(\sT)$, let $\xi = g*\nu$. Since $\nu \in C^\rho(\sT)$,  $\xi \in \ell^2(\sT)$.  There is a version $\tilde{\xi}: \sT \to \left[-\frac{1}{2}, \frac{1}{2}\right)$, $\tilde{\xi} \equiv \xi \bmod 1$ such that $\Delta \tilde{\xi} = \nu-\Delta w$ which differs from $\nu$ by $\Delta w \in \sI$. Because $\tilde{\nu}$ is integer valued and $\Delta$ is bounded $\ell^2 \to \ell^2$,
\begin{equation}
 \|\tilde{\nu}\|_1 \leq \|\tilde{\nu}\|_2^2 = \|\Delta \tilde{\xi}\|_2^2 \ll \|\tilde{\xi}\|_2^2 = \sum_{x \in \sT} |\tilde{\xi}_x|^2 \ll \sum_{x \in \sT}(1-c(\tilde{\xi}_x)).
\end{equation}
An upper bound on $f(g*\nu)$ thus implies and upper bound on $\|\tilde{\nu}\|_1$.

The covering process described in Proposition 20 of \cite{HJL17} takes as input a vector $\nu \in \zed^\sT$ or $\nu \in \zed^{\sT_m}$ with $\|\nu\|_1 \leq B_0$ and returns a set $\sX'$ and radia $R_1(x), R_2(x)$ satisfying the conditions of Lemma \ref{irregular_bound_boundary_lemma} or Lemma \ref{regular_bound_boundary_lemma}, and such that
\begin{equation}
 \supp \nu \subset \bigcup_{x \in \sX'} B_{R_1(x)}(x), \qquad d\left(x, \supp \nu|_{B_{R_1(x)}(x)^c}\right) > 2 R_2(x)
\end{equation}
for each $x \in \sX'$ and the balls $\{B_{R_2(x)}(x)\}_{x \in \sX'}$ are pairwise disjoint. Let $\sX'' = \{x \in \sX': \nu|_{B_{R_1(x)}(x)} \not \in \sI\}$. 

First consider item (2) of the Proposition, so that $\nu$ is a function on $\sT$ which is anti-symmetric in a set $S$ of hyperplanes.  
Let $u^{(x)} = \nu|_{B_{R_1(x)}(x)}$ treated as a function which is anti-symmetric in $S$.  Given $x \in \sX''$, if $u^{(x)} \not \in C^\rho$ then $\sav(\xi)\geq \frac{3}{2}\gamma_j + 1$.  Also, if $|\sX''| \geq 2$ the savings from $g* u^{(x_1)}$ together with $g* u^{(x_2)}$ is approximately twice the savings from an individual component.  Thus it suffices to assume that $|\sX''| = 1$ and $u^{(x)} \in C^\rho(\sT)$.  

If $j \geq 1$, $g*u^{(x)}$ is in $\ell^2(\sT)$.  It follows that if $\tau_y$ indicates translation in a direction away from a hyperplane of $S$, then
\begin{equation}
 \liminf_{y \to \infty} \sav(g* (\tau_y u^{(x)})) \geq \gamma_{j-1}.
\end{equation}
Hence if $\gamma_j < \gamma_{j-1}$ then the inf is achieved by a function with support in a bounded neighborhood of 0. If $j = 0$ then there are no reflecting hyperplanes, and the savings is translation invariant, so that the claim still holds.  This suffices for (2).

For (1), argue similarly, using that as $m \to \infty$ the savings converges to $f_S$.

\end{proof}

\begin{lemma}\label{separated_clusters_lemma}
 Let $k \geq 1$ be fixed, and let $\nu_1, ..., \nu_k $ be bounded functions of bounded support and such that $\nu_i \in C^{\rho}(\sT_m)$.  Suppose the functions are  $R$-separated, in the sense that their supports have pairwise distance at least $R$.  Set $\nu = \sum_{i=1}^k \nu_i$.  Then as $R \to \infty$,
 \begin{equation}
  1 - \left|\hat{\mu}(\xi(\nu))\right| = O\left(\frac{\log R}{R^{2\beta -d}m^d} \right) + \sum_{i=1}^k \left(1- \left|\hat{\mu}(\xi(\nu_i))\right| \right).
 \end{equation}
The implicit constant depends upon $k$ and the bounds for the functions and their supports.
\end{lemma}
\begin{proof}
 Let $\nu_j = \sum_{i=1}^I \nu_{j,i}$ as a function on $\sT/m\Lambda$ with $\nu_{j,1}$ having support that intersects $\sT_m$.   Let $x_{j,i} \in \supp \nu_{j,i}$ and $\overline{\xi} = g_{\bT_m} * \nu$, $\overline{\xi}_j = g_{\bT_m} * \nu_j$, $\overline{\xi}_{j,i} = g_{\bT_m}* \nu_{j,i}$.  Let $R' = \left\lfloor \frac{R-1}{2}\right\rfloor$ so that the balls $B_{R'}(x_{i,j})$ are pairwise disjoint. 
 
 By the decay estimate in Lemma \ref{deriv_bound_lemma},
 \begin{equation}
  \left|\xi_{i,j}(y)\right| = O\left(\frac{1}{1 + d(x_{i,j}, y)^\beta} \right).
 \end{equation}
Thus,
\begin{align*}
 \sum_{y \in \sT_m} \left|s\left(\overline{\xi}_i(y)\right)\right| &\ll m^{d-\beta},\\
 \sum_{y \in \sT_m} \left|s\left(\overline{\xi}(y)\right)\right| &\ll m^{d-\beta}.
\end{align*}
It follows that
\begin{align*}
 \sav(\xi) &= \sum_{y \in \sT_m} \left(1 - c\left( \overline{\xi}(y)\right)\right) + O\left(\frac{1}{m^{2\beta-d}} \right),\\
 \sav(\xi_i) &= \sum_{y \in \sT_m} \left(1 - c\left( \overline{\xi}_i(y)\right)\right) + O\left(\frac{1}{m^{2\beta-d}} \right).
\end{align*}

If $d(x_{i,j}, y) \leq R'$, then $\overline{\xi}(y) = \overline{\xi}_{i,j}(y) + O(R^{-\beta})$, so that
\begin{equation}
 c\left(\overline{\xi}(y) \right) = c\left(\overline{\xi}_{i,j}(y) \right) + O\left(\frac{\left|s\left(\overline{\xi}_{i,j}(y) \right) \right|}{R^\beta} \right) + O(R^{-2\beta}).
\end{equation}
The estimates hold
\begin{equation}
 \sum_{d(x_{i,j}, y) \leq R'} \left|s\left(\overline{\xi}_{i,j}(y) \right)\right| \ll \left\{\begin{array}{lll}\log R && d = 2\\
 R && d=3,4\\ R^2 && d \geq 5\end{array} \right..
\end{equation}
Meanwhile, $\left( \sum_{j=1}^I \overline{\xi}_{i,j}(y)\right)^2 \leq I \sum_{j=1}^I \overline{\xi}_{i,j}(y)^2$. Thus
\begin{equation}
 \sum_{y \in \sT_m, d(y, x_{i,1})> R'} 1 - c\left(\overline{\xi}_i(y)\right) \ll \frac{1}{R^{2\beta -d}}.
\end{equation}
It follows that
\begin{equation}
 1-\left|\hat{\mu}\left(\overline{\xi}_i\right) \right| = O\left(\frac{\log R}{R^{2\beta-d}m^d} \right) + \frac{1}{1 + |\sT_m|} \sum_{y \in \sT_m, d(y, x_{1,i}) \leq R'} \left(1-c\left(\overline{\xi}(x_i + y) \right) \right).
\end{equation}
Similarly,
\begin{equation}
 \sum_{\substack{z \in \sT_m\\ z \not \in \bigcup_{j=1}^k B_{r'}(x_{1,j})}} 1 - c\left( \overline{\xi}(z)\right) \ll \frac{1}{R^{2\beta -d}}.
\end{equation}
It follows
\begin{equation}
 1 - \left|\hat{\mu}\left(\overline{\xi} \right) \right| = O\left(\frac{\log R}{R^{2\beta-d}m^d} \right) + \sum_{i=1}^k \left(1 -\left|\hat{\mu}\left(\overline{\xi}_i \right) \right| \right). 
\end{equation}

\end{proof}
For larger frequencies, the analogue of Lemma \ref{savings_lemma} in the case of an open boundary is as follows.
\begin{lemma}\label{savings_lemma_boundary}
 Let $B \geq 1$ be a fixed parameter.  There is a function $\eta(B, R)$ tending to 0 as $R \to \infty$ such that for all $m$ sufficiently large, if $\nu \in \zed^{\sT_m}$ satisfies the following conditions:
 \begin{enumerate}
  \item $\nu$ is $R$-reduced
  \item $\|\nu\|_{L^\infty} = O(1)$
  \item $\nu$ has an $R$-cluster $C$ for which $\left\| \nu|_C \right\|_1 \leq B$ and which is a $j$ boundary cluster
 \end{enumerate}
then
\begin{equation}
 \sav(g*\nu; \nbd(C)) \geq (1+|\sT_m|) \gap_{m,j} - \eta(B,R).
\end{equation}

\end{lemma}

\begin{proof}  The proof is the same as the proof of Lemma \ref{savings_lemma}. 

\end{proof}

\begin{proof}[Proof of Theorem \ref{spectral_gap_theorem}]
When the boundary is open, let $\gamma' = \min_j \gamma_j$ and let $j$ be minimal such that $\gamma_j = \gamma'$.  By the minimality, there are $B_0, R_0$ such that there is a vector $\nu \in \sC(B_0, R_0)$ with reflection symmetry in a set $S$ of hyperplanes, $|S|=j$ such that $\xi = g *\nu$, and  $\gamma' = f_S(\xi)$. By Proposition \ref{convergence_prop} it is possible to find a sequence $\xi^{(m)} \in \hat{\sG}_m$ such that
\begin{equation}
 \sav(\xi^{(m)}) \to f_S(\xi)
\end{equation}
as $m \to \infty$.

To complete the proof, it suffices to show that  
\begin{equation}
\gamma_0 = \liminf_{m \to \infty} (1+|\sT_m|)\gap_m 
\end{equation}
satisfies $\gamma_0 \geq \gamma$.
Let $\xi^{(m_j)} \in \hat{\sG}_{m_j}$ satisfy $\sav(\xi^{(m_j)}) \to \gamma_0$, and let $\nu^{(m_j)}$ be the sequence of prevectors, which may be assumed to satisfy $\|\nu^{(m_j)}\|_1 \leq B_0$ and $\diam \supp \nu^{(m_j)} \leq R_0$.  Given $R>0$ let $j_R$ be maximal such that infinitely often $\supp \nu^{(m)}$ has distance at most $R$ from $j_R$ boundary hyperplanes.  Let $j^* = \sup_R j_R$, which is achieved for some $R_1$.  Let $r_m \to \infty$ sufficiently slowly so that only finitely many $m_j$ have $\nu^{(m_j)}$ a boundary cluster with more than $j^*$ boundaries.  Take a subsequence $m_{j_i}$ for which $\nu^{(m_{j_i})}$ is a $j^*$ boundary cluster and has distance at most $R_1$ from each of the $j^*$ boundaries.  By Proposition \ref{spectral_gap_prop}, for each $m_{j_i}$ there is a translation of $\nu^{(m_{j_i})}$ to $\nu_{m_{j_i}} \in \sC(B, R) \subset C^\rho(\sT)$ with reflection anti-symmetry in a family $S$ of $j^*$ hyperplanes.  By Proposition \ref{convergence_prop},
\begin{equation}
 \sav(\xi^{(m_{j_i})}) - f_S(g * \nu^{(m_{j_i})}) \to 0
\end{equation}
as $i \to \infty$, which proves that $\gamma_0 = \gamma'$.

In the case of a periodic boundary, let $\{\nu_n\}_n$ be a sequence of functions in $C^1(\sT)$ with $f(g*\nu_n) \to \gamma$.  For each fixed $n$, as $m \to \infty$, $\sav(g_{\bT_m} * \nu_n) \to f(g*\nu_n)$ so $\limsup |\bT_m| \gap_m \leq \gamma$.  To prove the reverse direction, 
let $\xi_{m_k} \in \hat{\sG}_{m_k}$ be a sequence such that $\sav(\xi_{m_k}) \to \liminf |\bT_m|\gap_m$.  Let $\nu_{m_k} = \Delta \xi_{m_k}$, which is integer valued and has sum 0, hence is in $C^1(\bT_m)$.  Perform a clustering on $\nu_{m_k}$.  Arguing as in the proof of Proposition \ref{spectral_constant_prop}, conclude that after eliminating clusters in $\sI$, there are at most two non-empty clusters in $\nu_{m_k}$ for all $k$ sufficiently large.  Since the clusters are of bounded size, after passing to a subsequence we may assume that, up to translation, the two clusters are the same for all $k$ sufficiently large.  In the limit, the savings from the neighborhood of each cluster tends to at least $\gamma_0 - o_R(1)$ as $k \to \infty$.  It follows that if there are two clusters, the lim inf is at least $2 \gamma_0 \geq \gamma$.  If there is only one cluster, the function in the cluster is $C^1$, and hence the liminf is at least $\gamma$.
\end{proof}

\section{Mixing analysis}
An $L^2$ version of Theorem \ref{mixing_theorem} is as follows.

\begin{theorem}\label{L_2_theorem}
For a fixed tiling $\sT$ in $\bR^d$, let $c_0 = \gamma_0^{-1}$.  Let $m \geq 2$ and let $t_m^{\mix} = \frac{c_0}{2} |\bT_m|\log |\bT_m|$.  For each fixed $\epsilon>0$, sandpiles on $\bT_m$ satisfy
\begin{align*}
 \lim_{m \to \infty} \left\|P_m^{\lceil (1-\epsilon) t_m^{\mix}\rceil} \delta_{\sigma_{\full}} - \bU_{\sR_m} \right\|_{L^2(d\bU_{\sR_m})} &= \infty,\\
 \lim_{m \to \infty} \left\|P_m^{\lfloor (1+\epsilon) t_m^{\mix}\rfloor} \delta_{\sigma_{\full}} - \bU_{\sR_m} \right\|_{L^2(d\bU_{\sR_m})} &= 0.
\end{align*}

If the tiling $\sT$ satisfies the reflection condition and condition A  then set $t_m^{\mix} = \frac{\Gamma}{2} |\sT_m| \log m$.  For each fixed $\epsilon > 0$, sandpiles on $\sT_m$ satisfy
\begin{align*}
 \lim_{m \to \infty} \left\|P_m^{\lceil (1-\epsilon) t_m^{\mix}\rceil} \delta_{\sigma_{\full}} - \bU_{\sR_m} \right\|_{L^2(d\bU_{\sR_m})} &= \infty,\\
 \lim_{m \to \infty} \left\|P_m^{\lfloor (1+\epsilon) t_m^{\mix}\rfloor} \delta_{\sigma_{\full}} - \bU_{\sR_m} \right\|_{L^2(d\bU_{\sR_m})} &= 0.
\end{align*}
\end{theorem}

The proof of the lower bound of both the total variation and $L^2$ theorems uses the following Lemma adapted from Diaconis and Shahshahani \cite{DS87}.
\begin{lemma}\label{lower_bound_lemma}
 Let $\sG$ be a finite abelian group, let $\mu$ be a probability measure on $\sG$ and let $N\geq 1$. Let $\sX \subset \hat{\sG} \setminus\{0\}$.  Suppose that the following inequalities hold for some parameters $0 < \epsilon_1, \epsilon_2 < 1$,
 \begin{align}
  \sum_{\xi \in \sX} \left|\hat{\mu}(\xi)\right|^N & \geq \frac{|\sX|^{\frac{1}{2}}}{\epsilon_1}\\
  \notag \sum_{\xi_1, \xi_2 \in \sX} \left|\hat{\mu}(\xi_1 - \xi_2)\right|^N &\leq (1 +\epsilon_2^2)\left( \sum_{\xi \in \sX} \left|\hat{\mu}(\xi)\right|^N\right)^2.
 \end{align}
Then
\begin{equation}
 \left\|\mu^{*N} - \bU_{\sG}\right\|_{\TV(\sG)} \geq 1 - 4\epsilon_1^2 -4\epsilon_2^2.
\end{equation}

\end{lemma}
\begin{proof} See Lemma 27 in \cite{HJL17}.
\end{proof}

\begin{proof}[Proof of Theorem \ref{mixing_theorem}, lower bound]
 First consider the $\bT_m$ case.  Let $\nu \in \sC(B_0, R_0)$ be such that $\xi = g*\nu$ satisfies $f(\xi) = \gamma_0$.  Let $R> R_0$ and let $\{\nu_i^1\}_{i=1}^M$, $\{\nu_i^2\}_{i=1}^M$ be two collections of $R$-separated translates of $\nu$ with those translates $\nu_{i}^1$ having distance $\gg m$ from those translates $\nu_{i}^2$.    Here $M \asymp \frac{m^d}{R^d}$. Let $\nu_{i,j} = \nu_{i}^1 - \nu_{j}^2$. Let $\xi_{i,j} = g* \nu_{i,j}$.
 
 Let $N = \left\lfloor \left( \frac{d}{2}\log m -c \right) |\bT_m| \frac{1}{\gamma_0}\right\rfloor$.  Let $\sX = \{\xi_{i,j}\}_{i,j=1}^{M}$.  By Lemma \ref{separated_clusters_lemma},
\begin{equation}
1- \left|\hat{\mu}(\xi_{i,j})\right| = \frac{2 \gamma_0}{|\bT_m|} + O\left(\frac{\log m}{m^{2\beta}} \right). 
\end{equation}
Thus
\begin{equation}
 \left|\hat{\mu}(\xi_{i,j})\right|^N = e^{2c}m^{-d }\left(1 + O\left(\frac{(\log m)^2}{m^{2\beta -d}} \right)\right).
\end{equation}
It follows that the first condition of Lemma \ref{lower_bound_lemma} holds with $\epsilon_1 = O\left(R^d e^{-2c} \right)$.

If the supports of $\nu_{i_1}^1$ and $\nu_{i_2}^1$ have distance at least $\rho$, and the supports of $\nu_{j_1}^2$ and $\nu_{j_2}^2$ have distance at least $\rho$,
\begin{equation}
 1 - |\hat{\mu}(\xi_{i_1,j_1} - \xi_{i_2, j_2})| = \frac{4 \gamma_0}{|\bT_m|} + O\left(\frac{\log \rho}{\rho^{2\beta -d}m^d} \right),
\end{equation}
and, hence,
\begin{equation}
 \left|\hat{\mu}(\xi_{i_1,j_1}-\xi_{i_2,j_2}) \right|^N = e^{4c} m^{-2d + O\left(\frac{\log \rho}{\rho^{2\beta -d}} \right)}.
\end{equation}

Meanwhile, if $i_1 = i_2$ or $j_1 = j_2$ but the other functions have support at distance at least $\rho$, then
\begin{equation}
 1 - |\hat{\mu}(\xi_{i_1,j_1} - \xi_{i_2, j_2})| = \frac{2 \gamma_0}{|\bT_m|} + O\left(\frac{\log \rho}{\rho^{2\beta -d}m^d} \right)
\end{equation}
and hence
\begin{equation}
\left|\hat{\mu}(\xi_{i_1,j_1}-\xi_{i_2,j_2}) \right|^N = e^{2c} m^{-d + O\left(\frac{\log \rho}{\rho^{2\beta -d}} \right)}. 
\end{equation}

If $R$ is a large enough fixed constant, then
\begin{equation}
 \sum_{\substack{1 \leq i_1, i_2, j_1, j_2 \leq M\\ \min(d(\nu_{i_1}^1, \nu_{i_2}^1), d(\nu_{j_1}^2, \nu_{j_2}^2)) < \log m}} \left|\hat{\mu}(\xi_{i_1, j_1}-\xi_{i_2, j_2})\right|^N \ll \frac{m^{2d}}{R^{2d}}+ e^{4c} O\left(m^{\frac{3d}{2}} \right) .
\end{equation}
Meanwhile
\begin{align*}
 &\sum_{\substack{1 \leq i_1, i_2,j_1, j_2 \leq M\\ \min(d(\nu_{i_1}^1, \nu_{i_2}^1),  d(\nu_{j_1}^2, \nu_{j_2}^2)) \geq \log m}}\left|\hat{\mu}(\xi_{i_1, j_1}-\xi_{i_2, j_2})\right|^N \\&= \sum_{\substack{1 \leq i_1, i_2,j_1,j_2 \leq M\\ d(\nu_{i_1}, \nu_{i_2}), d(\nu_{j_1}, \nu_{j_2}) \geq \log m}} e^{4c}m^{-2d}\left(1 + O\left(\frac{\log \log m}{\log m}\right)\right)\\
 &\leq M^4 \left|\hat{\mu}(\xi)\right|^{4N}\left(1 + O\left(\frac{\log \log m}{\log m}\right)\right).
\end{align*}
Therefore, since $M^4 \left|\hat{\mu}(\xi)\right|^{4N} \asymp e^{4c} \frac{m^{2d}}{R^{4d}}$,
\begin{align}
 &\sum_{1 \leq i_1, i_2, j_1,j_2 \leq M} \left|\hat{\mu}(\xi_{i_1,j_1} - \xi_{i_2,j_2})\right|^N \\&\notag\leq M^4\left|\hat{\mu}(\xi)\right|^{4N}\left(1 +O\left(\frac{\log \log m}{\log m} + \frac{R^{2d}}{e^{4c}} + \frac{1}{m^{\frac{d}{2}}} \right) \right),
\end{align}
and thus the second condition holds with $\epsilon_2 = O\left(R^{d} e^{-2c}\right).$

In the case of an open boundary, let $\Gamma = \Gamma_j$ be maximized at the co-dimension $j$ boundary. Recall $\Gamma_j = \frac{d-j}{\gamma_j}$ and hence, either $j=0$ or $\gamma_j < \gamma_{j-1}$.  In either case, there is a set $S$ of $j$ hyperplanes and a vector $\nu \in \sC(B_0, R_0)$ with reflection anti-symmetry in $S$ such that $\gamma_j = f_S(g*\nu)$.  

Let, as above, $R$ be a large constant, and let $\{\nu_i\}_{i=1}^M$, $M \asymp \frac{m^{d-j}}{R^{d-j}}$ be $R$-spaced translates of $\nu$ parallel to $S$, which are $j$-clusters in $\sT_m$.  Let $\xi_i = g_{\sT_m}*\nu_i$ and $\sX = \{\xi_i\}_{i=1}^M$.  By Proposition \ref{convergence_prop}, for each $i$, uniformly in $m$,
\begin{equation}
 \sav(\xi_i) = (1 + o(1)) \gamma_j
\end{equation}
although, note that the savings may differ across $\sX$.  Given $\epsilon > 0$, let 
\begin{equation}
 N = \left \lfloor(1-\epsilon)\frac{d-j}{2 \gamma_j} |\sT_m| \log m\right \rfloor.
\end{equation}
Hence,
\begin{equation}
 \left|\hat{\mu}(\xi_i)\right|^N = \frac{m^{o(1)}}{m^{\frac{d-j}{2}(1-\epsilon)}}.
\end{equation}
It follows that the condition on $\epsilon_1$ holds with 
\begin{equation}
 \epsilon_1 = \frac{1}{m^{\frac{d-j}{2}\epsilon + o(1)}}.
\end{equation}
Meanwhile, for $\rho > R$, if $d(\nu_i, \nu_j)> \rho$,
\begin{equation}
 1-\left|\hat{\mu}(\xi_i-\xi_j)\right| = 1 - \left|\hat{\mu}(\xi_i)\right| + 1 - \left|\hat{\mu}(\xi_j)\right| + O\left(\frac{\log \rho}{\rho^{2\beta-d}m^d} \right).
\end{equation}
Thus, arguing as before, 
\begin{equation}
 \sum_{\substack{1 \leq i,j \leq M\\ d(\nu_i, \nu_j) < \log m}}\left|\hat{\mu}(\xi_i-\xi_j)\right|^N = O\left(\frac{m^{d-j}}{R^{d-j}} \right) + O\left( m^{\frac{d-j}{2}}\right).
\end{equation}
Meanwhile, 
\begin{align*}
&\sum_{\substack{1 \leq i,j \leq M\\ d(\nu_i, \nu_j) \geq \log m}}\left|\hat{\mu}(\xi_i-\xi_j)\right|^N \\&\leq \left(\sum_{1 \leq i \leq M} \left|\hat{\mu}(\xi_i)\right|^N \right)^2 \left(1 + O\left(\frac{\log \log m}{\log m}\right)\right).
\end{align*}
It follows that
\begin{align*}
 &\sum_{\substack{1 \leq i,j \leq M}}\left|\hat{\mu}(\xi_i-\xi_j)\right|^N\\& \leq \left(\sum_{1 \leq i \leq M} \left|\hat{\mu}(\xi_i)\right|^N \right)^2 \left(1 + O\left(\frac{\log \log m}{\log m}+ \frac{m^{o(1)}}{m^{(d-j)\epsilon}} \right)\right).
\end{align*}
Thus the condition on $\epsilon_2$ is satisfied with $\epsilon_2 = O \left(\frac{\log \log m}{\log m} \right).$
\end{proof}

\begin{proof}[Proof of Theorem \ref{L_2_theorem}, lower bound] By Parseval,
\begin{equation}
 \left\| P_m^{N} \delta_{\sigma_{\full}} - \bU_{\sR_m}\right\|_2^2 = \sum_{0 \neq \xi \in \hat{\sG}_m} \left|\hat{\mu}(\xi)\right|^{2N}.
\end{equation}
 By Cauchy-Schwarz, the condition $\sum_{\xi \in \sX} \left|\hat{\mu}(\xi)\right|^N \geq \frac{|\sX|^{\frac{1}{2}}}{\epsilon_1}$ implies
 \begin{equation}
  \sum_{\xi \in \sX}\left|\hat{\mu}(\xi)\right|^{2N}\geq \frac{1}{\epsilon_1^2}.
 \end{equation}
The theorem thus follows from the previous lower bound.
\end{proof}
\subsection{Proof of upper bound} The upper bound in Theorem \ref{mixing_theorem} is obtained from the upper bound in Theorem \ref{L_2_theorem} by applying Cauchy-Schwarz followed by Parseval.

Let $R = R(\epsilon)$ be a parameter.  In the case of $\bT_m$, let $\xi \in \hat{\sG}_m$, and let $\nu$ be its $R$-reduced prevector.  Perform a clustering on $\nu$ in which points $x_i$, $x_t$ in its support are connected in a cluster if there is a sequence of points $x_i = x_0, x_1, ..., x_n=x_t$ from the support such that $B_R(x_i) \cap B_R(x_{i+1}) \neq \emptyset$.   Let $\sN(V,K)$ denote the number of $R$-reduced prevectors $\nu$ of $L^1$ mass $V$ in $K$ clusters.  In the case of $\sT_m$, let $r_m$ be a radius tending slowly to infinity, $r_m \leq \log m$. Given a set $S$ of bounding hyperplanes, say that a cluster $C$ is of type $S$ if $S$ is a maximal set of hyperplanes such that the cluster intersects the $r_m$ neighborhood of the intersection of the planes in $S$. If there is more than one such maximal $S$, choose one to which $C$ belongs arbitrarily. Let $\sN(V, \{K_S\})$ be the number of $R$-reduced prevectors $\nu$ of $L^1$ mass $V$ with $K_S$ boundary clusters of type $S$.

\begin{lemma}
 The following upper bounds hold:
 \begin{align*}
  \sN(V,K) &\leq \exp\left(K \log (m^d) + O(V\log R) \right)\\
  \sN(V, \{K_S\}) & \leq \exp\left(\sum_S K_S \log\left(m^{d-|S|}r_m^{|S|} \right) + O(V\log R) \right).
 \end{align*}

\end{lemma}
\begin{proof}
The case of $\sN(V, \{K_S\})$ is demonstrated, the other case being similar.  For each of the $K_{S}$ clusters of type $S$, choose base points of the clusters in $O\left(\exp\left(K_S \log \left(Cm^{d-|S|} r_m^{|S|}\right)\right)\right)$ ways. Given a string of length $V$, allocate the vertices to belong to the various clusters in $O(2^{|V|})$ ways by splitting the string at $\sum K_S -1$ places.  For each cluster of size $k$, choose an unlabeled tree on $k$ nodes in $\exp(O(k))$ ways, see \cite{O48} for the asymptotic count.  Traverse the tree from the root down, placing a vertex at distance $O(R^d)$ from its parent vertex.  Now assign the height of each vertex in $O(1)$ ways.  This obtains the claimed bound.
\end{proof}
The upper bound in Theorem \ref{mixing_theorem} follows from the upper bound in Theorem \ref{L_2_theorem} by applying Cauchy-Schwarz, so the upper bound in Theorem \ref{L_2_theorem} is demonstrated.
\begin{proof}[Proof of Theorem \ref{L_2_theorem}, upper bound]
The open boundary case is demonstrated, the periodic boundary case being easier.
 
 Let $N = \left\lceil(1 + \epsilon)\frac{\Gamma}{2}(1+|\sT_m|)\log m \right \rceil$.
Write $\Xi(V,\{K_S\})$ for the collection of nonzero frequencies $\xi \in \hat{\sG}_m$ such that the $R$-reduced prevector of $\xi$ has $L^1$ norm $V$ in $\{K_S\}$ $R$-clusters.   Thus, with $K=|K_S| = \sum_S K_S$,
\begin{equation}
 \left\|P_m^N \delta_{\sigma_{\full}} - \bU_{\sR_m}\right\|_{L^2(d\bU_{\sR_m})}^2 = \sum_{K_S, |K_S| \geq 1} \sum_{V \geq |K_S|} \sum_{\xi \in \Xi(V, \{K_S\})} \left|\hat{\mu}(\xi)\right|^{2N}.
\end{equation}
Let $\Xi(V, K) = \bigcup_{|K_S|=K} \Xi(V, K)$.
It follows from Lemma \ref{size_bound_lemma} that for some $c>0$, for $\xi \in \Xi(V,K)$ satisfies
\begin{equation}
 \left|\hat{\mu}(\xi)\right|^{2N} \leq \exp(-cV \log m).
\end{equation}

Let $A>0$ be a fixed integer constant. Then,
\begin{align}
 & \sum_{K \geq 1} \sum_{V \geq AK} \sum_{\xi \in \Xi(V,K)} \left|\hat{\mu}(\xi)\right|^{2N}\\
 & \notag \leq \sum_{K \geq 1} \sum_{V \geq AK} \sN(V,K) \exp(-cV\log m)\\
 &\notag \leq \sum_{K\geq 1} \sum_{V \geq AK} \exp\left(K \log(m^d) - V[c \log m -O(\log R)] \right).
\end{align}
If $m$ is sufficiently large, the inner sum is bounded by $\ll \exp(-\frac{cAK}{2}\log m)$.  Now choose $A$ large enough so that the sum over $K$ is bounded by $\ll m^{-1}$.

Let $0<\delta<1$ be a parameter and set $B = A \delta^{-1}$.  Apply Lemma \ref{savings_lemma_boundary} to choose $R = R(\epsilon)$ such that the savings from a $j$ boundary $R$ cluster of size at most $B$  is at least $ \gamma_j \left(1-\frac{\epsilon}{2}\right).$ If $\xi \in \Xi(V, K_S)$ with $V < AK$, then its $R$-reduced prevector has at least $(1-\delta)K$ clusters of size at most $B$.  Hence, with $\delta' = \frac{\gamma_0}{\gamma_d}\delta$ assumed to be sufficiently small, 
\begin{equation}
 1- \left|\hat{\mu}(\xi)\right| \geq (1-\delta') \sum_S K_S \gap_{m,j} \left(1 - \frac{\epsilon}{2}\right) \geq \frac{\left(1 - \frac{5\epsilon}{6}\right)}{1 + |\sT_m|}\sum_S K_S \gamma_{|S|}.
\end{equation}
Thus, using $\Gamma \geq \Gamma_j = \frac{d-j}{\gamma_j}$, and using $(1-x) \leq e^{-x}$,
\begin{align*}
 \left|\hat{\mu}(\xi)\right|^{2N} &\leq \exp\left(-(1+\epsilon)\left(1 - \frac{5\epsilon}{6}\right)\Gamma \sum_S K_S \gamma_j \log m \right)\\
 & \leq \exp\left(-(1 +\beta) \Gamma \sum_S K_S \gamma_j \log m\right)\\
 & \leq \exp\biggl(-(1 + \beta) \sum_{S, |S| < d} K_S (d-|S|) \log m \\ & \qquad\qquad- (1+\beta) K_{[d]} \Gamma \gamma_d \log m\biggr)
\end{align*}
where $\beta = \beta(\epsilon) > 0$.

Thus the sum over $V < AK$ is bounded by, for some $c>0$,
\begin{align*}
 &{\sum_{|K_S| \geq 1}} \sum_{|K_S| \leq V < A|K_S|}\sum_{\xi \in \Xi(V,K)}\left|\hat{\mu}(\xi)\right|^{2N}\\
 & \leq {\sum_{|K_S| \geq 1}} \sum_{|K_S| \leq V < A|K_S|}\\&\times \sN(V, \{K_S\}) \exp\left(-(1 + \beta) \sum_{|S|<d} K_S (d-|S|) \log m- c K_{[d]}\log m \right)\\
 &\ll {\sum_{|K_S| \geq 1}} \sum_{|K_S| \leq V < A|K_S|}\exp \biggl(\sum_S \left((-\beta K_S (d-|S|)\log m) + K_S |S| \log r_m\right)\\ &\qquad\qquad - cK_{[d]}\log m +O(|K|)\biggr)\\
 & \ll m^{-\frac{\beta}{2}}.
\end{align*}

\end{proof}

\appendix
\section{Green function estimates on general tilings}\label{Green_fn_appendix}

The Green's function estimates are based on the following local limit theorem for probability measures with exponentially decaying tail on $\zed^d$.

\begin{theorem*}[Theorem \ref{local_limit_theorem}]
 Let $\mu$ be a probability measure on $\zed^d$, satisfying the following conditions
 \begin{enumerate}
  \item (Lazy) $\mu(0)>0$
  \item (Symmetric) $\mu(x) = \mu(-x)$
  \item (Generic) $\supp(\mu)$ generates $\bR^d$.  There is a constant $k > 0$ such that $\mu^{*k}$ assigns positive measure to each standard basis vector.
  \item (Exponential tails) There is a constant $c > 0$ such that, for all $r \geq 1$,
  \begin{equation}
   \mu(|x| > r) \ll e^{-c r}.
  \end{equation}
 \end{enumerate}
 Let $\Cov(\mu) = \sigma^2$ where $\sigma$ is a positive definite symmetric matrix. For all $\ua \in \bN^d$ there is a polynomial  $Q_{\ua}(x_1, ..., x_d)$, depending on $\mu$, of degree at most $a_i$ in $x_i$ such that,   for all $N \geq 1$, and all $n \in \zed^d$,
 \begin{align*}
  \delta_1^{*a_1} * &\delta_2^{*a_2} * \cdots * \delta_d^{*a_d}* \mu^{*N}(n) = \frac{\exp\left(-\frac{\left|\sigma^{-1} \left(n + \frac{\ua}{2}\right) \right|^2}{2N} \right)}{N^{\frac{d + |\ua|}{2}}}\\
  &\times \Biggl(Q_{\ua}\left(\frac{n + \frac{\ua}{2}}{\sqrt{N}} \right)+ O\left(\frac{1}{N}\left(1 +\frac{\|n\|}{\sqrt{N}}\right)^{|\ua|+4} \right)  \Biggr)\\
  &+ O_\epsilon\left(\exp\left( -N^{\frac{3}{8} -\epsilon}\right) \right).
 \end{align*}
In the case of the gradient convolution operator $\nabla = \begin{pmatrix} \delta_1 \\ \vdots \\ \delta_d\end{pmatrix}$,
\begin{align*}
 \nabla \mu^{*N}(n) &= -\frac{\sigma^{-2} n}{N} \frac{\exp\left(-\frac{\|\sigma^{-1}n\|^2}{2N} \right)}{(2\pi)^{\frac{d}{2}} N^{\frac{d}{2}} \det \sigma}\\&+ O\left(\frac{\exp\left(-\frac{\|\sigma^{-1}n\|^2}{2N} \right)}{N^{\frac{d+2}{2}}}\left(1+\frac{\|n\|}{\sqrt{N}} \right)^5\right) + O_\epsilon\left(\exp\left( -N^{\frac{3}{8} -\epsilon}\right) \right).
\end{align*}

\end{theorem*}

\begin{proof}
 If $\|n\|^2 \geq N^{\frac{3}{2}-\frac{\epsilon}{2}}$ apply Chernoff's inequality.  By Fourier inversion,
 \begin{align*}
  &\delta_1^{*a_1} * \delta_2^{*a_2} * \cdots * \delta_d^{*a_d}* \mu^{*N}(n)\\& = (2i)^{|\ua|}\int_{(\bR/\zed)^d}s\left(\frac{x_1}{2}\right)^{a_1}\cdots s\left(\frac{x_d}{2}\right)^{a_d}\hat{\mu}(x)^N e\left(x^t \left(n + \frac{\ua}{2} \right) \right)dx.
 \end{align*}
By symmetry,
\begin{align*}
 \hat{\mu}(x) &= \sum_{n \in \zed^d} \mu(n) c(n \cdot x)\\
 &= 1 -2\pi^2 \sum_{n \in \zed^d} \mu(n) \left(|n \cdot x|^2 +O\left(\|n\|^4 \|x\|^4 \right) \right)\\
 &= 1-2\pi^2 \|\sigma x\|^2 + O(\|x\|^4).
\end{align*}
Since $\mu(0)>0$, and since $\mu^{*k}$ assigns positive measure to each standard basis vector, for each $\delta > 0$ there is $c_1>0$ such that if $\|x\|_{(\bR/\zed)^d} > \delta$ then $\left|\hat{\mu}(x)\right| \leq 1-c_1$.  Combining this observation with Taylor expansion about 0, it follows that there is $c_2 > 0$ such that $\left|\hat{\mu}(x)\right| \leq 1 - c_2 \|x\|_{(\bR/\zed)^d}^2$. Using this, truncate to, for some $c_3 >0$, $\|x\|_{(\bR/\zed)^d} \leq c_3 N^{-\frac{1}{4}}$.  

Write the remaining part of the integral as
\begin{align*}
 &(2i)^{|\ua|} \int_{\|x\| \leq c_3N^{-\frac{1}{4}}}s\left(\frac{x_1}{2} \right)^{a_1} \cdots s\left(\frac{x_d}{2} \right)^{a_d}\\&\exp\left(-2\pi^2 N \|\sigma x\|^2  + 2\pi i x^t \left(n + \frac{\ua}{2}\right)+ O\left(N\|x\|^4\right)  \right)dx.
\end{align*}
Write the main term in the exponentials as
\begin{align*}
 &-\frac{1}{2}\left(2\pi \sqrt{N} \sigma x - i \frac{\sigma^{-1}\left(n + \frac{\ua}{2}\right)}{\sqrt{N}} \right)^t \left(2\pi \sqrt{N} \sigma x - i \frac{\sigma^{-1}\left(n + \frac{\ua}{2}\right)}{\sqrt{N}} \right)\\& - \frac{1}{2} \frac{\left\|\sigma^{-1}\left(n + \frac{\ua}{2} \right) \right\|^2}{N}.
\end{align*}

Substitute
\begin{equation}
 y = x - i \frac{\sigma^{-2}\left(n + \frac{\ua}{2} \right)}{2\pi N}.
\end{equation}
Shift the integrals in the complex plane so that
\begin{equation}
 2\pi \sqrt{N} \sigma \left(x - i \frac{\sigma^{-2}\left(n + \frac{\ua}{2} \right)}{2\pi N} \right) = 2\pi \sqrt{N}\sigma y
\end{equation}
becomes real.  This introduces an integral on $\|\RE(x)\| = c_3 N^{-\frac{1}{4}}$ on which $\|\IM(x)\| \leq N^{-\frac{1}{4} - \frac{\epsilon}{4}}$.  Throughout this integral the integrand is bounded by $\exp(-c_4 N^{\frac{1}{2}})$ so this contributes an error term.  On the shifted integral, write 
\[\exp(O(N\|x\|^4)) = 1 + O(N\|x\|^4) = 1+ O\left(N\|\RE(x)\|^4 + \frac{\|n\|^4}{N^3}\right).\]  Write, by Taylor expansion,
\[s\left(\frac{x_1}{2} \right)^{a_1} \cdots s\left(\frac{x_d}{2}\right)^{a_d} = \pi^{|\ua|}x_1^{a_1}\cdots x_d^{a_d}\left(1 + O\left(\|\RE(x)\|^2 + \frac{\|n\|^2}{N^2} \right)\right).\]  
Bound each term $\left|\sigma^{-2}\left(n  + \frac{\ua}{2}\right)_j\right|\ll \|n\|$.  In integrating away this error, each factor of $|x_j|$ contributes a term of order $\frac{1}{\sqrt{N}}$, which obtains a bound of
\begin{equation}
 \ll \exp\left(- \frac{1}{2} \frac{\left\|\sigma^{-1}\left(n + \frac{\ua}{2} \right) \right\|^2}{N}\right)\frac{1 + \left(\frac{\|n\|}{\sqrt{N}} \right)^{|\ua|+4}}{N^{\frac{d+|\ua|+2}{2}}}. 
\end{equation}
Note that the error of size $\frac{\|n\|^2}{N^2}$ is bounded by $\frac{1}{N} + \frac{\|n\|^4}{N^3}$.

Write the main term as
\begin{align*}
&(2\pi i)^{|\ua|} \exp\left(- \frac{1}{2} \frac{\left\|\sigma^{-1}\left(n + \frac{\ua}{2} \right) \right\|^2}{N}\right)
\\& \int_{\|y\| \leq c_3 N^{-\frac{1}{4}}} \prod_{j=1}^d\left(y_j + i \left(\frac{\sigma^{-2} \left(n + \frac{\ua}{2} \right)}{2\pi N} \right)_j \right)^{a_j}  \exp\left( -2\pi^2 N\|\sigma y\|^2\right)dy.
\end{align*}
Extend the integral to $\bR^d$ with negligible error, and substitute $z = 2\pi \sqrt{N}\sigma y$ to obtain
\begin{align*}
 &\frac{\exp\left(- \frac{1}{2} \frac{\left\|\sigma^{-1}\left(n + \frac{\ua}{2} \right) \right\|^2}{N}\right)}{(2\pi)^d N^{\frac{d+|\ua|}{2}} \det \sigma}\\
 &\int_{\bR^d} \exp\left(- \frac{\|z\|^2}{2} \right) \prod_{j=1}^d \left(i(\sigma^{-1}z)_j - \left(\frac{\sigma^{-2}\left(n + \frac{\ua}{2} \right)}{\sqrt{N}} \right)_j \right)^{a_j} dz.
\end{align*}

 This produces the claimed main term. Note that only terms with even powers of $z$ are preserved by the integral, which proves the formula for the gradient.
\end{proof}



Let $\eta$ be a function of bounded support on $\sT$.  Let $\varrho_\eta$ be the signed measure on $\Lambda$ obtained by starting simple random walk from $\eta$ and stopping it on the first non-negative step at which it visits $\Lambda$.  

\begin{lemma}\label{measure_decomp_lemma}
 Let $\sT$ be a tiling in $\bR^d$. There is a constant $c = c(\eta)>0$ such that the following holds.  If $\eta \in C^0(\sT)$ then there is $c>0$ such that $\|\varrho_\eta(x)\| \ll e^{-c\|x\|}$. 
 
 If $\eta \in C^1(\sT)$ then there are functions $f_1, f_2, ..., f_d$ on $\Lambda$ such that \[\varrho_\eta = \sum_{j=1}^d f_j * \delta_j\] and satisfying $|f_j(x)| \leq e^{-c\|x\|}$. 
 
 If $\eta \in C^2(\sT)$ then there are functions $f_{i,j}$, $1 \leq i \leq j \leq d$ on $\Lambda$ such that 
 \[
 \varrho_\eta = \sum_{1 \leq i \leq j \leq d} f_{i,j} * \delta_i * \delta_j
 \]
and satisfying $|f_{i,j}(x)| \ll e^{-c\|x\|}$.
\end{lemma}

\begin{proof}
 In the case that $\eta \in C^0(\sT)$, the exponential decay condition follows from the fact that the stopped random walk has a distribution with exponentially decaying tails.
 
 To prove the two remaining claims, given a radius $R$, let $\varrho_{\eta, R}$ denote the measure $\varrho_\eta$ restricted to $\|x\| \leq R$.  Due to the exponentially decaying tails, the mass of this measure is exponentially small in $R$, and if $\eta$ is $C^2$, the moment is exponentially small in $R$.  Hence, it follows that there is a bounded measure $\nu_{\eta, R}$ of total mass exponentially small in $R$, such that $\varrho'_{\eta, R} = \varrho_{\eta, R} + \nu_{\eta, R}$ has the same regularity as $\varrho_\eta$.  Since $\varrho'_{\eta, R} \in C^j(\Lambda)$, it can be written as a sum of translates of $\{\delta_{i}\}_{1 \leq i \leq d}$ if $\eta \in C^1(\sT)$ or $\{\delta_{i,j}\}_{1 \leq i \leq j \leq d}$ of $\eta \in C^2(\sT)$.  Arrange this sum such that $\varrho'_{\eta, 2R} - \varrho'_{\eta, R}$ is the sum of $O(e^{-c'R})$ translates, which is easily achieved in the case of $C^1(\sT)$ by balancing each function value in the support with an opposing value at the origin, the total number of $\delta_i$ needed to achieve this being $O(R)$.  In the case of $C^2(\sT)$, first write the difference of a sum of translates of $\delta_i$, then balance each $\delta_i$ with a corresponding term at the origin, the number needed for a single on again being $O(R)$.  The polynomial growth is now dominated by the exponential decay of $\varrho'$.  Letting $R \to \infty$ obtains the required decomposition.
\end{proof}

Let $\sT$ be a tiling, periodic with period $\Lambda$, which is identified with $\zed^d$ via a choice of basis. Assume $0 \in \sT$.  Let $\varrho$ be the measure obtained by stopping the random walk started at 0 at its first return to $\Lambda$. Let $\varrho_{\frac{1}{2}} = \frac{1}{2}(\varrho + \delta_0)$ be the half-lazy version of $\varrho$.  Given $m \geq 1$, let $\varrho_{\bT_m}(x) = \varrho(x + m \Lambda)$ and $\varrho_{\frac{1}{2}, \bT_m}(x) = \varrho_{\frac{1}{2}}(x + m \Lambda).$

\begin{lemma}
 Let $\sT$ be a tiling of $\bR^d$ which is periodic in lattice $\Lambda \cong \zed^d$.  The Green's function of  $\sT$ started from 0 is given on $\Lambda$ by, in dimension $d=2$,
 \begin{equation}
  g_{0}(n) = \frac{1}{2 \deg(0)}\sum_{N=0}^\infty \varrho_{\frac{1}{2}}^{*N}(n) - \varrho_{\frac{1}{2}}^{*N}(0),
 \end{equation}
and in dimension $d \geq 3$ by
\begin{equation}
 g_{0}(n) = \frac{1}{2 \deg(0)} \sum_{N=0}^\infty \varrho_{\frac{1}{2}}^{*N}(n).
\end{equation}
For all $m$ sufficiently large, on $\sT/m\Lambda$, when restricted to $\Lambda/m\Lambda$, the Green's function is given by
\begin{equation}
 g_{0, \bT_m}(n) = \frac{1}{2\deg(0)} \sum_{N=0}^\infty \left(\varrho_{\frac{1}{2}, \bT_m}^{*N}(n) - \frac{1}{m^d}\right).
\end{equation}

\end{lemma}

\begin{proof}
This is very similar to the proof of Lemma 29 in \cite{HJL17}.  Recall that, in dimension 2,
\begin{equation}
 g_{0}(n) = \frac{1}{\deg(0)} \sum_{N=0}^\infty \varrho^{*N}(n) - \varrho^{*N}(0),
\end{equation}
and in dimension at least 3,
\begin{equation}
 g_{0}(n) = \frac{1}{\deg(0)} \sum_{N=0}^\infty \varrho^{*N}(n).
\end{equation}
Since, by definition, $\varrho^{*2}(0) > 0$, the measure $\varrho^{*2}$ satisfies the conditions of the local limit theorem above, see also \cite{LL10}.  Thus after taking consecutive odd and even terms together, the sums converge absolutely.

Expanding by the binomial theorem, in the $d = 2$ case,
\begin{align*}
&\frac{1}{2\deg(0)} \sum_{N=0}^\infty \left(\varrho_{\frac{1}{2}}^{*N}(n) - \varrho_{\frac{1}{2}}^{*N}(0) \right)\\
&= \frac{1}{2\deg(0)} \sum_{N=0}^\infty \frac{1}{2^N} \left(\sum_{k=0}^N \binom{N}{k} \left(\varrho_{\frac{1}{2}}^{*k}(n) - \varrho_{\frac{1}{2}}^{*k}(0) \right)\right)\\
&= \frac{1}{2\deg(0)}\sum_{k=0}^\infty \left(\varrho_{\frac{1}{2}}^{*k}(n) - \varrho_{\frac{1}{2}}^{*k}(0) \right) \sum_{N=k}^\infty \binom{N}{k} 2^{-N}.
\end{align*}
The inner sum evaluates to 2, from the identity $\left(\frac{1}{1-x} \right)^k = \sum_{N=0}^\infty \binom{N+k}{k} x^N$, which proves the first claim.  The claim in dimensions $d \geq 3$ is similar.  

To prove the identity on $\Lambda/m\Lambda$, expand both sides in characters of the group.
\end{proof}

The following lemma gives decay estimates for the Green's function on $\sT$.

\begin{lemma}\label{greens_fn_decay_lemma_T}
 Let $\sT$ be a tiling in $\bR^d$ with period lattice $\Lambda$, and let $\eta$ be a function on $\sT$ of bounded support.  Let $g_\eta = g*\eta$.  If $\eta \not \in C^1(\sT)$, for $x \in \Lambda$,
 \[
  g_\eta(x) \ll \left\{ \begin{array}{ccc} \log (2 + \|x\|)&& d= 2\\ \frac{1}{(1 + \|x\|)^{d-2}} && d \geq 3 \end{array}\right..
 \]
If $D^{\ua} = \delta_1^{*a_1} *... * \delta_d^{*a_d}$ is a discrete differential operator and $\|\ua\| + j \geq 1$, then
\[
 D^{\ua} g_\eta(x) \ll \frac{1}{(1 + \|x\|)^{d + |\ua|+j-2}} .
\]

\end{lemma}

\begin{proof}
The claims are first proved for the Green's function $g_0$ started at 0.  
 In this case, the claims regarding the Green's function itself were proved in Lemma \ref{green_fn_growth_lemma}.  To prove the claims regarding the discrete derivative, write
 \begin{equation}
  D^{\ua} g_0 (x) = \frac{1}{2 \deg 0} \sum_{n=0}^\infty D^{\ua} \varrho_{\frac{1}{2}}^{*N}(x).
 \end{equation}
For $N < \frac{\|x\|^2}{(1 + \log (2 + \|x\|))^2}$, Chernoff's inequality implies that $\varrho_{\frac{1}{2}}^{*N}(x) = O_A((1 + \|x\|)^{-A})$, so that this part of the sum may be ignored.  In the remaining part of the sum, the local limit theorem obtains, for some $c > 0$,
\[
 D^{\ua} \varrho_{\frac{1}{2}}^{*N}(x) \ll \frac{\exp\left(- c\frac{\|x\|^2}{N} \right)}{N^{\frac{d+|\ua|}{2}}}.
\]
Summed in $N$, this obtains the bound claimed.

Now, given $\eta$, if $\eta \not \in C^0(\sT)$, write on $\Lambda$, $g*\varrho_\eta = g_\eta$,
\begin{align*}
 g_\eta (x) &= \sum_{y \in \Lambda} \varrho_\eta(y) g_0(x-y)
  \ll \sum_{y \in \Lambda} e^{-c \|y\|} |g_0(x-y)|.
\end{align*}
Due to the bound for $g_0$, $y$ may be truncated at $\|y\| \ll (1 + \log (2 + \|x\|))$, from which the claim follows.  The proof in case of $C^j$ for $j = 1, 2$ is similar, by writing $\varrho_\eta$ as a sum of translates of first or second derivative operators.

\end{proof}

The remaining lemmas  obtain analogues of the decay estimates for $Dg$ on $\sT$ in the setting of the periodic case $\bT_m$.  This is accomplished by a split space-frequency representation on $\bT_m$ in which small convolutions $\varrho^{*N}$, which are localized in space, are treated in space domain, and large values of $\varrho^{*N}$ are treated in frequency domain.

\begin{lemma}\label{gradient_lemma}
Let $\sT$ be a tiling of $\bR^d$ with periodic lattice $\Lambda$ identified with $\zed^d$ by a choice of basis. Let  $\sigma^2 = \Cov(\varrho)$.   For $m \geq 1$ and for $1\leq \|x\|_{(\zed/m\zed)^d} \ll \left(\frac{m^2}{\log m} \right)^{\frac{d-1}{2d}}$, 
\begin{equation}
 \nabla g_{0, \bT_m}(x) = - \frac{\Gamma\left(\frac{d}{2}\right)\sigma^{-2}x }{ \deg(0)\pi^{\frac{d}{2}} \|\sigma^{-1} x\|^d \det \sigma}  + O\left(\frac{1}{\|\sigma^{-1} x\|^d} \right).
\end{equation}

\end{lemma}
\begin{proof}
 Let $T = \frac{C m^2}{ \log m}$ for a constant $C>0$. Let $\sigma_{\frac{1}{2}}^2 = \Cov(\rho_{\frac{1}{2}})$, so $\sigma_{\frac{1}{2}} = \frac{1}{\sqrt{2}} \sigma$.  Let $R =2+\left\|\sigma_{\frac{1}{2}}^{-1} x\right\|$.  Then
 \begin{align*}
  \nabla g_{0, \bT_m}(x) &= \frac{1}{2 \deg(0)} \sum_{n \in \zed^d} \sum_{0 \leq N <T} \begin{pmatrix} \delta_1 \\ \delta_2 \\ \vdots \\ \delta_d\end{pmatrix} *\varrho_{\frac{1}{2}}^{*N}( x + m n)\\
  &+ \frac{1}{2 \deg(0)} \sum_{T \leq N} \begin{pmatrix} \delta_1 \\ \delta_2 \\ \vdots \\ \delta_d\end{pmatrix} *\varrho_{\frac{1}{2}, \bT_m}^{*N}( x ).
 \end{align*}
In the first sum, by applying Chernoff's inequality, those terms with $n \neq 0$ contribute an acceptable error term if $C$ is sufficiently small.  Similarly, discard those terms with $N \ll \frac{R^2}{\log R}$ as an error term. Applying the local limit theorem, the first term has a main term, 
\begin{equation} \frac{1}{2\deg(0)}\sum_{\frac{R^2}{\log R} \ll N \leq T}\left(- \frac{\sigma_{\frac{1}{2}}^{-2}x}{(2\pi)^{\frac{d}{2}}\det \sigma_{\frac{1}{2}}}\frac{\exp\left(-\frac{\left\|\sigma_{\frac{1}{2}}^{-1} x\right\|^2}{2N} \right)}{N^{\frac{d}{2}+1}} \right).
\end{equation}
The error is bounded by 
\begin{equation}
 \ll O_A(R^{-A}) + \sum_{\frac{R^2}{\log R} \ll N \leq T} \frac{\exp\left(-\frac{\left\|\sigma_{\frac{1}{2}}^{-1}x\right\|^2}{2N}\right)}{N^{\frac{d+2}{2}}}\left(1 + \frac{\left\|\sigma_{\frac{1}{2}}^{-1}x\right\|}{\sqrt{N}}\right) \ll \frac{1}{\left\|\sigma_{\frac{1}{2}}^{-1}x\right\|^d}. 
\end{equation}

The sum may be replaced with an integral,
\begin{align*}
 &\left(1 + O\left(\frac{1}{R}\right)\right)\frac{1}{2 \deg(0)}\left(- \frac{\sigma^{-2}_{\frac{1}{2}}x}{\pi^{\frac{d}{2}} \left\|\sigma_{\frac{1}{2}}^{-1} x\right\|^d \det \sigma_{\frac{1}{2}}} \right) \int_{\frac{c}{\log R}}^{\frac{2T}{R^2}}  \frac{\exp(-1/x)}{x^{\frac{d}{2}}} \frac{dx}{x}\\
 &= \left(1 + O\left(\frac{1}{R}\right)\right)\left(- \frac{\Gamma\left(\frac{d}{2}\right)\sigma_{\frac{1}{2}}^{-2}x }{2 \deg(0)\pi^{\frac{d}{2}} \left\|\sigma^{-1}_{\frac{1}{2}} x\right\|^d \det \sigma_{\frac{1}{2}}} \right).
\end{align*}
By Fourier inversion on the group $(\zed/m\zed)^d$, the tail of the sum is given by
\begin{align*}
 \frac{1}{2\deg(0)} \frac{1}{m^d} \sum_{0 \neq \xi \in (\zed/m\zed)^d} \begin{pmatrix}  e\left(\frac{\xi_1}{m}\right)-1 \\ \vdots \\  e\left(\frac{\xi_d}{m}\right)-1\end{pmatrix}\frac{\left(\frac{1}{2} + \frac{1}{2}\hat{\varrho}\left(\frac{\xi}{m} \right) \right)^T}{1-\hat{\varrho}\left(\frac{\xi}{m}\right)}e\left(\frac{\xi \cdot x}{m} \right).
\end{align*}
This is bounded in norm by, for some $c>0$,
\begin{align*}
 &\ll \frac{1}{m^d} \sum_{0 \neq \xi \in (\zed/m\zed)^d}\frac{\left(1-c \frac{\|\xi\|^2}{m^2} \right)^T}{\frac{\|\xi\|}{m}}\\&\ll  \int_{\left(\bR/\zed\right)^d} \frac{\exp(-cT\|x\|^2)}{\|x\|}dx\\&\ll \int_0^\infty \exp(-cTr^2) r^{d-1} \frac{dr}{r} \ll T^{-\frac{d-1}{2}}.
\end{align*}
The claimed asymptotic holds, since $T \gg R^{\frac{2d}{d-1}}$.
\end{proof}

\begin{lemma}\label{G_0_asymp_lemma}
Let $\sT$ be a tiling of $\bR^d$, $d > 2$ with periodic lattice $\Lambda$ identified with $\zed^d$ by a choice of basis. Set $\sigma^2 = \Cov(\varrho)$. For $m \geq 1$ and for $1\leq \|x\|_{(\zed/m\zed)^d} \ll \left(\frac{m^2}{\log m} \right)^{\frac{d-2}{2(d-1)}}$, 
\begin{equation}
  g_0(x) =  \frac{\Gamma\left(\frac{d}{2}\right) }{2 \deg(0)(\pi)^{\frac{d}{2}} \|\sigma^{-1} x\|^{d-2} \det \sigma}  + O\left(\frac{1}{\|\sigma^{-1} x\|^{d-1}} \right).
\end{equation}

\end{lemma}
\begin{proof}
 Let $T = \frac{C m^2}{ \log m}$ for a constant $C>0$. Let $\sigma_{\frac{1}{2}}^2 = \Cov(\rho_{\frac{1}{2}})$, so that $\sigma_{\frac{1}{2}} = \frac{1}{\sqrt{2}} \sigma$. Let $R =2+ \left\|\sigma_{\frac{1}{2}}^{-1} x\right\|$. Write
 \begin{align*}
   g_0(x) &= \frac{1}{2 \deg(0)} \sum_{n \in \zed^d} \sum_{0 \leq N <T} \varrho_{\frac{1}{2}}^{*N}( x + m n)\\
  &+ \frac{1}{2 \deg(0)} \sum_{T \leq N} \varrho_{\frac{1}{2}, \bT_m}^{*N}( x ).
 \end{align*}
In the first sum, by applying Chernoff's inequality, those terms with $n \neq 0$ contribute an acceptable error term if $C$ is sufficiently small.  Similarly, discard those terms with $N \ll \frac{R^2}{\log R}$ as an error term. Applying the local limit theorem, the first term becomes, with error $O_A(R^{-A})$,
\begin{equation} \frac{1}{2\deg(0)}\sum_{\frac{R^2}{\log R} \ll N \leq T}\left( \frac{\exp\left(-\frac{\left\|\sigma_{\frac{1}{2}}^{-1} x\right\|^2}{2N} \right)}{(2\pi)^{\frac{d}{2}}\det \sigma_{\frac{1}{2}} N^{\frac{d}{2}}} \right)\left(1 + O\left(\frac{1}{R}\right)\right).
\end{equation}
With the same relative error the sum may be replaced with an integral,
\begin{align*}
 &\left(1 + O\left(\frac{1}{R}\right)\right)\frac{1}{ 4\deg(0)}\left( \frac{1}{\pi^{\frac{d}{2}} \left\|\sigma_{\frac{1}{2}}^{-1} x\right\|^{d-2} \det \sigma_{\frac{1}{2}}} \right) \int_{\frac{c}{\log R}}^{\frac{2T}{R^2}}  \frac{\exp(-1/x)}{x^{\frac{d}{2}-1}} \frac{dx}{x}\\
 &= \left(1 + O\left(\frac{1}{R}\right)\right)\left( \frac{\Gamma\left(\frac{d}{2}-1\right) }{ 4\deg(0)\pi^{\frac{d}{2}} \left\|\sigma_{\frac{1}{2}}^{-1} x\right\|^{d-2} \det \sigma_{\frac{1}{2}}} \right).
\end{align*}
The main term can be obtained by using $\sigma_{\frac{1}{2}} = \frac{1}{\sqrt{2}} \sigma$.
By Fourier inversion on the group $(\zed/m\zed)^d$, the tail of the sum is given by
\begin{align*}
 \frac{1}{2\deg(0)} \frac{1}{m^d} \sum_{0 \neq \xi \in (\zed/m\zed)^d} \frac{\left(\frac{1}{2} + \frac{1}{2}\hat{\varrho}\left(\frac{\xi}{m} \right) \right)^T}{1-\hat{\varrho}\left(\frac{\xi}{m}\right)}e\left(\frac{\xi \cdot x}{m} \right).
\end{align*}
This is bounded in norm by, for some $c>0$,
\begin{align*}
 &\ll \frac{1}{m^d} \sum_{0 \neq \xi \in (\zed/m\zed)^d}\frac{\left(1-c \frac{\|\xi\|^2}{m^2} \right)^T}{\left(\frac{\|\xi\|}{m}\right)^2}\\&\ll  \int_{\left(\bR/\zed\right)^d} \frac{\exp(-cT\|x\|^2)}{\|x\|^2}dx\\&\ll \int_0^\infty \exp(-cTr^2) r^{d-2} \frac{dr}{r} \ll T^{-\frac{d-2}{2}}.
\end{align*}
The claimed asymptotic holds, since $T \gg R^{\frac{2(d-1)}{d-2}}$.
\end{proof}

\begin{lemma}\label{deriv_decay_lemma_T_m}
 Keep the notation of the previous lemma.  The discrete derivatives satisfy, for any $\ua \in \bN^d$, $|\ua| \geq 1$, and for all $x \in \Lambda$, 
 \begin{equation}
  D^{\ua} g_{0, \bT_m}(x) \ll_{\ua} \frac{1}{1 + \|x\|_{(\zed/m\zed)^d}^{|\ua| + d-2}}.
 \end{equation}

\end{lemma}
\begin{proof}
Assume that among the representatives of $x \bmod m\zed^d$, $\|x\|$ is minimal. Let $R = 2 + \left\|\sigma^{-1}x \right\|_{(\zed/m\zed)^d}$.  Split the sum as
\begin{align*}
 D^{\ua} g_{0, \bT_m}(x) &=\frac{1}{2\deg(0)} \sum_{n \in \zed^d}\sum_{0 \leq N < R^2} \delta_1^{*a_1}*\cdots*\delta_d^{*a_d}\varrho_{\frac{1}{2}}^{*N}(x + m n)\\
 &+ \frac{1}{2\deg(0)} \sum_{N > R^2} \delta_1^{*a_1}*\cdots*\delta_d^{*a_d}\varrho_{\frac{1}{2}, \bT_m}^{*N}(x).
\end{align*}
In the first sum, use Chernoff's inequality to discard those terms with $N \ll \frac{R^2}{\log R}$, and those terms with $m^2\|n\|^2 \gg R^2 \log R$.  

By the local limit theorem, the first sum is bounded by
\begin{align*}
 \ll \sum_{\frac{R^2}{\log R} \ll N \leq R^2} \sum_{n \in \zed^d} \frac{\exp\left(- \frac{\left\|\sigma^{-1}\left(x + m n \right)\right\|^2}{2N} \right)}{N^{\frac{d + |\ua|}{2}}}\left(1 + \frac{\|x + mn\|}{\sqrt{N}} \right)^{|\ua|}.
\end{align*}
By the exponential decay, the sum over $n$ is bounded by a constant times the $n = 0$ term.  Meanwhile, the sum over those terms with $n = 0$ is bounded by $\ll \frac{1}{1 + \|x\|^{|\ua| + d-2}}$.

Expanding the tail of the sum in characters, and bounding the sum in absolute value, it is bounded by
\begin{align*}
 &\ll \frac{1}{m^d} \sum_{0 \neq \xi \in (\zed/m\zed)^d} \prod_{j=1}^d \left|1 - e\left(\frac{\xi_j}{m} \right)\right|^{a_j} \frac{\left|\frac{1 + \hat{\varrho}\left(\frac{\xi}{m}\right)}{2} \right|^{R^2}}{1 - \left|\hat{\varrho}\left(\frac{\xi}{m}\right)\right|}\\
 &\ll \int_{(\bR/\zed)^d} \frac{\prod_{j=1}^d |\xi_j|^{a_j}}{\|\xi\|^2}\exp(-c R^2 \|\xi\|^2) d\xi\\
 &\ll \int_0^\infty e^{-cr^2 R^2} r^{|\ua| + d-2} \frac{dr}{r}\\
 &\ll \frac{1}{R^{|\ua| + d -2}}.
\end{align*}

\end{proof}

The remaining lemmas treat the convolution of the Green's function with a measure $\eta$ of bounded support on the tiling $\sT$.  Note that the estimates are stated for the argument in the lattice $\Lambda$, but the regularity of $\eta$  is invariant under translating $\sT$, which permits recovering estimates for all $t \in \sT$.

\begin{lemma*}[Lemma \ref{deriv_bound_lemma}]
Let $\sT$ be a tiling of $\bR^d$ which is $\Lambda \cong \zed^d$ periodic.
 Let $\eta$ be of class $C^j(\sT)$ for some $0 \leq j \leq 2$.  Let $D^{\ua}$ be a discrete differential operator on the lattice $\Lambda$ and assume that $|\ua| + j + d - 2 > 0$.  For $m \geq 1$, for $x \in \Lambda$, 
 \begin{equation}
  D^{\ua} g_{\eta, \bT_m} (x) \ll \frac{1}{1 + \|x\|_{(\zed/m\zed)^d}^{|\ua| + j + d -2}}.
 \end{equation}

\end{lemma*}

\begin{proof}
Assume without loss of generality that $x \in \Lambda$ satisfies $\|x\| = \|x\|_{(\zed/m\zed)^d}$, which can be assumed to be larger than any fixed constant.

 Let $\varrho_\eta$ be the signed measure on $\sT$ obtained by stopping random walk started from $\eta$ at the first time that it reaches $\Lambda$.  Thus, for $x \in \Lambda/m\Lambda$, $g_{\eta, \bT_m}(x) = g_{\bT_m} * \varrho_\eta$.
 
 In the case $j = 0$, bound, using Lemma \ref{deriv_decay_lemma_T_m},
 \begin{align*}
 D^{\ua} g_{\eta, \bT_m}(x) &= \sum_{y \in \Lambda} \varrho_\eta(y) D^{\ua}g_{0,\bT_m}(x-y) \\&\ll \sum_{y \in \Lambda}e^{-c\|y\|} \frac{1}{1 + \|x-y\|_{(\zed/m\zed)^2}^{d+|\ua|-2}} \\&\ll \frac{1}{1 + \|x\|_{(\zed/m\zed)^2}^{d + |\ua|-2}}. 
 \end{align*}
The last estimate holds by splitting on $\|y\| \ll \log \|x\|$, and bounding the values of $\frac{1}{1 + \|x-y\|_{(\zed/m\zed)^2}^{d+|\ua|-2}}$ with larger $\|y\|$ by a constant.
 In the case $j= 1$, by Lemma \ref{measure_decomp_lemma} write $\varrho_\eta = \sum_{i = 1}^d f_i * \delta_i$.  Then
 \begin{align*}
  D^{\ua} g_{\eta, \bT_m}(x) &= \sum_{i=1}^d\sum_{y \in \Lambda} f_i(y) D^{\ua}\delta_i *g_{0,\bT_m}(x-y) \\&\ll \sum_{i=1}^d\sum_{y \in \Lambda}e^{-c\|y\|} \frac{1}{1 + \|x-y\|_{(\zed/m\zed)^2}^{d+|\ua|-1}} \\&\ll \frac{1}{1 + \|x\|_{(\zed/m\zed)^2}^{d + |\ua|-1}}.
 \end{align*}

 In the case $j = 2$, by Lemma \ref{measure_decomp_lemma} write $\varrho_\eta = \sum_{1 \leq i \leq j \leq d} f_{i,j}*\delta_i*\delta_j$.  Then
 \begin{align*}
    D^{\ua} g_{\eta, \bT_m}(x) &= \sum_{1 \leq i\leq j \leq d}\sum_{y \in \Lambda} f_{i,j}(y) D^{\ua}\delta_i *\delta_j*g_{0,\bT_m}(x-y) \\&\ll \sum_{1 \leq i \leq j \leq d}\sum_{y \in \Lambda}e^{-c\|y\|} \frac{1}{1 + \|x-y\|_{(\zed/m\zed)^2}^{d+|\ua|}} \\&\ll \frac{1}{1 + \|x\|_{(\zed/m\zed)^2}^{d + |\ua|}}.
 \end{align*}

\end{proof}

\begin{lemma*}[Lemma \ref{G_eta_eval}]
  Let $\sT$ be a tiling of $\bR^d$ with period lattice $\Lambda$ identified with $\zed^d$ via a choice of basis. Let $\sigma^2 = \Cov(\varrho)$.  Let $\eta$ be  of class $C^1(\sT)$, and let $\varrho_\eta$ be the signed measure on $\Lambda$ obtained by stopping simple random walk on $\sT$ started from $\eta$ when it reaches $\Lambda.$ Let $\varrho_\eta$ have  mean $v$. For $n \in \Lambda$, $1 \leq \|n\| \ll \left(\frac{m^2}{\log m}\right)^{\frac{d-1}{2d}}$,
 \begin{equation}
  g_{\eta, \bT_m}(n) = \frac{\Gamma\left(\frac{d}{2} \right) v^t \sigma^{-2}n}{ \deg(0)\pi^{\frac{d}{2}}\|\sigma^{-1}n\|^d \det \sigma} + O\left(\frac{1}{\|\sigma^{-1}n\|^d} \right).
 \end{equation}
If $d \geq 3$ and $\eta \not \in C^1(\sT)$ has total mass $C$,
\begin{equation}
 g_{\eta, \bT_m}(n) = \frac{C\Gamma\left(\frac{d}{2} -1\right) }{2 \deg(0)\pi^{\frac{d}{2}}\|\sigma^{-1}n\|^{d-2} \det \sigma} + O\left(\frac{1}{\|\sigma^{-1}n\|^{d-1}} \right).
\end{equation}

\end{lemma*}

\begin{proof}
 Let $h = \sum_{i = 1}^d h_i  \delta_i = -v^t \cdot \nabla$ be a sum of first derivative operators which has the same mean as $\varrho_\eta$.  The difference $\eta-h$ is $C^2$, hence by the previous lemma\begin{equation}g_{\bT_m}*(\eta-h)(n) \ll \frac{1}{\|\sigma^{-1}n\|^d}.\end{equation}
 For the measure $h$, by Lemma \ref{gradient_lemma},
 \begin{equation}
  (g_{\bT_m}*h)(n) =  \frac{\Gamma\left(\frac{d}{2} \right) v^t \sigma^{-2}n}{ \deg(0)\pi^{\frac{d}{2}}\|\sigma^{-1}n\|^d \det \sigma} + O\left(\frac{1}{\|\sigma^{-1}n\|^d} \right).
 \end{equation}
The second claim follows similarly, by choosing $h$ to be a point mass at 0 with value  equal to the sum of the values of $\eta$.  Apply Lemma \ref{G_0_asymp_lemma} to the difference $g_{\bT_m}*(\eta-h)$.
\end{proof}

\begin{lemma*}[Lemma \ref{Green_fn_asymptotic}]
Let $d \geq 2$ and let $\ua \in \bN^d$.  If $|\ua| + \frac{d}{2} > 2$ then  for each fixed $n, v \in \sT$, 
 \begin{equation}
  D^{\ua} g_{v, \bT_m}(n) \to D^{\ua} g_{v}(n)
 \end{equation}
as $m \to \infty$.
\end{lemma*}

\begin{proof}

 As a function on $\Lambda$, the Fourier transform of  $D^{\ua} g_{0}$,
 \begin{equation}
  \widehat{D^{\ua} g_{0} }(x) = \frac{\prod_{j=1}^d (e(x_j)-1)^{\ua_j}}{(\deg(0))(1-\hat{\varrho}(x))}.
 \end{equation}
On $\Lambda/m\Lambda$, the discrete Fourier transform is obtained by taking points which are $\frac{1}{m}$ times a vector in $\zed^d$.  By inverse Fourier transform
\begin{equation}
 D^{\ua} g_{0, \bT_m}(n) = \frac{1}{m^d} \sum_{0 \neq x \in (\zed/m\zed)^d} \widehat{D^{\ua} g_{0} }\left(\frac{x}{m}\right) e\left(\frac{n \cdot x}{m} \right).
\end{equation}
Note that the summand can be unbounded near 0, but is integrable, and the sum avoids 0.  Letting $m \to \infty$ obtains the integral
\begin{equation}
 D^{\ua} g_{0}(n) = \int_{(\bR/\zed)^d} \widehat{D^{\ua} g_{0} }(x) e\left(n \cdot x \right)dx.
\end{equation}
When $n \not \in \Lambda$, use that $D^{\ua} g_0(n)$ is a mixture of nearby lattice values, and that the mixture decays exponentially.  Since $D^{\ua} g(n')$ also decays as $\|n'\| \to \infty$, the claim follows.

By translation invariance, the argument for $v = 0$ handles also the case of $v \in \Lambda$.  When $v \not \in \Lambda$, write
\begin{align*}
  g_{v, \bT_m}(x) = -c_v + \frac{1}{\deg x} \E\left[ \sum_{j=0}^{T_v-1} \one\left(Y_{v,j} = x\right)\right] + \E\left[g_{Y_{v, T_v}, \bT_m}(x)\right].
\end{align*}
Since $T_v$ has distribution which decays exponentially, as $m \to \infty$ there is a probability exponentially small in $m$ that $Y_{v, n}$, $1 \leq n \leq T_v$ exits $\left(-\frac{m}{2}, \frac{m}{2}\right]^d$, so up to exponentially small error, the value of \begin{equation}\frac{1}{\deg x} \E\left[ \sum_{j=0}^{T_v-1} \one\left(Y_{v,j} = x\right)\right]\end{equation} is the same whether interpretted on $\sT$ or on $\bT_m$.  Also, $c_v \to 0$ as $m \to \infty$ so this may be discarded as an error term.  The probability that $Y_{v, T_v}$ leaves a fixed ball about 0 tends to 0 as the radius tends to infinity, and by the derivative condition, the Green's function is bounded.  Hence the convergence of
\begin{equation}
 \E\left[g_{Y_{v, T_v}, \bT_m}(x)\right] \to \E\left[g_{Y_{v, T_v}}(x)\right]
\end{equation}
holds from the convergence of the Green's function on fixed balls about 0.
\end{proof}

\bibliographystyle{plain}

\begin{thebibliography}{1}

\bibitem{ADMR10}
Azimi-Tafreshi, N., H. Dashti-Naserabadi, S. Moghimi-Araghi, and Philippe Ruelle. 
\newblock ``The Abelian sandpile model on the honeycomb lattice." 
\newblock \emph{Journal of Statistical Mechanics: Theory and Experiment} 2010, no. 02 (2010): P02004.

\bibitem{BTW88}
Bak, P., C. Tang, and K. Wiesenfeld. 
\newblock ``Self-organized criticality." 
\newblock \emph{Physical Review A} 38.1 (1988): 364.

\bibitem{BIP93}
Brankov, J. G., E. V. Ivashkevich, and V. B. Priezzhev. 
\newblock ``Boundary effects in a two-dimensional Abelian sandpile." 
\newblock \emph{Journal de Physique} I 3.8 (1993): 1729-1740.

\bibitem{DFF03}
Dartois, Arnaud, Francesca Fiorenzi, and Paolo Francini.
\newblock ``Sandpile group on the graph $D_n$ of the dihedral group." 
\newblock \emph{European Journal of Combinatorics} 24.7 (2003): 815-824.


\bibitem{D89}
Dhar, Deepak. 
\newblock ``Self-organized critical state of sandpile automaton models." 
\newblock \emph{Physical Review Letters} 64.14 (1990): 1613.


\bibitem{DM92}
Deepak Dhar, and Majumdar, Satya N. 
\newblock ``Equivalence between the Abelian sandpile model and the $q\to 0$ limit of the Potts model." 
\newblock \emph{Physica A: Statistical Mechanics and its Applications} 185.1-4 (1992): 129-145.


\bibitem{DS10}
Dhar, Deepak, and Sadhu, Tridib. 
\newblock ``Pattern formation in growing sandpiles with multiple sources or sinks.'' \emph{J. Stat. Phys.} 138 (2010), no. 4-5, 815--837. 

\bibitem{DS87}
Diaconis, Persi, and Mehrdad Shahshahani. 
\newblock ``Time to reach stationarity in the Bernoulli--Laplace diffusion model." 
\newblock \emph{SIAM Journal on Mathematical Analysis} 18.1 (1987): 208-218.

\bibitem{D88}
Diaconis, Persi. 
\newblock ``Group representations in probability and statistics.'' \emph{Lecture Notes-Monograph Series} 11 (1988): i-192.

\bibitem{FLW10}
Fey, Anne, Lionel Levine, and David B. Wilson.
\newblock ``Driving sandpiles to criticality and beyond." 
\newblock \emph{Physical review letters} 104.14 (2010): 145703.

\bibitem{G16}
Gamlin, Samuel. 
\newblock \emph{Boundary conditions in Abelian sandpiles.}
\newblock Diss. University of Bath, 2016.

\bibitem{H17}
Hough, Robert. 
\newblock ``Mixing and cut-off in cycle walks." 
\newblock \emph{Electronic Journal of Probability} 22 (2017).

\bibitem{HJL17}
Hough, R.D., Jerison, D.C. and Levine, L. 
\newblock ``Sandpiles on the square lattice.''
\newblock \emph{Commun. Math. Phys.}  367: 33 (2019). 

\bibitem{HS19}
Hough, Robert and Hyojeong Son.
\newblock ``The spectral gap of abelian sandpiles on tiling graphs.''
\newblock Preprint, (2019).

\bibitem{I94}
Ivashkevich, E. V. 
\newblock ``Boundary height correlations in a two-dimensional Abelian sandpile." 
\newblock \emph{Journal of Physics A: Mathematical and General} 27.11 (1994): 3643--3653.


\bibitem{JPR06}
Jeng, Monwhea, Geoffroy Piroux, and Philippe Ruelle. 
\newblock ``Height variables in the Abelian sandpile model: scaling fields and correlations."
\newblock \emph{Journal of Statistical Mechanics: Theory and Experiment} 2006.10 (2006).

\bibitem{JLP15}
Jerison, Daniel C., Lionel Levine, and John Pike. 
\newblock ``Mixing time and eigenvalues of the abelian sandpile Markov chain." 
\newblock arXiv:1511.00666 (2015).



\bibitem{KW16}
Kassel, Adrien, and David B. Wilson. 
\newblock``The looping rate and sandpile density of planar graphs." 
\newblock \emph{The American Mathematical Monthly} 123.1 (2016): 19-39.

\bibitem{LL10}
Lawler, Gregory F., and Vlada Limic. 
\newblock \emph{Random walk: a modern introduction.}
\newblock Vol. 123. Cambridge University Press, 2010.

\bibitem{LP10}
Levine, Lionel and Propp, James.
\newblock ``What is … a sandpile?'' 
\newblock \emph{Notices Amer. Math. Soc.} 57 (2010), no. 8, 976–979. 

\bibitem{LH02}
Lin, Chai-Yu, and Chin-Kun Hu. 
\newblock ``Renormalization-group approach to an Abelian sandpile model on planar lattices." 
\newblock \emph{Physical Review E} 66.2 (2002): 021307.


\bibitem{O48}
Otter, Richard. 
\newblock ``The number of trees." \emph{Annals of Mathematics} (1948): 583-599.

\bibitem{NOT17}
Nassouri, Estelle, Stanislas Ouaro, and Urbain Traore. 
\newblock ``Growing sandpile problem with Dirichlet and Fourier boundary conditions." 
\newblock \emph{Electronic Journal of Differential Equations} 2017.300 (2017): 1-19.

\bibitem{PS95}
Papoyan, Vl V., and R. R. Shcherbakov. 
\newblock ``Abelian sandpile model on the Husimi lattice of square plaquettes." 
\newblock \emph{Journal of Physics A: Mathematical and General} 28.21 (1995): 6099.

\bibitem{PS13}
Pegden, Wesley, and Charles K. Smart. 
\newblock ``Convergence of the Abelian sandpile."
\newblock \emph{Duke Mathematical Journal} 162.4 (2013): 627-642.

\bibitem{PR05}
Piroux, Geoffroy, and Philippe Ruelle. 
\newblock ``Boundary height fields in the Abelian sandpile model."
\newblock \emph{Journal of Physics A: Mathematical and General} 38.7 (2005): 1451.

\bibitem{P94}
Priezzhev, Vyatcheslav B. 
\newblock ``Structure of two-dimensional sandpile. I. Height probabilities." 
\newblock \emph{Journal of statistical physics} 74.5-6 (1994): 955-979.

\bibitem{SV09}
Schmidt, Klaus, and Evgeny Verbitskiy. 
\newblock ``Abelian sandpiles and the harmonic model."
\newblock \emph{Communications in Mathematical Physics} 292.3 (2009): 721.


\bibitem{S15}
Sokolov, Andrey, et al. 
\newblock ``Memory on multiple time-scales in an abelian sandpile.'' 
\newblock \emph{Physica A: Statistical Mechanics and its Applications} 428 (2015): 295-301.

\bibitem{S13}
Spitzer, Frank. 
\newblock \emph{Principles of random walk.} 
\newblock Vol. 34. Springer Science \& Business Media, 2013.

\bibitem{S04}
Steele, J. Michael. 
\newblock \emph{The Cauchy-Schwarz Master Class.} Cambridge University Press, 2004.

\bibitem{TV06}
Tao, Terence, and Van H. Vu. 
\newblock \emph{Additive combinatorics.} Vol. 105. Cambridge University Press, 2006.

\end{thebibliography}

\end{document}